\documentclass[12pt,preprint]{imsart}
\RequirePackage[numbers]{natbib}
\RequirePackage[colorlinks,citecolor=blue,urlcolor=blue]{hyperref}
\usepackage{amsmath,amsmath,amssymb,amsfonts,qtree}
\usepackage{amsthm}
\usepackage[american]{babel}
\usepackage{graphicx}
\usepackage{flafter}
\usepackage[section]{placeins}
\usepackage{float}
\usepackage{makecell}
\usepackage{comment}
\usepackage[mathscr]{eucal}
\usepackage{bbm}
\usepackage{todonotes}
\presetkeys{todonotes}{color=red!30}{linecolor=black!50}
\startlocaldefs

\def\P{{\mathbb P}}

\def\Bbb E{\mathbb{E}}
\def\Bbb R{\mathbb{R}}
\newtheorem{definition}{Definition}
\newtheorem{example}{Example}
\newtheorem{corollary}{Corollary}[section]
\usepackage{mathtools}
\makeatletter \@addtoreset{equation}{section}

\makeatother

\newtheorem{lemma}{Lemma}[section]
\newtheorem{theorem}{Theorem}[section]
\newtheorem{proposition}{Proposition}[section]

\newtheorem{remark}{Remark}[section]

\DeclareMathOperator*{\Motimes}{\text{\raisebox{0.25ex}{\scalebox{0.8}{$\bigotimes$}}}}

\font\tencmmib=cmmib10 \skewchar\tencmmib '60
\newfam\cmmibfam
\textfont\cmmibfam=\tencmmib

\font\tenmsb=msbm10 
\def\Bbb#1{\hbox{\tenmsb#1}}

\def\lessim{\ \lower4pt\hbox{$
\buildrel{\displaystyle <}\over\sim$}\ }
\def\gessim{\ \lower4pt\hbox{$\buildrel{\displaystyle >}
\over\sim$}\ }

\def\go0{\to 0}

\def\leftitem#1{\item{\hbox to\parindent{\enspace#1\hfill}}}

\def\QED{{$\hfill \bbox$}}

\def\sg{\sigma}

\def\sg2{\sigma^2}

\def\__{_{\infty}}
\def \thetah{\hat{\theta}}
\def \thetat{\tilde{\theta}}

\def \QED{\hfill$\square$}

\numberwithin{equation}{section} 

\newcommand{\1}{{\rm 1}\kern-0.24em{\rm I}}

\endlocaldefs

\begin{document}

\begin{frontmatter}
\title{Estimation of Smooth Functionals in Normal Models: Bias Reduction and Asymptotic Efficiency}
\runtitle{}

\begin{aug}
\author{\fnms{Vladimir} \snm{Koltchinskii}\thanksref{t1}\ead[label=e1]{vlad@math.gatech.edu}} and 
\author{\fnms{Mayya} \snm{Zhilova}\thanksref{m1}\ead[label=e2]{mzhilova@math.gatech.edu}}
\thankstext{t1}{Supported in part by NSF Grant DMS-1810958} 
%NSF Grants DMS-1509739 and CCF-1523768}
\thankstext{m1}{Supported in part by NSF Grant DMS-1712990}
\runauthor{V. Koltchinskii and M. Zhilova}

\affiliation{Georgia Institute of Technology\thanksmark{m1}}

\address{School of Mathematics\\
Georgia Institute of Technology\\
Atlanta, GA 30332-0160\\
\printead{e1}\\
% \phantom{E-mail:\ }
\printead*{e2}
}
\end{aug}
\vspace{0.2cm}
{\small \today}
\vspace{0.2cm}

\begin{abstract}
Let $X_1,\dots, X_n$ be i.i.d. random variables sampled from a normal distribution 
$N(\mu,\Sigma)$ in ${\mathbb R}^d$ with unknown parameter $\theta=(\mu,\Sigma)\in \Theta:={\mathbb R}^d\times {\mathcal C}_+^d,$ where ${\mathcal C}_+^d$ is the cone of positively definite covariance operators in ${\mathbb R}^d.$ 
Given a smooth functional $f:\Theta \mapsto {\mathbb R}^1,$ the goal is to estimate $f(\theta)$ based on 
$X_1,\dots, X_n.$ Let 
$$
\Theta(a;d):={\mathbb R}^d\times \Bigl\{\Sigma\in {\mathcal C}_+^d: \sigma(\Sigma)\subset [1/a, a]\Bigr\}, a\geq 1,
$$ 
where $\sigma(\Sigma)$ is the spectrum of covariance $\Sigma.$  
Let $\thetah:=(\hat \mu, \hat \Sigma),$ where $\hat \mu$ is the sample mean 
and $\hat \Sigma$ is the sample covariance, based on the observations $X_1,\dots, X_n.$
For an arbitrary functional $f\in C^s(\Theta),$ $s=k+1+\rho, k\geq 0, \rho\in (0,1],$
we define a functional $f_k:\Theta \mapsto {\mathbb R}$ such that 
\begin{align*}
&
\sup_{\theta\in \Theta(a;d)}\|f_k(\thetah)-f(\theta)\|_{L_2({\mathbb P}_{\theta})}
%\\
%&
\lesssim_{s, \beta}
\|f\|_{C^{s}(\Theta)}
\biggr[\biggl(\frac{a}{\sqrt{n}}
\bigvee a^{\beta s}\biggl(\sqrt{\frac{d}{n}}\biggr)^{s}
\biggr)\wedge 1\biggr],
\end{align*}
where $\beta =1$ for $k=0$ and $\beta>s-1$ is arbitrary for $k\geq 1.$ 
This error rate is minimax optimal and similar bounds hold for more general loss functions. 
If $d=d_n\leq n^{\alpha}$ for some $\alpha\in (0,1)$ and $s\geq \frac{1}{1-\alpha},$
the rate becomes $O(n^{-1/2}).$ Moreover, for $s>\frac{1}{1-\alpha},$ the estimators $f_k(\thetah)$ is shown 
to be asymptotically efficient. The crucial part of the construction of estimator $f_k(\thetah)$ is a bias 
reduction method studied in the paper for more general statistical models than normal.
\end{abstract}



\end{frontmatter}

\section{Introduction.}
\label{intro}

The main goal of this paper is to develop estimators of general smooth functionals of parameters 
of high-dimensional normal models with minimax optimal risk with respect to convex loss functions
(including quadratic loss). In particular, we are interested in developing efficient estimators with parametric 
$\sqrt{n}$ error rate for sufficiently smooth functionals (under optimal assumptions on their smoothness).  
To achieve this goal, we further develop a general approach to bias reduction in functional estimation problems 
initially studied for particular models in \cite{Jiao, Koltchinskii_2017, Koltchinskii_2018, Koltchinskii_Zhilova}. 
Although, in principle, this approach and the results on bias reduction obtained below could be applicable to more general classes of statistical models than normal model, the development 
of concentration bounds in this more general context poses additional challenging problems and is beyond 
the scope of this paper.

\subsection{Main results.}
\label{MainResults}

Let $X_1,\dots, X_n$ be i.i.d. random variables in ${\mathbb R}^d$ sampled from a normal distribution 
$N(\mu, \Sigma)$ with unknown mean $\mu$ and covariance $\Sigma.$ It will be assumed that the 
space ${\mathbb R}^d$ is equipped with the standard Euclidean inner product and the corresponding 
norm. Let ${\mathcal S}^d$ be the space of all $d\times d$ symmetric operators equipped 
with the operator norm and let ${\mathcal C}_+^d\subset {\mathcal S}^d$ be the cone of all {\it positively definite} covariance operators. The parameter of our model is $\theta=(\mu,\Sigma)$ and the parameter space is $\Theta:= {\mathbb R}^d\times {\mathcal C}_+^d\subset {\mathbb R}^d\times {\mathcal S}^d.$ 
The space ${\mathbb R}^d\times {\mathcal S}^d$ will be equipped with the norm
\footnote{With a little abuse of notation, all the above norms are denoted $\|\cdot\|.$} 
$$
\|(w,W)\|:= \|w\|+ \|W\|, w\in {\mathbb R}^d, W\in {\mathcal S}^d.
$$
Given a smooth functional $f: \Theta \mapsto {\mathbb R},$
the goal is to estimate $f(\theta)$ based on observations $X_1,\dots, X_n.$

Let $\ell: {\mathbb R}\mapsto {{\mathbb R}}_+\cup \{+\infty\}$ be a nonnegative convex loss function such that 
$\ell(u)=\ell (-u), u\in {\mathbb R},$ $\ell(0)=0$ and that $\ell$ is  
nondecreasing on ${\mathbb R}_+.$ Let ${\mathcal L}$ denote the set 
of all such loss functions. In what follows, it will be convenient to use the Orlicz norm $\|\cdot\|_{L_{\ell}({\mathbb P})},$ associated with loss $\ell$ and  defined as follows:
$$
\|\xi\|_{L_{\ell}({\mathbb P})}:= \inf\Bigl\{c>0: {\mathbb E}\ell\Bigl(\frac{|\xi|}{c}\Bigr)\leq 1\Bigr\}.
$$
In particular, if $\ell(u):=|u|^p, u\in {\mathbb R}, p\geq 1,$ then $\|\xi\|_{L_{\ell}({\mathbb P})}=\|\xi\|_{L_p({\mathbb P})}.$
The ``subexponential loss" $\ell(u):= e^{|u|}-1, u\in {\mathbb R}$ is usually denoted $\psi_1$ and the
subgaussian loss $\ell(u):=e^{u^2}-1, u\in {\mathbb R}$ is usually denoted $\psi_2,$ leading 
to $\psi_1$ and $\psi_2$-norms, respectively. In the cases when there is no ambiguity about 
the probability measure(s) involved, we will write   
$\|\cdot\|_{L_{\ell}({\mathbb P})}=\|\cdot\|_{\ell}.$

Let $\thetah=(\hat \mu, \hat \Sigma),$
where
$$
\hat \mu := \bar X = \frac{X_1+\dots + X_n}{n}, \ \ \hat \Sigma:= \frac{1}{n-1}\sum_{j=1}^n (X_j-\bar X)\otimes (X_j-\bar X)
$$
are the sample mean and the sample covariance, respectively. 
Given $\Sigma\in {\mathcal C}_+^d,$ denote by $\sigma(\Sigma)$ the spectrum of matrix $\Sigma.$ 
For $a\geq 1,$ denote
$$
\Theta (a;d):={\mathbb R}^d \times \Bigl\{\Sigma \in {\mathcal C}_+^d: \sigma(\Sigma)\subset [1/a, a]\Bigr\}.
$$

The following result will be proved at the end of the paper (see sections \ref{sec:approx_homot} and \ref{sec:conc}).\footnote{The precise definition of space $C^s(\Theta)$ and $C^s$-norm is given in Section \ref{preliminaries}.}

\begin{theorem}
\label{main_theorem_1}
Suppose $f\in C^s(\Theta)$ for some $s=k+1+\rho, k\geq 0, \rho\in (0,1].$
Let $\ell\in {\mathcal L}$ be a loss function such that $\ell(u)\leq e^{bu}, u\geq 0$
for some constant $b>0.$ Then, there exists a functional $f_k: \Theta\mapsto {\mathbb R}$
(with $f_0=f$) such that
\begin{align}
\label{upper_bound_functional}
&
\sup_{\theta\in \Theta(a;d)}\|f_k(\thetah)-f(\theta)\|_{L_{\ell}({\mathbb P}_{\theta})}
\lesssim_{s, \ell, \beta}
\|f\|_{C^{s}(\Theta)}
\biggr[\biggl(\frac{a}{\sqrt{n}}
\bigvee a^{\beta s}\biggl(\sqrt{\frac{d}{n}}\biggr)^{s}
\biggr)\wedge 1\biggr],
\end{align}
where $\beta=1$ for $k=0$ and $\beta >s-1$ is an arbitrary number for $k\geq 1.$
\end{theorem}

Suppose that $d=d_n\leq n^{\alpha}$ for some $\alpha\in (0,1).$ If $s\geq \frac{1}{1-\alpha},$
then 
$$a^{\beta s}\biggl(\sqrt{\frac{d}{n}}\biggr)^{s}\leq \frac{a^{\beta s}}{\sqrt{n}}$$ 
and bound \eqref{upper_bound_functional} of Theorem \ref{main_theorem_1}
implies that 
\begin{align*}
\sup_{\|f\|_{C^{s}(\Theta)}\leq 1}\sup_{\theta\in \Theta(a;d_n)}\|f_k(\thetah)-f(\theta)\|_{L_{\ell}({\mathbb P}_{\theta})}
=O(n^{-1/2}).
\end{align*}

For quadratic $\ell (u)=u^2, u\in {\mathbb R},$ 
the following result (that itself is a simple corollary of Theorem 2.2 in \cite{Koltchinskii_Zhilova})
shows some form of minimax optimality of estimator $f_k(\thetah).$ 

\begin{theorem}
The following minimax bound holds:
\begin{align}
\label{sup_f_inf_T}
&
\sup_{\|f\|_{C^s(\Theta)}\leq 1}\inf_{T}\sup_{\theta\in \Theta(a;d)}\|T(X_1,\dots, X_n)-f(\theta)\|_{L_2({\mathbb P}_{\theta})}
\gtrsim_{s}
\biggr[\biggl(\frac{1}{\sqrt{n}}
\bigvee \biggl(\sqrt{\frac{d}{n}}\biggr)^{s}
\biggr)\wedge 1\biggr],
\end{align}
where the infimum is taken over all estimators $T(X_1,\dots, X_n).$
\end{theorem}

This bound shows that for $d=d_n\geq n^{\alpha}, \alpha\in (0,1]$ and $s<\frac{1}{1-\alpha},$
there are functionals $f$ with $\|f\|_{C^{s}(\Theta)}\leq 1$ such that $f(\theta)$ could not be estimated with a rate better than $n^{-s(1-\alpha)/2},$ which is slower than $n^{-1/2}.$ In other words, the threshold 
$\frac{1}{1-\alpha}$ on smoothness $s$ needed for the existence of $\sqrt{n}$-consistent estimators of $f(\theta)$ is sharp. Moreover, for $d\asymp n$ (or $\alpha=1$) even consistent estimators do not exist for some functionals $f$ of an arbitrary degree of smoothness $s.$ 

Note also that the lower bound  
\eqref{sup_f_inf_T} is attained for the functionals depending only on the mean $\mu$ and 
for the smallest parameter set $\Theta(1;d)$ for $d=1$ (see Theorem 2.2 in \cite{Koltchinskii_Zhilova}).
We do not know at the moment what is the precise dependence on $a$ in such bounds 
as \eqref{upper_bound_functional}, \eqref{sup_f_inf_T}. 

It turns out that for $s=k+1+\rho\leq 2,$ we have $k=0$ and $f_k = f_0=f,$ 
so $f_k(\thetah)= f(\thetah)$ is just the usual plug-in estimator.  
For $s>2,$ we have $k\geq 1.$ In this case, the plug-in estimator would be suboptimal due to its large bias 
and  a non-trivial bias reduction (for instance, the one leading to our estimator $f_k(\thetah)$) becomes crucial.

In addition to Theorem \ref{main_theorem_1}, we establish asymptotic efficiency of estimator 
$f_k(\thetah)$ provided that $d=d_n\leq n^{\alpha}$ for some $\alpha\in (0,1)$ and $s>\frac{1}{1-\alpha}.$
Let $f:\Theta\mapsto {\mathbb R}$ be a continuously differentiable functional with 
$f_{\mu}^{\prime}(\mu, \Sigma)\in {\mathbb R}^d, f_{\Sigma}^{\prime}(\mu,\Sigma)\in {\mathcal S}^d$ being its 
partial derivatives with respect to $\mu$ and $\Sigma.$ Denote
$$
\sigma_f^2 (\theta) := \|\Sigma^{1/2} f_{\mu}^{\prime}(\mu, \Sigma)\|^2 + 2\|\Sigma^{1/2}f_{\Sigma}^{\prime}(\mu,\Sigma)\Sigma^{1/2}\|_2^2,\ \theta =(\mu,\Sigma)\in \Theta,
$$
$\|\cdot\|_2$ being the Hilbert--Schmidt norm.
For simplicity, the result will be stated and proved only in the case of quadratic loss.

\begin{theorem}
\label{main_theorem_2}
Suppose $d=d_n\leq n^{\alpha}$ for some $\alpha\in (0,1).$ Then, for all $s=k+1+\rho> \frac{1}{1-\alpha}, k\geq 0, \rho\in (0,1],$
\begin{align}
\label{efficient_funct}
&
\sup_{\|f\|_{C^s(\Theta)}\leq 1}\sup_{\theta\in \Theta(a;d_n)}\Bigl| n {\mathbb E}_{\theta}\Bigl(f_k(\thetah)-f(\theta)\Bigr)^2 
- \sigma^2_f(\theta)\Bigr| \to 0
\end{align}
as $n\to\infty.$ Moreover, for all $\sigma_0>0,$
\begin{align}
\label{normal_approx_funct}
&
\sup_{\|f\|_{C^s(\Theta)}\leq 1}
\sup_{\theta\in \Theta(a;d_n), \sigma_f(\theta)\geq \sigma_0}
\sup_{x\in {\mathbb R}}\Bigl|{\mathbb P}_{\theta}\Bigl\{
\frac{\sqrt{n}(f_k(\thetah)-f(\theta))}{\sigma_f(\theta)}\leq x
\Bigr\} -{\mathbb P}\{Z\leq x\}\Bigr|
\to 0
\end{align}
as $n\to\infty,$ where $Z\sim N(0,1).$
\end{theorem}

Finally, the local minimax lower bound of the next theorem validates the claim of asymptotic efficiency 
of estimator $f_k(\thetah).$ 

\begin{theorem}
Let $f:\Theta\mapsto {\mathbb R}$ be a differentiable functional 
with derivative $f^{\prime}$ and let\footnote{In the expression $\|f^{\prime}(\theta)-f^{\prime}(\theta_0)\|,$ $\|\cdot\|$ denotes the norm of a linear functional on the space ${\mathbb R}^d\times {\mathcal S}^d.$} 
$$
\omega_{f^{\prime}}(\theta_0;\delta):= \sup_{\|\theta-\theta_0\|\leq \delta}\|f^{\prime}(\theta)-f^{\prime}(\theta_0)\|
$$  
be a local continuity modulus of $f'$ at point $\theta_0\in \Theta.$
For all $\beta>2,$ there exists a constant $D_{\beta}>0$ such that 
for all $\delta>0$ and all $\theta_0\in \Theta (a;d)$ satisfying the condition 
$\{\theta: \|\theta-\theta_0\|\leq \delta\}
\subset \Theta (a;d),$ 
the following bound holds:
\begin{align} 
\inf_{T_n}\sup_{\|\theta-\theta_0\|\leq \delta}\frac{n {\mathbb E}_{\theta}\Bigl(T_n(X_1,\dots, X_n)-f(\theta)\Bigr)^2}{\sigma_f^2(\theta)}
\geq 
1- D_{\beta} \biggl[\frac{a\ \omega_{f'}(\theta_0;\delta)}{\sigma_f(\theta_0)}+ a^{\beta}\delta + \frac{a^2}{\delta^2 n}\biggr],
\end{align}
where the infimum is taken over all estimators $T_n=T_n(X_1,\dots, X_n).$
\end{theorem}

The proof of this minimax bound is based on Van Trees inequality and it will not be provided in this paper (see \cite{Koltchinskii_2017, Koltchinskii_Nickl, Koltchinskii_Zhilova} for the proofs of similar statements). Essentially, the bound shows 
(in the case when dimension $d$ is fixed) 
that, if $\sigma_f(\theta_0)$ is bounded away from zero and $f^{\prime}$ is continuous at $\theta_0,$
then the following version of H\`ajek-LeCam asymptotic minimax lower bound holds:
\begin{align*} 
\lim_{c\to\infty}\liminf_{n\to \infty}\inf_{T_n}\sup_{\|\theta-\theta_0\|\leq \frac{c}{\sqrt{n}}}\frac{n {\mathbb E}_{\theta}\Bigl(T_n(X_1,\dots, X_n)-f(\theta)\Bigr)^2}{\sigma_f^2(\theta)}
\geq 1
\end{align*}
(for this, it is enough to take $\delta=\frac{c}{\sqrt{n}}$). Similar conclusion holds in the 
case when $d=d_n\to \infty$ as $n\to\infty$ (although, in this case, an asymptotic formulation 
of the result would involve a sequence of functionals on $\Theta=\Theta_n$ and 
a sequence of points $\theta_0=\theta_0^{(n)}\in \Theta_n$). 

\begin{remark}
\normalfont
In \cite{Koltchinskii_2017}, the results similar to Theorem \ref{main_theorem_2} on efficient estimation of  
smooth functionals of unknown covariance of Gaussian model with zero mean 
were proved for functionals of the form $\langle f(\Sigma), B\rangle,$ where $f$
is a smooth function in the real line (in fact, a function from Besov space $B^s_{\infty,1}(\mathbb R)$ for 
$s>\frac{1}{1-\alpha}$) and $B$ is an operator with nuclear norm $\|B\|_1$ bounded 
by $1.$ The method of proof developed in \cite{Koltchinskii_2017} relied on special 
properties of Wishart operators in the spaces of orthogonally invariant functions on the cone 
of covariance operators. This method allows one to prove asymptotic efficiency for estimators 
of somewhat more general functionals than $\langle f(\Sigma), B\rangle$ (defined in terms of a certain differential operator acting 
on orthogonally invariant functions of covariance), but it could not be applied to arbitrary 
smooth functionals studied in this paper. Even such functionals as $\sum_{i} \langle f_i(\Sigma), B_i\rangle,$ 
where $\|f_i\|_{B^s_{\infty,1}(\mathbb R)}\leq 1$ and $\sum_{i}\|B_i\|_1\leq 1,$ seem to be 
beyond the scope of methods of \cite{Koltchinskii_2017}, but could be handled using Theorem \ref{main_theorem_2}.
\end{remark}

\begin{remark}
\normalfont
A number of special functionals in various problems of high-dimensional statistics, in particular, 
in problems related to estimation of spectral characteristics of unknown covariance, could be represented 
in terms of smooth functionals. In particular, if $C\subset \sigma(\Sigma)$ 
is a cluster of the spectrum of $\Sigma$ ``well separated" from the rest of the spectrum,
the spectral projection $P_C$ on the direct sum of eigenspaces corresponding to the eigenvalues of $\Sigma$ in the cluster $C$ could be easily represented as $f(\Sigma)$ for a smooth function $
f$ in the real line (that is equal to $1$
on $C$ and vanishes outside of a neighborhood of $C$). One can then apply the methods of estimation 
of smooth functionals developed in this paper to functionals of the form $\langle f(\Sigma), B\rangle$
or $\|f(\Sigma)-A\|_2^2$ (for given operators $A,B$) that often occur in estimation and testing problems
in principal component analysis. Of course, in many cases, one can also develop more specialized methods 
for these special problems (see \cite{Koltchinskii_Lounici_bilinear, Koltchinskii_Lounici_AOS, Koltchinskii_Nickl}). 
\end{remark}

\begin{comment}

Finally, the following asymptotic minimax lower bound validates the claim of asymptotic efficiency 
of estimator $f_k(\thetah).$ For $\sigma_0>0,$ denote 
$$
\Theta_{f,\sigma_0}(a;d) := \{\theta \in \Theta (a;d): \sigma_f(\theta)\geq \sigma_0\}
$$
and, for $s>1,$ 
$$
{\mathcal F}_{s,\sigma_0}(a;d):= \{f: \|f\|_{C^s(\Theta)}\leq 1, \Theta_{f,\sigma_0}(a;d) \neq \emptyset\}.
$$

\begin{theorem}
For all $s>1,$ $\sigma_0>0$ and arbitrary sequence $\{d_n\}$
\begin{align}
\liminf_{n\to\infty}\inf_{T_n}\inf_{f\in {\mathcal F}_{s,\sigma_0}(a;d_n)}\sup_{\theta\in \Theta_{f,\sigma_0}(a;d_n)}\frac{n {\mathbb E}_{\theta}\Bigl(T_n(X_1,\dots, X_n)-f(\theta)\Bigr)^2}{\sigma_f^2(\theta)}
\geq 1,
\end{align}
where the infimum is taken over all estimators $T_n.$
\end{theorem}

\end{comment}

The construction of estimator $f_k(\thetah)$ of theorems \ref{main_theorem_1}, \ref{main_theorem_2} 
is based on a method of bias reduction considered earlier 
in  \cite{Jiao}, \cite{Koltchinskii_2017}. This method will be further developed in the current paper in a general framework described below.  
Let $(\Theta, {\mathcal B}_{\Theta})$ be a measurable space 
and let $\thetah = \thetah (X)$ be an estimator of parameter $\theta$ based on an observation $X\sim P_{\theta}, \theta \in \Theta$ in some measurable space $(S,{\mathcal A}).$
Define 
$$
P(\theta;A) := {\mathbb P}_{\theta}\{\thetah \in A\},\ A\in {\mathcal A}.
$$
Assume that $P:\Theta\times {\mathcal B}_{\Theta}\mapsto [0,1]$ is a Markov kernel.  
Clearly, $P(\theta,\cdot)$ is the distribution of estimator $\thetah (X), X\sim P_{\theta}, \theta\in \Theta.$

Consider a measurable function $f:\Theta \mapsto {\mathbb R}$ and suppose we want to construct  
an estimator of $f(\theta)$ based on $\thetah.$ In order to find an estimator of the form $g(\hat \theta)$
with a small bias, we need to solve approximately the equation ${\mathbb E}_{\theta} g(\thetah)=f(\theta), \theta \in \Theta.$ Define the following integral operator with respect to Markov kernel $P:$
$$
({\mathcal T}g)(\theta):= {\mathbb E}_{\theta} g(\thetah)= \int_{\Theta}g(t)P(\theta;dt).
$$
It could be viewed, for instance, as an operator from the space $L_{\infty}(\Theta)$ of bounded 
measurable functions on $\Theta$ into itself. 
Let ${\mathcal B}:={\mathcal T}-{\mathcal I}.$ Informally, the solution of equation ${\mathcal T}g=({\mathcal I}+{\mathcal B})g=f$ could be represented as Neumann series:  $g=({\mathcal I}-{\mathcal B}+{\mathcal B}^2-\dots)f.$ Define 
$
{\mathcal E}_k := \sum_{j=0}^k (-1)^j {\mathcal B}^j 
$
and
$$
f_k(\theta):= ({\mathcal E}_k f)(\theta)= \sum_{j=0}^k (-1)^j ({\mathcal B}^j f)(\theta), \theta\in \Theta.
$$
Then, the bias of estimator $f_k(\thetah)$ of $f(\theta)$ is equal to 
\begin{align}
\label{bias_formula}
&
\nonumber
{\mathbb E}_{\theta} f_k(\thetah)-f(\theta) = ({\mathcal T}f_k)(\theta)-f(\theta)
\\
&
=(\mathcal{I}+{\mathcal B})\sum_{j=0}^k (-1)^j ({\mathcal B}^j f)(\theta)=
(-1)^k ({\mathcal B}^{k+1} f)(\theta), \theta \in \Theta.
\end{align}
If $\thetah$ is close to $\theta$ with a high probability, then operator ${\mathcal T}$ is close 
to identity and operator ${\mathcal B}$ is small. Thus, one could expect that, for a sufficiently 
large $k,$ the function $({\mathcal B}^{k+1} f)(\theta), \theta \in \Theta$ would be even of smaller order,
resulting in a bias reduction for estimator $f_k(\thetah).$ Also observe that 
$$
{\mathcal B} f(\theta)= {\mathbb E}_{\theta}f(\thetah) - f(\theta)
$$
is the bias of the plug-in estimator $f(\thetah)$ of $f(\theta).$ To reduce this bias, 
one can use a plug-in estimator ${\mathcal B}f(\thetah)$ of ${\mathcal B}f(\theta)$
and subtract it from $f(\thetah),$ yielding the estimator 
$$
f_1(\thetah) = f(\thetah)- ({\mathcal B}f)(\thetah),
$$
whose bias is equal to $-({\mathcal B}^2 f)(\theta).$
Iterating this procedure $k$ times yields the estimator $f_k(\thetah).$

Define a Markov chain 
$\{\thetah^{(k)}:k\geq 0\}$ with transition kernel $P$ and with $\thetah^{(0)}=\theta.$ 
For this chain, $\thetah^{(0)}\coloneqq \theta,\, \thetah^{(1)} \coloneqq \thetah$ and, conditionally
on $\thetah^{(0)}, \dots, \thetah^{(k)},$ $\thetah^{(k+1)}\sim P(\thetah^{(k)},\cdot), k\geq 0,$
which could be viewed as iterative applications of parametric bootstrap to the estimator $\hat \theta$
of parameter $\theta.$ In what follows, $\{\thetah^{(k)}:k\geq 0\}$ will be called \it a bootstrap chain. \rm
It easily follows from Chapman-Kolmogorov equation 
that 
\begin{align}
\label{T^k}
({\mathcal T}^{k} g)(\theta)={\mathbb E}_{\theta} g(\thetah^{(k)}), k\geq 0.
\end{align}
By Newton's binomial formula and representation \eqref{T^k},
\begin{align}
\label{Newton_binom}
&
\nonumber
({\mathcal B}^k f)(\theta)= (({\mathcal T}-{\mathcal I})^k f)(\theta)=
\sum_{j=0}^k (-1)^{k-j} {k\choose j} ({\mathcal T}^j f)(\theta)
\\
&
={\mathbb E}_{\theta}\sum_{j=0}^k (-1)^{k-j} {k\choose j} f(\thetah^{(j)}).
\end{align}
Note that representation \eqref{Newton_binom} provides a way to compute estimator $f_k(\thetah)$ (based on the computation of ${\mathcal B}^j f(\thetah), j\leq k$) using Monte 
Carlo simulation.
The sum $\sum_{j=0}^k (-1)^{k-j} {k\choose j} f(\thetah^{(j)})$ is the $k$-th order 
difference of function $f$ computed along the sample path of the 
Markov chain $\{\hat \theta^{(j)}: j\geq 0\}.$ In order to bound ${\mathcal B}^k f(\theta),$ we need to control the average sizes of such $k$-th order differences for a sufficiently smooth function $f.$ 

If $f$ is a function 
in the real line and $\Delta_h f(x):= f(x+h)-f(x)$ is its first order difference with step $h,$ then it is well known 
from classical analysis that, for the $k$-th order difference of $f,$ 
\begin{align}
\label{kthdiff}
\Delta_h^k f(x)= \sum_{j=0}^k (-1)^{k-j} {k\choose j} f(x+jh)= f^{(k)}(x)h^k + o(h^k)\ {\rm as}\ h\to 0
\end{align}
provided that $f$ is $k$ times continuously differentiable. 

Suppose now $\Theta$ is a subset of a Banach space $E$ and, for $\theta\in \Theta, X\sim P_{\theta}$
and for some small $\delta>0,$ 
$$
\sup_{\theta\in \Theta}P_{\theta}\{\|\thetah-\theta\|\geq \delta\} = \sup_{\theta\in \Theta}P(\theta; \Theta \setminus B(\theta;\delta))
$$ 
is a small number.\footnote{$B(\theta;\delta):=\{\theta':\|\theta'-\theta\|<\delta\}.$} In this case, the bootstrap chain $\{\thetah^{(k)}: k\geq 0\}$ moves with small steps (of the order at most $\delta$) with a high probability since 
$$
{\mathbb P}\{\|\thetah^{(k+1)}-\thetah^{(k)}\|\geq \delta|\thetah^{(k)}\}=P(\thetah^{(k)};\Theta\setminus B(\thetah^{(k)};\delta))
$$ 
is small. The basic question is whether for a $C^k$ function $f$ on $\Theta$
we would have an analogue of property \eqref{kthdiff} in the sense that  
\begin{align*}
{\mathcal B}^k f(\theta)={\mathbb E}_{\theta}\sum_{j=0}^k (-1)^{k-j} {k\choose j} f(\thetah^{(j)})=O(\delta^k).
\end{align*}
In view of \eqref{bias_formula}, it would mean that, for a $C^{k+1}$ function $f,$ the bias 
of the estimator $f_k(\thetah)$ of $f(\theta)$ would be of the order $O(\delta^{k+1}),$
which provides a way to reduce the bias of estimation of $f(\theta)$ by orders of magnitude comparing with the error rate $\delta$ of estimator $\thetah$ itself (provided that $f$ is sufficiently smooth). 
This approach to bias reduction has been already used, in particular, in \cite{Koltchinskii_2017, Koltchinskii_2018, Koltchinskii_Zhilova}. In \cite{Koltchinskii_2017}, it was used in a problem of estimation of 
functionals of unknown covariance $\Sigma$ of a normal model $X_1,\dots, X_n$ i.i.d. $\sim N(0;\Sigma)$
in ${\mathbb R}^d.$ In this case, the operator ${\mathcal T}$ is so called Wishart operator well studied 
in the theory of Wishart matrices. The analysis of the bias reduction problem in \cite{Koltchinskii_2017}
relied heavily on the properties of Wishart operators (especially, on the spaces of orthogonally invariant functions of matrices). This approach had some limitations. In particular, it was impossible to establish 
risk bounds and asymptotic efficiency for general smooth functionals of covariance,
but only for some classes of smooth functionals (such as functionals 
of the form $\langle f(\Sigma),B\rangle,$ where $f$ is a smooth function in the real line and $B$ is a matrix 
with bounded nuclear norm). On the other hand, in \cite{Koltchinskii_Zhilova}, the problem was studied 
in the case of general Gaussian shift models $X=\theta +\xi,$ where $\theta$ is an unknown mean vector in a separable 
Banach space $E$ and $\xi\sim N(0;\Sigma)$ is a mean zero Gaussian noise in $E$ with given covariance operator $\Sigma.$ Due to the simplicity of operator ${\mathcal T}$ in this case,\footnote{$(Tg)(\theta)={\mathbb E}g(\theta+\xi), \theta \in E$} it was possible to solve the problem for general 
smooth functionals.  

In the current paper, we develop certain analytic tools that allow us to provide bounds on H\"older 
norms of functions ${\mathcal B}^k f$ and $f_k={\mathcal E}_k f$ needed to control the bias of estimator $f_k(\thetah)$
of $f(\theta)$ and also to establish concentration properties of this estimator
(see sections \ref{random_homotopies}, \ref{sec:Holder_norms} and \ref{sec:bootstrap_chain}
and, especially, Theorem \ref{th_H_bound}). In particular, these results are applicable to general 
smooth functionals of parameters of normal model. 
Our approach is based on certain coupling techniques 
(random homotopies) that provide a way to represent or approximate (in distribution) the estimator $\thetah$
and the bootstrap chain $\{\thetah^{(k)}: k\geq 0\}$ in terms of certain smooth functions on the parameter space 
$\Theta.$ A random homotopy between $\theta$ and $\thetah$ is a smooth stochastic process 
$H(\theta;t), \theta \in \theta, t\in [0,1]$ with values in $\Theta$ such that $H(\theta;0)=\theta, \theta\in \Theta$ and 
$H(\theta;1)\overset{d}{=} \thetah(X), X\sim P_{\theta}, \theta \in \Theta.$ 
Superpositions of i.i.d. copies of stochastic process $H$ are then used to represent 
the bootstrap chain. Random homotopies are also used to provide representations 
of functions ${\mathcal B}^k f$ and to bound their H\"older norms. In particular, under 
proper smoothness assumptions on $H,$ we show that, for $s=k+1+\rho, \rho \in (0,1],$ ${\mathcal E}_k=\sum_{j=0}^k (-1)^j {\mathcal B}^j$ 
is a bounded linear operator from $C^{s}(\Theta)$ into $C^{1+\rho}(\Theta),$
which provides a way to study bias and concentration properties of estimator $f_k(\thetah).$ 

We apply the technique of random homotopies to study estimators of smooth functionals 
of parameters of high-dimensional normal models with nearly optimal error rates. 
This technique is rather general in nature and, in principle,
it could be used to develop efficient estimators of smooth functionals for more general classes of models, in  
particular, high-dimensional exponential families. This, however, would require the development of 
concentration bounds for estimators $f_k(\thetah)$ in more general settings than normal models, which is a challenging 
problem beyond the scope of this paper.

\begin{comment}
Returning to the case of general normal model 
$$
X_1,\dots, X_n\  {\rm i.i.d.}\ \sim N(\mu;\Sigma),\ \theta=(\mu,\Sigma)\in 
\Theta={\mathbb R}^d\times {\mathcal C}_+^d,
$$ 
we use estimator $\thetah=(\hat \mu, \hat \Sigma),$
where
$$
\hat \mu := \bar X = \frac{X_1+\dots + X_n}{n}, \ \ \hat \Sigma:= \frac{1}{n-1}\sum_{j=1}^n (X_j-\bar X)\otimes (X_j-\bar X)
$$
are the sample mean and the sample covariance, respectively. 
The estimator $\hat T_k$ of Theorem \ref{main_theorem_1} is then defined as 
$\hat T_k:= f_k(\thetah).$ 
\end{comment}

%$$
%\hat T_k(X_1,\dots, X_n):= 
%\begin{cases}
%f_k(\thetah)\  \  {\rm if}\  \ 
%\frac{a}{\sqrt{n}}\bigvee a^{\beta}\biggl(\sqrt{\frac{d}{n}}\biggr)^{s}\leq 1\\
%0\ \ \ \ \ \ \ \ \ \ \ \ {\rm otherwise.}
%\end{cases}
%$$

The following notations have been already used in this section and will be used throughout 
the paper. For $A,B\geq 0,$ $A\lesssim B$ means that there exists an absolute constant
$C>0$ such that $A\leq CB.$ The notation $A\gtrsim B$ is equivalent to $B\lesssim A$
and $A\asymp B$ is equivalent to $A\lesssim B$ and $A\gtrsim B.$ When constant 
$C$ depends on some parameter(s) $p,$ we might use subscript $p$ to emphasize 
this dependence, say, $A\lesssim_{p} B.$

\subsection{Related work.}
\label{Relatedwork}

The problem of estimation of functionals of parameters of high-dimensional and infinite-dimensional 
(nonparametric) statistical models has a long history. Early references include \cite{Levit_1, Levit_2, Ibragimov}
and an incomplete list of further important references includes \cite{Bickel_Ritov, Ibragimov_Khasm_Nemirov,
Girko, Girko-1, Donoho_1, Donoho_2, Donoho_Nussbaum, Nemirovski_1990, Birge, Laurent, Lepski, 
Nemirovski, Cai_Low_2005a, Cai_Low_2005b, Klemela, Robins, van der Vaart, C_C_Tsybakov, Robins_1, Han, Mukherjee}. The problem has been 
most often studied for special statistical models (such as Gaussian sequence model, Gaussian white noise
model or density estimation model) and special functionals (such as (unbounded) linear functionals, quadratic 
functionals, some norms in classical Banach spaces, some classes of integral functionals of unknown density). 
In \cite{Ibragimov_Khasm_Nemirov, Nemirovski_1990, Nemirovski}, a problem of estimation of general smooth 
functionals of an unknown function (signal) observed in a Gaussian white noise was studied. The complexity of the 
problem was characterized by the rate of decay of Kolmogorov diameters of the parameter space of the model.
The goal was to determine sharp thresholds on the smoothness of the functional such that 
its efficient estimation with parametric error rate becomes possible when the degree of smoothness 
is above the threshold. This approach is close to the one developed in our paper. 

The difficulty of functional estimation problem in high-dimensional and nonparametric settings is related to the fact that natural plug-in estimators $f(\thetah)$ (with $\thetah$ being the maximum likelihood estimator or its regularized versions with optimal error rates) fail to achieve optimal rates of estimation of $f(\theta)$ as soon as the dimension or other 
relevant complexity characteristics of the problem become sufficiently large. In such case, the development of better estimators is a challenge with bias reduction becoming an important part of the problem. There are very 
few general approaches to this problem. One of them is based on the notion of higher order influence functions
and utilizing the techniques of $U$-statistics to construct estimators with reduced bias and optimal convergence rates. This approach was initiated in \cite{Robins} (see also \cite{van der Vaart}, \cite{Robins_1},
\cite{Mukherjee} and further references therein).   
Unfortunately, as it was pointed out in these papers, higher order influence functions do not always exist and,
due to this difficulty, the method is usually not universally applicable to all smooth functionals.  

The approach to bias reduction studied in our paper was considered earlier in \cite{Jiao, Koltchinskii_2017, 
Koltchinskii_Zhilova}. In \cite{Jiao}, it was studied in the case of estimation of smooth functions of parameter 
$\theta$ of classical binomial model $X\sim B(n;\theta), \theta\in [0,1]$ (although, the authors were motivated 
by the problems of estimation of functionals of high-dimensional parameters). In this case, the operator ${\mathcal T}$ defined in Section \ref{MainResults} maps function $g$ into the Bernstein polynomial of degree $n,$
approximating this function. Bounds on functions ${\mathcal B}^k f$ needed to control the bias of estimator 
$f_k(\thetah)$ (with $\thetah=X/n$ being the frequency) were proved in \cite{Jiao} based on some results 
in classical approximation theory. In \cite{Koltchinskii_2017}, the problem of estimation of 
a special class of smooth functionals of high-dimensional covariance of a normal model was 
studied, the operator ${\mathcal T}$ was so called Wishart operator and the analysis of the bias  
reduction problem was largely based on special properties of such operators. 
Some other approaches to efficient estimation of special functionals of covariance 
in normal models (such as linear forms of principal components) were developed 
in recent papers \cite{Koltchinskii_Lounici_bilinear, Koltchinskii_Xia, Koltchinskii_Nickl}.

The paper is organized as follows. In Section \ref{preliminaries}, we introduce 
basic definitions and facts concerning multilinear forms, tensor products and smoothness in linear normed spaces (including the definition of H\"older spaces). In Section \ref{random_homotopies},
the notion of random homotopy between parameter of statistical model and 
its estimator is introduced. In Section \ref{sec:Holder_norms}, we develop 
bounds on H\"older norms of functions ${\mathcal B}^k f$ and $f_k.$
In Section \ref{sec:bootstrap_chain}, random homotopies are used to provide 
representations and approximations of bootstrap chains. In Section \ref{sec:RepresentB^k}, we develop representation formulas for functions
${\mathcal B}^k f$ (which also provide an alternative derivation of bounds on their 
H\"older norms). In Section \ref{normal_proofs}, concentration bounds for estimators
$f_k(\thetah)$ are proved in the case of normal models and the proof of Theorem 
\ref{main_theorem_1} is completed.

\section{Preliminaries: multilinear forms, tensor products and differentiability.}
\label{preliminaries}

For a linear space $E,$ let $E^{\#}$ denote its algebraic dual (the space of all linear functionals on $E$).
For $x\in E$ and $u\in E^{\#},$ we will use inner product notation $\langle x,u\rangle$ for the value 
of functional $u$ on vector $x.$ If $E$ is a Banach space, let $E^{\ast}$ denote its topological 
dual (the space of all continuous linear functionals on $E$). We will use the notation $\|\cdot\|$
for the norms of $E, E^{\ast}$ and other Banach spaces (sometimes, providing it with subscripts 
to avoid a confusion). Given linear spaces $E_1,\dots, E_k, F,$ let ${\mathcal L}_k(E_1,\dots, E_k;F)$
be the space of all $F$-valued $k$-linear forms $M[x_1,\dots, x_k], x_1\in E_1,\dots, x_k\in E_k$
(in other words, all the mappings $M: E_1\times \dots \times E_k\mapsto F$ linear with respect to 
each of their $k$ variables). If $E_1,\dots, E_k, F$ are Banach spaces, let 
${\mathcal M}_k(E_1,\dots, E_k;F)$ be the space of all $F$-valued continuous (bounded) $k$-linear forms $M: E_1\times \dots \times E_k\mapsto F.$ Such forms have bounded operator norms:
$$
\|M\|:= \sup_{\|x_1\|\leq 1, \dots, \|x_k\|\leq 1}\|M[x_1, \dots, x_k]\|<\infty.
$$
In what follows, the spaces ${\mathcal M}_k(E_1,\dots, E_k;F)$ of bounded multilinear 
forms are always equipped with operator norms.  
For $F={\mathbb R},$ we will write 
$$
{\mathcal L}_k(E_1,\dots, E_k):={\mathcal L}_k(E_1,\dots, E_k;{\mathbb R})
\ {\rm and}\  
{\mathcal M}_k(E_1,\dots, E_k):={\mathcal M}_k(E_1,\dots, E_k;{\mathbb R}).
$$
For $k=0,$ we have ${\mathcal M}_0(F)={\mathcal L}_0(F)=F$ and,  
for $k=1,$ ${\mathcal L}_1(E)= E^{\#},$ 
${\mathcal M}_1(E)= E^{*}.$ Sometimes, it could be convenient 
to view $0$-forms (vectors) as functions of an empty (blank) variable 
and to write $x=x[\ ], x\in F.$
If $E_1=\dots=E_k=E,$ we write ${\mathcal L}_k(E;F):= {\mathcal L}_k(E,\underset{k}{\dots} ,E;F)$
and ${\mathcal M}_k(E;F):= {\mathcal M}_k(E,\underset{k}{\dots} ,E;F).$
We denote by ${\mathcal L}_k^s(E;F)$ and ${\mathcal M}_k^s(E;F)$ the subspaces 
of symmetric multilinear forms and symmetric continuous multilinear forms, resp. 

One can identify $k$-linear forms $M\in {\mathcal L}_k(E_1,\dots, E_k;F)$ with linear 
mappings from $E_1\otimes \dots \otimes E_k$ into $F,$ where $E_1\otimes \dots \otimes E_k$
is the (algebraic) tensor product of linear spaces $E_1,\dots, E_k$ (if $E_1=\dots = E_k=E,$
we denote $E^{\otimes k}:= E\otimes \underset{k}{\dots} \otimes E$).
Similarly, continuous (bounded) 
linear forms $M\in {\mathcal M}_k(E_1,\dots, E_k;F)$ could be identified with bounded linear operators 
from a topological tensor product $E_1\otimes \dots \otimes E_k$ into $F.$ For simplicity, assume 
that $k=2$ (the generalization to the case $k>2$ is straightforward). Given $x_1\in E_1, x_2 \in E_2,$ the tensor product $x_1\otimes x_2$ could be defined 
as the linear functional $x_1\otimes x_2\in {\mathcal L}_2(E_1,E_2)^{\#}$ such that 
$$
\langle M, x_1\otimes x_2\rangle = M[x_1,x_2], M\in {\mathcal L}_2(E_1,E_2).
$$ 
The algebraic tensor product $E_1\otimes E_2$ is then defined as the linear span 
(in ${\mathcal L}_2(E_1,E_2)^{\#}$) of the set of functionals $x_1\otimes x_2, x_1\in E_1, x_2\in E_2.$ It is easy 
to check that the mapping $E_1\times E_2\ni (x_1,x_2)\mapsto x_1\otimes x_2\in E_1\otimes E_2$
is bilinear. It is also well known (so called ``universal mapping property") that any bilinear form $M\in {\mathcal L}_2(E_1,E_2;F)$
could be identified (up to a linear isomorphism) with a linear mapping (or $1$-linear form) 
$\tilde M\in {\mathcal L}_1(E_1\otimes E_2;F)$ such that 
$$
M[x_1,x_2]=\tilde M [x_1\otimes x_2], x_1\in E_1, x_2\in E_2
$$
(in what follows, with a little abuse of notation, we will write $M$ instead of $\tilde M$). 

Note that, given $u_1\in E_1^{\#}, u_2\in E_2^{\#},$ one can define a bilinear functional 
$$
M_{u_1,u_2}[x_1,x_2]:= \langle x_1,u_1\rangle \langle x_2,u_2\rangle, x_1\in E_1, x_2\in E_2.
$$   
This allows one to identify the tensor product $x_1\otimes x_2$ with the bilinear form 
$$
E_1\otimes E_2\ni (u_1,u_2) \mapsto \langle x_1,u_1\rangle \langle x_2,u_2\rangle =:(x_1\otimes x_2)[u_1,u_2]\in {\mathbb R},
$$
or with the linear mapping $x_1\otimes x_2: E_2^{\#}\mapsto E_1:$
$$
(x_1\otimes x_2)u:=x_1\langle x_2,u\rangle.
$$

There are many different ways to define topological tensor products of Banach spaces 
(or, more generally, linear topological spaces). In what follows, we will use so called 
\it projective \rm tensor products. For Banach spaces $E_1, E_2,$ define 
$$
\|x\|= \|x\|_{\pi}:=\inf\Bigl\{\sum_{i} \|x_i^{(1)}\|\|x_i^{(2)}\|: x=\sum_{i}x_i^{(1)}\otimes x_i^{(2)}\Bigr\},
x\in E_1\otimes E_2,
$$ 
where the infimum is taken over all the representations of $x\in E_1\otimes E_2$
as a finite sum $\sum_{i}x_i^{(1)}\otimes x_i^{(2)}.$ The completion of 
$E_1\otimes E_2$ with respect to this norm is called the projective tensor
product $E_1\otimes_{\pi} E_2.$ In what follows, we will drop the subscript 
$\pi$ in the above notation. It is well known that the projective tensor product 
norm is a cross-norm, so, it possesses the following properties:
$$
\|x_1\otimes x_2\|= \|x_1\|\|x_2\|, x_1\in E_1, x_2\in E_2\ {\rm and}\ 
\|u_1\otimes u_2\|= \|u_1\|\|u_2\|, u_1\in E_1^{\ast}, u_2\in E_2^{\ast}.
$$  
It is known that the projective norm is the only norm in $E_1\otimes E_2$
for which the universal mapping property extends to continuous bilinear  
forms: any continuous bilinear form $M\in {\mathcal M}_2(E_1,E_2;F)$
could be identified (up to an isomorphism) with a continuous linear mapping (or continuous $1$-linear form) 
$\tilde M\in {\mathcal M}_1(E_1\otimes E_2;F)$ such that 
$$
M[x_1,x_2]=\tilde M [x_1\otimes x_2], x_1\in E_1, x_2\in E_2.
$$
In what follows, we write interchangeably multilinear forms as $M[x_1,\dots, x_k]$
or as $M[x_1\otimes \dots \otimes x_k].$

If $E_1:=E, E_2:={\mathbb R},$ it is obvious that $x\otimes c=(cx)\otimes 1= c(x\otimes 1), x\in E, c\in {\mathbb R}$
and, moreover, the spaces $E\otimes {\mathbb R}$ and $E$ are isometric. This observation allows one to drop the factor ${\mathbb R}$ in tensor products of Banach spaces (and to replace the tensor product $x\otimes c$ by the vector $cx$). 

For $M_1 \in {\mathcal L}_k(E_1,\dots, E_k;F_1)$ and $M_2\in {\mathcal L}_l(E_{k+1},\dots, E_{k+l};F_2),$
define their tensor product $M_1\otimes M_2: {\mathcal L}_{k+l}(E_1,\dots, E_k, E_{k+1}, \dots, E_{k+l};F_1\otimes F_2)$
as follows:
\begin{align*}
&
(M_1\otimes M_2)[x_1,\dots, x_k,x_{k+1}, \dots, x_{k+l}]= M_1[x_1,\dots, x_k] M_2[x_{k+1}, \dots, x_{k+l}],
\\
&
x_j\in E_j, j=1,\dots, k+l.
\end{align*}
Obviously, if $M_1, M_2$ are continuous multilinear forms, so is $M_1\otimes M_2$
and 
\begin{align*}
&
\|M_1\otimes M_2\|=\|M_1\|\|M_2\|, 
\\
&
M_1\in {\mathcal M}_k(E_1,\dots, E_k;F_1), M_2\in {\mathcal M}_l(E_{k+1},\dots, E_{k+l};F_2).
\end{align*}
This definition easily extends to tensor products of several multilinear forms. 

Using representations of multilinear forms as linear transformations of tensor products, it is easy 
to write various superpositions of multilinear forms. For instance, if $M_1\in {\mathcal M}_2(E_1,E_2;F_1),$
$M_2\in {\mathcal M}_2(E_3,E_4;F_2),$ $M_3\in {\mathcal M}_2(F_1, F_2;F),$ then 
$M_3\circ (M_1\otimes M_2)\in {\mathcal M}_4(E_1,E_2, E_3, E_4;F)$ and 
\begin{align*}
&
(M_3\circ (M_1\otimes M_2))[x_1,x_2,x_3,x_4]=M_3[(M_1\otimes M_2)[(x_1\otimes x_2)\otimes (x_3\otimes x_4)]]
\\
&
=M_3[M_1[x_1\otimes x_2]\otimes M_2[x_3\otimes x_4]]= M_3[M_1[x_1,x_2], M_2[x_3,x_4]]. 
\end{align*}
It also immediate that 
$$
\|M_3\circ (M_1\otimes M_2)\| \leq \|M_1\|\|M_2\|\|M_3\|.
$$
If we view $0$-forms (vectors) $x_1,\dots, x_4$ as functions of an empty (blank) variable,
we can also write 
$$
M_3[M_1[x_1,x_2], M_2[x_3,x_4]]= M_3\circ (M_1\otimes M_2)\circ ((x_1\otimes x_2)\otimes (x_3\otimes x_4)).
$$

A function $f: D\subset E\mapsto F$ is called Fr\`echet differentiable at an interior point $x_0\in D$ iff there exists a
bounded linear operator $f^{\prime}(x_0)=(Df)(x_0): E\mapsto F$ (Fr\`echet derivative of $f$) such that 
$$
f(x_0+h)-f(x_0)=f^{\prime}(x_0)h + o(\|h\|)\ {\rm as}\ h\to 0.
$$
One can also view $f^{\prime}(x_0)=(Df)(x_0)$ as a $1$-linear form from the space 
${\mathcal M}_1(E;F)={\mathcal M}_1^s(E;F).$
The gradient notation $(\nabla f)(x_0)$ is also often used for the derivative 
$f^{\prime}(x_0)$ (especially, if $E={\mathbb R}^d, F={\mathbb R}$). 
Higher order Fr\`echet derivatives are then defined recursively as symmetric multilinear 
forms. Namely, assuming that $f^{(0)}(x)=f(x),$ that 
$f^{(k-1)}(x)=(D^{k-1} f)(x)\in {\mathcal M}_{k-1}^s(E;F)$ has been already defined in a neighborhood 
$U$ of point $x_0$ and 
that the mapping $U\ni x\mapsto f^{(k-1)}(x)\in {\mathcal M}_{k-1}^s(E;F)$ is Fr\`echet 
differentiable at point $x_0,$ we define $f^{(k)}(x_0) = (f^{(k-1)})'(x_0)$ as the Fr\`echet derivative of function $f^{(k-1)}(x)$ at point $x_0.$ 
It could be viewed as a bounded linear operator from $E$ into ${\mathcal M}_{k-1}^s(E;F),$
or, equivalently, as a bounded $k$-linear form $f^{(k)}(x_0)=(D^k f)(x_0)\in {\mathcal M}_k^s(E;F).$
We call $f$ $k$-times Fr\`echet continuously differentiable in an open set $U\subset E$ 
if it is defined and $k$ times Fr\`echet differentiable in $U$
and the function $U\ni x\mapsto f^{(k)}(x)\in {\mathcal M}_k^s(E;F)$ is continuous.
In all these definitions, we assume that the spaces of multilinear forms  
are equipped with operator norms.
  
Finally, we need to introduce H\"older spaces $C^s(U;F)$ of functions from an open subset $U\subset E$ to $F$ 
of smoothness $s>0.$ Let $s=k+\alpha, k\geq 0, \alpha \in (0,1].$ 
For $f:U\mapsto F,$ 
define 
\begin{align*}
&
\|f\|_{C^s(U;F)} :=
\\
& 
\max\biggl(\sup_{x\in U}\|f(x)\|, \max_{0\leq j\leq k-1}\sup_{x, x'\in U, x\neq x'} \frac{\|f^{(j)}(x)-f^{(j)}(x')\|}{\|x-x'\|},
\sup_{x,x'\in U, x\neq x'}\frac{\|f^{(k)}(x)-f^{(k)}(x')\|}{\|x-x'\|^{\alpha}}\biggr),
\end{align*}
where the norms of the derivatives mean operator norms of multilinear forms. 
%$$
%\|f\|_{C^s(D;F)} := 
%\inf_{U\supset D}\biggl[\max_{0\leq j\leq k}\sup_{x\in U} \|f^{(j)}(x)\|
%+ \sup_{x,x'\in U, x\neq x'}\frac{\|f^{(k)}(x)-f^{(k)}(x')\|}{\|x-x'\|^{\alpha}}\biggr],
%$$ 
%where the infimum is taken over all such extensions. 
Let $C^s(U;F)$ denote the space of $k$ times Fr\`echet continuously differentiable 
functions $f$ from $U$ into $F$ with $\|f\|_{C^s(U;F)}<+\infty.$

\begin{remark}
\normalfont
Note that this definition is not quite standard. It is common to define $C^s$-norms in terms of $\sup_{x\in U} \|f^{(j)}(x)\|$ up to the maximal order of the derivatives that exist. It is well known that 
$$
\sup_{x\in U} \|f^{(j)}(x)\|\leq  \sup_{x,x'\in U, x\neq x'}\frac{\|f^{(j-1)}(x)-f^{(j-1)}(x')\|}{\|x-x'\|}
$$
and, if $U$ is convex,  
$$
\sup_{x\in U} \|f^{(j)}(x)\|=\sup_{x,x'\in U, x\neq x'}\frac{\|f^{(j-1)}(x)-f^{(j-1)}(x')\|}{\|x-x'\|}.
$$
In the case of convex $U,$ the definition of $C^s$-norms we use coincides with more standard.  
\end{remark}

In the cases when the set $U$ and/or the space $F$ are known from the context (in particular, when $F={\mathbb R}$), we often write $C^s(U)$ or $C^s$  instead of $C^s(U;F).$

We will also use the following seminorm
\begin{align*}
&
\|f\|_{C^s(U;F)}^{-}:=
\max\biggl(\max_{0\leq j\leq k-1}\sup_{x, x'\in U, x\neq x'} \frac{\|f^{(j)}(x)-f^{(j)}(x')\|}{\|x-x'\|},
\sup_{x,x'\in U, x\neq x'}\frac{\|f^{(k)}(x)-f^{(k)}(x')\|}{\|x-x'\|^{\alpha}}\biggr)
\end{align*}
that could be finite even when function $f$ is not uniformly bounded.

\begin{comment}

For $s\in (0,1),$ let 
$$
\|f\|_{C^s(U;F)} := \sup_{x\in U}\|f(x)\|+ \sup_{x,x'\in U, x\neq x'}\frac{\|f(x)-f(x')\|}{\|x-x'\|^s}.
$$
If $s=k\geq 1$ is a natural number, let 
$$
\|f\|_{C^k(U;F)}:= \max_{0\leq j\leq k}\sup_{x\in U} \|f^{(j)}(x)\|.
$$ 
Finally, if $s=k+\alpha$ for a natural number $k$ and $\alpha \in (0,1),$
set 
$$
\|f\|_{C^s(U;F)} := \max_{0\leq j\leq k}\sup_{x\in U} \|f^{(j)}(x)\|
+ \sup_{x,x'\in U, x\neq x'}\frac{\|f^{(k)}(x)-f^{(k)}(x')\|}{\|x-x'\|^{\alpha}}.
$$ 
We define $C^k(U;F)$ as the space of $k$ times Fr\`echet continuously differentiable 
functions $f$ from $U$ to $F$ with $\|f\|_{C^k(U;F)}<+\infty.$ For $s=k+\alpha, k\in {\mathbb Z}_+, \alpha\in (0,1),$ $C^s(U;F)$ is the space of $k$ times Fr\`echet continuously differentiable 
functions $f$ from $U$ to $F$ with $\|f\|_{C^s(U;F)}<+\infty.$
%$$
%C^s(U;F):= \{f:U\mapsto F: \|f\|_{C^s(U;F)}<+\infty\}.
%$$

\end{comment}

\begin{remark}
\normalfont
For a continuous function $f:{\mathbb  R}\mapsto {\mathbb R},$ one can define its 
value $f(A)$ on a symmetric matrix $A\in {\mathcal S}^d$ (or, more generally, on self-adjoint 
operators in a Hilbert space) via standard continuous functional calculus. Assume that the space ${\mathcal S}^d$ is equipped 
with the operator norm. Using the methods of \cite{Aleksandrov_Peller}, 
it could be proved that 
$$
\|f\|_{C^s({\mathcal S}^d)}\lesssim_s \|f\|_{B^{s}_{\infty,1}({\mathbb R})},
$$
where $\|f\|_{C^s({\mathcal S}^d)}$ is the $C^s$-norm of function ${\mathcal S}^d\ni A\mapsto f(A)\in {\mathcal S}^d$ and $\|f\|_{B^{s}_{\infty,1}({\mathbb R})}$ is a Besov norm of $f:{\mathbb R}\mapsto {\mathbb R}$
(see \cite{Koltchinskii_2017}). This allows one to control H\"older norms of such functions of self-adjoint operators (in particular, of covariance operators).  
\end{remark}

For an arbitrary space $T$ and a Banach space $E,$ $L_{\infty}(T)$ denotes the space of uniformly 
bounded functions $f:T\mapsto E$ equipped with the norm:
$$
\|f\|_{L_{\infty}}:= \sup_{t\in T}\|f(t)\|.
$$ 
For a metric space $(T,d),$ ${\rm Lip}_d(T)$ denotes the space of Lipschitz functions on $T$
with the norm 
$$
\|f\|_{{\rm Lip}_{d}(T)}:= \|f\|_{L_{\infty}}\bigvee \sup_{t\neq t'}\frac{\|f(t)-f(t')\|}{d(t,t')}.
$$
Sometimes we also deal with the space ${\rm Lip}_{d,\rho}(T):= {\rm Lip}_{d^{\rho}}(T)$
of H\"older functions on $T$ with exponent $\rho\in (0,1]$ equipped with the norm 
$\|f\|_{{\rm Lip}_{d,\rho}(T)}:=\|f\|_{{\rm Lip}_{d^{\rho}}(T)}.$ It is easy to check that 
\begin{align}
\label{hoeld_lip}
\|f\|_{{\rm Lip}_{d,\rho}(T)}\leq 2 \|f\|_{{\rm Lip}_{d}(T)},\ f\in {\rm Lip}_d(T).
\end{align}

We will use the following simple proposition about differentiation under the expectation sign. Its proof 
is based on the dominated convergence theorem and is elementary.

\begin{proposition}
\label{diff_under_exp}
Let $E,F$ be Banach spaces and let $T \subset E$ be an open set equipped with the metric $d$ of Banach space $E.$ Let $\xi (t), t \in T$
be a stochastic process with values in $F.$ If $\xi$ is $k$ times continuously differentiable a.s. and, for all 
$j=0,\dots, k-1,$ 
$$
{\mathbb E}\|D^{j} \xi\|_{{\rm Lip}_d(T)}<\infty,
$$
then $T\ni t \mapsto {\mathbb E}\xi(t)$ is $k$ times continuously differentiable 
with 
$$
D^{k}{\mathbb E} \xi(t) = {\mathbb E} (D^k \xi)(t), t\in T.
$$
\end{proposition}

\begin{remark}
\normalfont
Note that we often use generic notations for function spaces and norms 
not necessarily providing them with subscripts (that would make the notations too complicated). 
The meaning of these notations should be clear from the context. 
In more ambiguous cases, we 
try to provide necessary clarifications.
For instance, the notation $\|\cdot\|$
is used for the norm of the underlying Banach space $E,$ for the norm of its dual space $E^{\ast}$
and for the norms of spaces of bounded multilinear forms ${\mathcal M}_k (E;F), k\geq 1.$ In the case of normal 
model (our main example), the same notation is used for the standard Euclidean norm in ${\mathbb R}^d,$ 
for the operator norm in the space of symmetric operators ${\mathcal S}^d$ and for the norm of the space 
${\mathbb R}^d\times {\mathcal S}^d,$ as defined in Section \ref{MainResults}. However, the Hilbert--Schmidt 
norm in ${\mathcal S}^d$ is defined by $\|\cdot\|_2$ and the nuclear norm by $\|\cdot\|_1$ (to distinguish them 
from the operator norm used by default).
For a function 
$f:U\subset E\mapsto F,$ where $E,F$ are Banach spaces,
$\|f^{(k)}\|_{L_{\infty}}$ means 
$$
\|f^{(k)}\|_{L_{\infty}}= \sup_{x\in U}\|f^{(k)}(x)\|,
$$
where $\|f^{(k)}(x)\|$ is (by default) the operator norm of the $k$-linear form 
$f^{(k)}(x)\in {\mathcal M}_k^s(E;F)$ (as it was the case in the definition of $C^s$-norms).
Similarly, for $d(x,x')=\|x-x'\|, x, x'\in E,$
$$
\|f^{(k)}\|_{{\rm Lip}_d(U)} =  \|f^{(k)}\|_{L_{\infty}}\bigvee \sup_{x,x'\in U, x\neq x'}
\frac{\|f^{(k)}(x)-f^{(k)}(x')\|}{\|x-x'\|}.
$$
\end{remark}

\section{Bias reduction and random homotopies.}
\label{random_homotopies}

Assume that $\Theta$ is a subset of a Banach space $E.$ 
A random homotopy is, roughly, 
a stochastic process $H(\theta;t), \theta \in \Theta, t\in [0,1]$ that continuously transforms the parameter
$\theta$ (for $t=0$) into a random variable with the same distribution as its estimator $\thetah $ (for $t=1$).

\begin{definition}
Given 
a probability space $(\Omega, {\mathcal F}, {\mathbb P}),$ let $H: \Theta\times [0,1]\times\Omega \mapsto \Theta$ be an a.s. continuous stochastic process  
%$H(\theta;t)\coloneqq H(\theta;t;\omega),\theta \in \Theta, t\in [0,1]$ 
such that, for all $\theta \in \Theta,$ 
\begin{gather*}
H(\theta;0) \coloneqq \theta, \  H(\theta;1) \overset{d}{=} \thetah,\ {\it where}\ \thetah \sim P(\theta;\cdot).
\end{gather*}
$H$ will be called a random homotopy between the parameter $\theta$ and its estimator $\thetah.$
\end{definition}
In what follows, it will be usually assumed that the process $H$ satisfies further smoothness assumptions
(for instance, it is $C^k$ for some $k\geq 1$). We provide below several examples of random homotopies.
 
 \begin{example} {\bf Random shift model.} \normalfont
 Let $X=\theta + \xi$ be an observation of an unknown 
 parameter $\theta \in \Theta= E$ in random noise $\xi$ with ${\mathbb E}\xi =0$ and with known distribution. 
 We will call this model a {\it random shift model.} In particular, if $\xi \sim N(0;\Sigma)$
 is a Gaussian random vector with mean zero and known covariance operator $\Sigma,$ it will be called 
 a Gaussian shift model. Let $\thetah = \thetah (X):=X$ (in the case of Gaussian shift model in 
 ${\mathbb R}^d$ with $\Sigma=\sigma^2 I_d,$ it is the maximum 
 likelihood estimator (MLE)). A random homotopy could be defined simply as 
 $H(\theta;t)= \theta + t\xi, \theta \in E, t\in [0,1].$
 More generally, let $\Theta\subset E$ be a subset of Banach space $E$ and let 
 $$
 \thetah = \thetah(X) = P_{\Theta}(X):= {\rm argmin}_{t\in \Theta} \|X-t\|
 $$
be the metric projection of $X$ onto $\Theta$ (assume, for simplicity, that the minimal point 
exists and is unique). This is again the MLE in the case of Gaussian shift model with spherically symmetric noise
in the space ${\mathbb R}^d$ (equipped with the standard Euclidean norm). 
We can now define a random homotopy as 
 $$
 H(\theta;t)= P_{\Theta}(\theta + t\xi), \theta \in \Theta, t\in [0,1].
 $$ 
 \begin{comment}
 In the case when $\Theta$ is a smooth Riemannian manifold isometrically embedded in $E$ and 
 $\thetah\in \Theta$ is an arbitrary estimator of $\theta,$ one can alternatively define a random homotopy 
as 
$$
H(\theta;t) := \exp_{\theta} (t\exp_{\theta}^{-1}(\thetah(\theta+\xi))), \theta\in \Theta, t\in [0,1],
$$ 
where $\exp_{\theta}: T\Theta_{\theta}\mapsto \Theta$  is the exponential map.
\end{comment}
\end{example}

\begin{comment} 
 
\begin{example}{\bf Random ODE model.}
\normalfont
Let $v: [0,1]\times E\mapsto E$ be a stochastic process (with known distribution). Consider 
the following differential equation in Banach space $E:$
$$
\dot x(t) = v(t,x(t)),\ x(0)=\theta \in \Theta\subset E.
$$
Suppose it has unique solution $x(t)$ on the interval $[0,1].$
Let $X=x(1)$ be an observation of parameter $\theta$ 
and let $\thetah= \thetah (X)$ be its estimator 
based on $X.$ Suppose that $E\ni x\mapsto \thetah(x)\in \Theta$
is a ``smooth" function such that $\thetah (x)=x, x\in \Theta$
(for instance, $\thetah(x)=P_{\Theta}(x), x\in E$). 
A random shift model is a special case of this more general model
with $v(t;x)=\xi, t\in [0,1], x\in E.$
A random homotopy $H(\theta;t)$ could be now defined as follows:
$$
H(\theta;t):= \thetah (x(t)), \theta \in \Theta, t\in [0,1].
$$
\end{example}

\end{comment}

\begin{example}{\bf Unknown covariance.} 
\normalfont
Let $X_1,\dots, X_n$ be i.i.d. observations of a random vector 
$X$ in $E={\mathbb R}^d$ with mean zero and unknown covariance $\Sigma={\mathbb E}(X\otimes X).$
Moreover, assume that $X=\Sigma^{1/2}Z,$ where $Z$ is a random vector with mean zero, covariance 
matrix $I_d$ and known distribution. In particular, $Z$ could be standard normal implying that $X\sim N(0,\Sigma).$
Let $\hat \Sigma:= n^{-1}\sum_{j=1}^n X_j\otimes X_j$ be the 
sample covariance based on observations $X_1,\dots, X_n.$ In this case, we can define a random 
homotopy as 
$$
H(\Sigma;t):= \Sigma^{1/2}\Biggl((1-t)I_d+ t n^{-1}\sum_{j=1}^n Z_j\otimes Z_j\Biggr) \Sigma^{1/2}, \Sigma\in {\mathcal C}_+^d, t\in [0,1],
$$
where $Z_1,\dots, Z_n$ are i.i.d. copies of $Z.$
\end{example}

\begin{example}\label{ex_mean_covariance}{\bf Unknown mean and covariance.}
\normalfont
Let $X_1,\dots, X_n$ be i.i.d. observations of a random vector 
$X$ in $E={\mathbb R}^d$ with unknown mean $\mu={\mathbb E}X$ and unknown covariance $\Sigma={\mathbb E}(X-{\mathbb E} X)\otimes (X-{\mathbb E}X).$ Assume that $X=\mu +\Sigma^{1/2}Z,$
where $Z$ is a random vector with mean zero, covariance 
matrix $I_d$ and known distribution. In particular, $Z$ could be standard normal implying that $X\sim N(\mu,\Sigma).$ Let 
$$
\bar X:=n^{-1}\sum_{j=1}^n X_j\ {\rm and}\ \hat \Sigma:= \frac{1}{n-1}\sum_{j=1}^n (X_j-\bar X)
\otimes (X_j-\bar X)
$$ 
be the sample mean and 
the sample covariance based on observations $X_1,\dots, X_n.$ 
We can define a random 
homotopy as 
\begin{align*}
&
H((\mu,\Sigma);t):=
\Bigl(\mu + t\Sigma^{1/2}\bar Z,  
\Sigma^{1/2}\Bigl((1-t)I_d+ t \hat \Sigma_Z\Bigr)\Sigma^{1/2}\Bigr),
\\
&
\mu\in {\mathbb R}^d, \Sigma\in {\mathcal C}_+^d, 
t\in [0,1],
\end{align*}
where $Z_1,\dots, Z_n$ are i.i.d. copies of $Z,$ 
$$
\bar Z:=n^{-1}(Z_1+\dots+Z_n), 
\hat \Sigma_Z:=\frac{1}{n-1}\sum_{j=1}^n (Z_j-\bar Z)\otimes (Z_j-\bar Z).
$$
\end{example}

\begin{example}
{\bf More general couplings.}
\normalfont
A more general class of examples could be described as follows. 
Let $(S,d)$ be a metric space with measure $P$ on its Borel $\sigma$-algebra
and let $\xi\sim P.$ Consider a measurable space $(T,{\mathcal B}_T)$ 
and let $g_{\theta}:S\mapsto T, \theta\in \Theta$ be a family of measurable 
functions. Suppose that $P_{\theta}:= P\circ g_{\theta}^{-1}, \theta\in \Theta$
(implying that $g_{\theta}(\xi)\sim P_{\theta}$). Finally, let $\gamma: \Theta \times \Theta \times [0,1]\mapsto 
\Theta$ be a fixed smooth function. Given $\theta, \theta'\in \Theta,$ 
$\gamma(\theta, \theta'; t), t\in [0,1]$ provides a smooth path in $\Theta$ between the points 
$\theta$ and $\theta'.$
For instance, if $\Theta$ is a convex subset of $E,$ one can take 
$$
\gamma (\theta, \theta', t):= (1-t) \theta + t\theta', \theta, \theta'\in \Theta, t\in [0,1].
$$
\begin{comment}
and if $\Theta$ is a smooth Riemannian manifold isometrically embedded in $E,$ one can take 
$$
\gamma (\theta, \theta', t):=\exp_{\theta} (t\exp_{\theta}^{-1}(\theta')), \theta, \theta'\in \Theta, t\in [0,1].
$$
\end{comment}
Given an estimator $\thetah (X)$ of parameter $\theta\in \Theta$ based on an observation 
$X\sim P_{\theta},$ 
one can define a random homotopy as follows:
$$
H(\theta;t):= \gamma(\theta, \thetah(g_{\theta}(\xi)), t), \theta\in \Theta, t\in [0,1].
$$
Smoothness of $H(\theta;t)$ would follow from proper smoothness of the mapping 
$\Theta\ni\theta\mapsto \thetah(g_{\theta}(x))\in \Theta, x\in S.$

In the case when $S$ is a compact Riemannian manifold and $P_{\theta}, \theta \in \Theta$
is a statistical model, where measures $P_{\theta}$ are absolutely continuous with respect to the normalized  Riemannian 
volume $P$ with smooth densities $p_{\theta}$ bounded away from zero, one can use well known 
\it Moser's 
coupling \rm (see \cite{Villani}) to construct smooth mappings $g_{\theta}:S\mapsto S$ such that $P_{\theta}=P\circ g_{\theta}^{-1}.$ To this end, let $p$ denote 
the density of measure $P$ (in fact, $p=1$) and let $u_{\theta}$ be a solution 
of Poisson equation 
$
\Delta u = p-p_{\theta}.
$
Define a vector field
$$
v_{\theta}(t;x):= \frac{\nabla u_{\theta}(x)}{(1-t)p(x)+t p_{\theta}(x)}, x\in S, t\in [0,1].
$$
and let $T^{t}(x)=T^t_{\theta}(x)$ be the flow on the manifold $S$ generated by $v_{\theta}.$
Then one can define $g_{\theta}(x):=T^1_{\theta}(x), x\in S.$
In the case when $(\theta,x) \mapsto p_{\theta}(x)$ and $x\mapsto \thetah(x)$
are smooth, this allows us to construct a smooth random homotopy between $\theta$ and $\thetah.$

There exists also a version of Moser's coupling (and the corresponding smooth random homotopies) in the case of non-compact Riemannian manifold
$S$ with Riemannian volume $\mu$ and with reference measure $P(dx)=e^{-V(x)}\mu(dx),$ where $V$ is a smooth function of $S.$
In this case, $p=1$ (as before) and $u_{\theta}$ is a solution of the equation 
$$
\Delta u - \langle \nabla V,\nabla u\rangle = p-p_{\theta}.
$$
\end{example}

\vskip 2mm

Let $\Theta\subset E$ be an open subset and let $H(\theta;t), \theta\in \Theta, t\in [0,1]$ be a random homotopy between $\theta$ and $\thetah$
and suppose that stochastic process $H$ is a.s. $k+1$ times continuously differentiable in $\Theta\times [0,1].$ 
Denote $\dot H (\theta;t):= \frac{d}{dt} H(\theta;t).$ 

It will be convenient to use the following norms for functions $V:\Theta \times [0,1]\mapsto F$ with values in a Banach space $F$
such that $V(\cdot;t)\in C^s(\Theta), t\in [0,1]$ for some $s>0:$ 
$$
\|V\|_{C^{s,0}(\Theta\times [0,1])}:= \sup_{t\in [0,1]}\|V(\cdot;t)\|_{C^{s}(\Theta)}.
$$
Denote 
$$
C^{s,0}(\Theta\times [0,1]):=\Bigl\{V:\Theta \times [0,1]\mapsto F: \|V\|_{C^{s,0}(\Theta\times [0,1])}<+\infty\Bigr\}.
$$
Similarly, define 
$$
\|V\|_{C^{s,0}(\Theta\times [0,1])}^{-}:= \sup_{t\in [0,1]}\|V(\cdot;t)\|_{C^{s}(\Theta)}^{-}.
$$

The following result will play an important role in the paper. 

\begin{theorem}
\label{th_H_bound}
Suppose that $\Theta\subset E$ is an open set, 
$f\in C^s(\Theta)$ for some $s=k+1+\rho, k\geq 1, \rho\in (0,1]$
and  
\begin{align}
\label{moment_of_H_deriv_AAA}
{\mathbb E}\Bigl(\|H\|_{C^{s-1,0}(\Theta\times [0,1])}^{-}\vee 1\Bigr)^{s-1}
\|\dot H\|_{C^{s-1,0}(\Theta\times [0,1])}<\infty.
\end{align} 
Then, for some constant $D_s\geq 1$ and for all $1\leq j\leq k$  
\begin{align}
\label{bound_on_B^k_deriv_Lip_AAA}
\|B^{j} f\|_{C^{1+\rho}(\Theta)} \leq D_s \|f\|_{C^{s}(\Theta)}
\biggl({\mathbb E}\Bigl(\|H\|_{C^{s-1,0}(\Theta\times [0,1])}^{-}\vee 1\Bigr)^{s-1}
\|\dot H\|_{C^{s-1,0}(\Theta\times [0,1])}\biggr)^j.
\end{align}
\end{theorem}

As a corollary, we have a simple way to control the smoothness of function $f_k:$

\begin{corollary}
\label{cor_bd_f_k_hold}
Suppose the conditions of Theorem \ref{th_H_bound} hold 
and 
\begin{align}
\label{cond_leq_0.5}
D_s {\mathbb E}\Bigl(\|H\|_{C^{s-1,0}(\Theta\times [0,1])}^{-}\vee 1\Bigr)^{s-1}
\|\dot H\|_{C^{s-1,0}(\Theta\times [0,1])}\leq 1/2,
\end{align} 
then 
\begin{align*}
\|f_k\|_{C^{1+\rho}(\Theta)} \leq 2 \|f\|_{C^s(\Theta)}.
\end{align*}
\end{corollary}

\begin{proof}
In view of bound \eqref{bound_on_B^k_deriv_Lip_AAA} and condition \eqref{cond_leq_0.5},
\begin{align*}
&
\|f_k\|_{C^{1+\rho}(\Theta)} = \biggl\|\sum_{j=0}^k (-1)^j {\mathcal B}^j f\biggr\|_{C^{1+\rho}(\Theta)}
\leq \sum_{j=0}^k \|{\mathcal B}^j f\|_{C^{1+\rho}(\Theta)}
\\
&
\leq \sum_{j=0}^k 2^{-j} \|f\|_{C^s(\Theta)}
\leq 2\|f\|_{C^s(\Theta)}.
\end{align*}
\qed
\end{proof}

Another consequence of \eqref{bound_on_B^k_deriv_Lip_AAA} is the following bound 
on the bias of estimator $f_k(\thetah).$

\begin{theorem}
Suppose the conditions of Theorem \ref{th_H_bound} hold. 
Then, for all $\theta\in \Theta,$ 
\begin{align}
\label{bias_XXX}
\nonumber
|{\mathbb E}_{\theta} f_k(\thetah)-f(\theta)|
\lesssim_s 
\|f\|_{C^{s}(\Theta)}
&
\biggl({\mathbb E}\Bigl(\|H\|_{C^{s-1,0}(\Theta\times [0,1])}^{-}\vee 1\Bigr)^{s-1}
\|\dot H\|_{C^{s-1,0}(\Theta\times [0,1])}\biggr)^k 
\\
&
\times \biggl(\biggl\|{\mathbb E}\int_0^1 \dot H (\theta;t) dt\biggr\|+ {\mathbb E}\|\dot H\|_{L_{\infty}(\Theta\times [0,1])}^{1+\rho}\biggr).
\end{align}
\end{theorem}

\begin{proof}
Note that 
\begin{align*}
&
({\mathcal B}^{k+1} f)(\theta)
={\mathbb E}_{\theta} ({\mathcal B}^k f)(\thetah)-({\mathcal B}^k f)(\theta)
\\
&
= {\mathbb E} \Bigl(({\mathcal B}^k f)(H(\theta;1))-({\mathcal B}^k f)(H(\theta;0))\Bigr)
={\mathbb E}\int_0^1 ({\mathcal B}^k f)'(H(\theta;t))[\dot H(\theta;t)]dt
\\
&
= ({\mathcal B}^k f)'(\theta)\biggl[{\mathbb E}\int_0^1 \dot H (\theta;t) dt\biggr]
+ {\mathbb E}\int_0^1 \Bigl(({\mathcal B}^k f)'(H(\theta;t))-({\mathcal B}^k f)'(\theta)\Bigr)[\dot H(\theta;t)]dt.
\end{align*}
Thus, using \eqref{bias_formula} and the bound 
$\|H(\theta;t)-\theta\|\leq \|\dot H\|_{L_{\infty}(\Theta\times [0,1])}, \theta\in \Theta, t\in [0,1],$
we get
\begin{align*}
&
\nonumber
|{\mathbb E}_{\theta} f_k(\thetah)-f(\theta)|=|({\mathcal B}^{k+1} f)(\theta)|
\\
&
\leq 
\|{\mathcal B}^k f\|_{C^{1+\rho}(\Theta)}
\biggl(\biggl\|{\mathbb E}\int_0^1 \dot H (\theta;t) dt\biggr\|+ {\mathbb E}\|\dot H\|_{L_{\infty}(\Theta\times [0,1])}^{1+\rho}\biggr)
\end{align*}
It remains to use bound \eqref{bound_on_B^k_deriv_Lip_AAA}
for $j=k$ to complete the proof.

\qed
\end{proof}

\begin{remark}
\normalfont
For $s=1+\rho, \rho\in (0,1]$ and $k=0,$ bound \eqref{bias_XXX} takes the following 
form:
\begin{align}
\label{bias_XXX_AAA'''}
%\nonumber
|{\mathbb E}_{\theta} f(\thetah)-f(\theta)|
%&
%\lesssim_s 
%\|f\|_{C^{s}(\Theta)}
%(\|{\mathbb E}_{\theta}\thetah-\theta\|+ 
%&{\mathbb E}_{\theta}\|\thetah-\theta\|^{1+\rho})
%\\
\lesssim_s 
\|f\|_{C^{s}(\Theta)}
\biggl(\biggl\|{\mathbb E}\int_0^1 \dot H (\theta;t) dt\biggr\|+ {\mathbb E}\|\dot H\|_{L_{\infty}(\Theta\times [0,1])}^{1+\rho}\biggr).
\end{align}
\end{remark}

If $\Theta\subset E$ is convex and $G(\theta), \theta \in \Theta$ is a stochastic process 
with values in $\Theta$ such that $G(\theta) \overset{d}{=}\hat \theta(X), X\sim P_{\theta},$
then one can define 
$$H(\theta;t):= \theta + tE(\theta), \theta \in \Theta, t\in [0,1],$$
where $E(\theta):= G(\theta)-\theta, \theta\in \Theta.$ In this case,
$$\dot H(\theta;t)= E(\theta), (DH)(\theta;t) = I+t(DE)(\theta),$$
the norm $\|\cdot\|_{C^{s-1}(\Theta)}^{-}$ of function $\theta$ is equal to $1,$
and bound \eqref{bias_XXX} becomes 
\begin{align}
\label{bias_UVW}
\nonumber
|{\mathbb E}_{\theta} f_k(\thetah)-f(\theta)|
\lesssim_s 
\|f\|_{C^{s}(\Theta)}
&
\biggl({\mathbb E}\Bigl(\|E\|_{C^{s-1}(\Theta)}\vee 1\Bigr)^{s-1}
\|E\|_{C^{s-1}(\Theta)}\biggr)^k 
\\
&
\times \biggl(\|{\mathbb E}E(\theta)\|+ {\mathbb E}\|E\|_{L_{\infty}(\Theta)}^{1+\rho}\biggr).
\end{align}
If, in addition $\thetah$ is an unbiased estimator of $\theta,$ we have 
${\mathbb E}E(\theta)={\mathbb E}_{\theta}\thetah -\theta=0,$
implying 
\begin{align}
\label{bias_ZZZ}
|{\mathbb E}_{\theta} f_k(\thetah)-f(\theta)|
\lesssim_s 
\|f\|_{C^{s}(\Theta)}
&
\biggl({\mathbb E}\Bigl(\|E\|_{C^{s-1}(\Theta)}\vee 1\Bigr)^{s-1}
\|E\|_{C^{s-1}(\Theta)}\biggr)^k {\mathbb E}\|E\|_{L_{\infty}(\Theta)}^{1+\rho}.
\end{align}

In view of the importance of this problem, we develop in the following sections two different approaches to bounding 
H\"older norms of ${\mathcal B}^k f$ leading to slightly different bounds.
One of these approaches is more involved and will 
be discussed in detail in Section \ref{sec:RepresentB^k}. It is based on representing 
${\mathcal B}^k f$ in terms of a Markov chain with transition kernel $P(\theta, \cdot), \theta \in \Theta$
(bootstrap chain) and on further representation of this Markov chain as a superposition of 
i.i.d. copies of random homotopy $H.$
This leads to explicit formulas for functions ${\mathcal B}^k f$ for sufficiently smooth function $f$
that could be of independent interest.  
%using formula \eqref{B^k_int} and certain explicit representations of partial derivatives 
%$\frac{\partial^k f(G_k(\theta;t_1,\dots, t_k))}{\partial t_1\dots \partial t_k}.$
Another approach (discussed in Section \ref{sec:Holder_norms}) does not rely 
on explicit representation formulas for ${\mathcal B}^k f,$ but rather on some bounds 
for the norm $\|{\mathcal B}\|_{C^{s}\mapsto C^{s-1}}$ of operator ${\mathcal B}.$

\section{Bootstrap chains: representations as superpositions of random homotopies.}
\label{sec:bootstrap_chain}

%Under further assumptions, we will develop certain representations for functions 
%${\mathcal B}^k f$ that would allow us to bound these functions and, as a consequence,
%the bias ${\mathbb E}_{\theta} f_k(\thetah)-f(\theta)$ of estimator $f_k(\thetah).$

We will now develop representations of bootstrap chains $\{\thetah^{(k)}:k\geq 0\},$ 
introduced in Section \ref{MainResults}, in terms of superpositions  
of independent random homotopies. This approach relies on some ideas 
that originated in dynamical systems literature (see, e.g., \cite{Jost} and references therein).
It will be assumed throughout the section that $\Theta\subset E$ is an open subset.

Let $H_{1},H_{2}, \dots $ be i.i.d. copies of the process $H(\theta;t), \theta\in \Theta, t\in [0,1].$ 
Introduce the following sequence of functions: 
\begin{align*}
G_{1}(\theta;t_{1}) &\coloneqq H_{1}(\theta;t_{1}), \\
G_{2}(\theta;t_{1},t_{2}) &\coloneqq H_{2}(G_{1}(\theta;t_{1});t_{2})= H_{2}(H_{1}(\theta;t_{1});t_{2}),\\
G_{k}(\theta;t_{1},t_{2},\dots,t_{k}) &\coloneqq H_{k}(G_{k-1}(\theta;t_{1},\dots,t_{k-1});t_{k})
\\&=
 H_{k}(H_{k-1}(\dots H_{2}( H_{1}(\theta;t_{1});t_{2}), \dots ;t_{k-1});t_{k}),
\end{align*}
where $t_{1},\dots,t_{k} \in [0,1]$. 
In other words, we can define the superposition of stochastic processes $F_1: \Theta\times [0,1]^{l}\times \Omega \mapsto \Theta$ and $F_2: \Theta\times [0,1]^{m}\times \Omega \mapsto \Theta$ as 
the process $F_1\bullet F_2: \Theta\times[0,1]^{l+m} \times \Omega \times \Theta$ 
such that 
$$
(F_1\bullet F_2)(\theta; t_{1},\dots, t_{l}, t_{l+1},\dots, t_{l+m})\coloneqq 
F_1(F_2(\theta;  t_{l+1},\dots, t_{l+m}); t_{1},\dots, t_{l}).
$$ 
With this notation, $G_k = H_k\bullet H_{k-1}\bullet \dots \bullet H_1, k\geq 1.$
It will be also convenient to set $G_0(\theta):= \theta, \theta\in \Theta.$
Note that $\thetah^{(0)}\coloneqq \theta, \thetah^{(1)} \overset{d}{=} G_{1}(\theta; 1)$. This property is generalized in the following lemma.

\begin{comment}
\begin{gather*}
\thetah^{(0)}\coloneqq \theta,\  \thetah^{(1)} \overset{d}{=} G_{1}(\theta; 1),\ 
\thetah^{(2)} \overset{d}{=} G_{2}(G_{1}(\theta; 1); 1), \dots\\
\thetah^{(k)} \overset{d}{=} (H_{l}\bullet\dots \bullet H_{1})(\theta;1,\dots,1).
\end{gather*}
\end{comment}

\begin{lemma}
\label{represent_homotopy}
Consider $\thetat^{(0)}\coloneqq \theta$, $\thetat^{(k)} \coloneqq G_{k}(\theta;1,\dots,1)$ for $k\geq 1$. 
It holds that 
\begin{gather*}
(\thetat^{(k)}, k\geq 0 ) \overset{d}{=} (\thetah^{(k)}, k\geq 0)
\end{gather*}
and
\begin{gather*}
\thetah^{(l)}\overset{d}{=}
G_{k}(\theta;t_{1},\dots, t_{k}) 
\overset{d}{=}
(H_{l}\bullet\dots \bullet H_{1})(\theta;1,\dots,1)
\end{gather*}
for any $0\leq l \leq k $ and $(t_{1},\dots, t_{k})\in \{0,1\}^k$  such that $\sum_{j=1}^{k} t_{j}=l$.
\end{lemma}

\begin{proof} 
By the definition of the sequences $\thetat^{(k)}$ and $G_k,$ 
$$
\thetat^{(k)}=G_{k}(\theta;1,\dots,1)=H_k(\thetat^{(k-1)},1).
$$
Using this fact and the definition of random homotopy $H_k,$ we get that, given $\thetat^{(0)}, \dots, \thetat^{(k-1)},$ $\thetat^{(k)}\sim P(\thetat^{(k-1)},\cdot).$
Therefore, $\{\thetat^{(k)}: k\geq 0\}$ is a Markov chain with transition kernel $P$ and 
$\thetat^{(0)}=\theta,$ implying the first part of lemma's statement.

To prove the second part, 
let $J\coloneqq \{1\leq j\leq k: t_{j} =1\}=\{j_{1}<\dots <j_{l} \}.$ Then 
\begin{align*}
&
G_k(\theta;t_1,\dots, t_k)=(H_{j_{l}}\bullet\dots \bullet H_{j_{1}})(\theta;t_{j_{1}},\dots,t_{j_{l}})
\\
&
\overset{d}{=}(H_{l}\bullet\dots \bullet H_{1})(\theta;1,\dots,1)\overset{d}{=}G_l(\theta;1,\dots, 1)
\overset{d}{=}\thetah^{(l)},
\end{align*}
because of the i.i.d. property of $\{H_{i}\}$ and the definitions of $G_{k}.$

\begin{comment}
The first part in lemma's statement follows from the property  $\thetat^{(k)} \overset{d}{=} 
(H_{k}\bullet\dots \bullet H_{1})(\theta;1,\dots,1)$ and by induction w.r.t. $k$. Indeed, 
\begin{align*}
&(H_{k}\bullet\dots \bullet H_{1})(\theta;1,\dots,1) | (H_{k-1}\bullet\dots \bullet H_{1})(\theta;1,\dots,1)  
\\&\overset{d}{=} (H_{k}\bullet\dots \bullet H_{1})(\theta;1,\dots,1) | {\thetat}^{(k-1)} 
\\&\overset{d}{=} H_{k}({\thetah}^{(k-1)};1) | {\thetah}^{(k-1)} \sim  \P_{\thetah^{(k-1)}}= \mathcal{L}(\thetah^{(k)}).
\end{align*}
\end{comment}

\QED
\end{proof}

For functions $\varphi (t_1,\dots, t_k), (t_1,\dots, t_k)\in [0,1]^k,$
denote 
$$
\Delta_i \varphi (t_1,\dots, t_k) := \varphi(t_1,\dots, t_k)_{|t_i=1}- \varphi(t_1,\dots, t_k)_{|t_i=0}. 
$$
Note that, by generalized Newton-Leibnitz formula,
$$
\Delta_1\dots \Delta_k \varphi  = \int_0^1 \dots \int_0^1 
\frac{\partial^k \varphi(t_1,\dots, t_k)}{\partial t_1\dots \partial t_k}dt_1\dots dt_k,
$$
provided that $\varphi $ is a $C^k$-function.

Using \eqref{Newton_binom} and the second claim of Lemma \ref{represent_homotopy}, we get  
\begin{align}
\label{finite_differ_B^k}
&
\nonumber
({\mathcal B}^k f)(\theta)={\mathbb E}_{\theta}\sum_{j=0}^k (-1)^{k-j} {k\choose j} f(\thetah^{(j)})
\\
&
\nonumber
= 
 {\mathbb E}\sum_{j=0}^k (-1)^{k-j} \sum_{(t_1,\dots, t_k)\in \{0,1\}^k, \sum_{i=1}^k t_i=j}f(G_k(\theta;t_1,\dots, t_k))
\\
&
\nonumber
= 
{\mathbb E}\sum_{(t_1,\dots, t_k)\in \{0,1\}^k} (-1)^{k-\sum_{i=1}^k t_i}f(G_k(\theta;t_1,\dots, t_k))
\\
&
={\mathbb E}\Delta_1\dots \Delta_k f(G_k(\theta;t_1,\dots, t_k)). 
\end{align}
It remains to represent $\Delta_1\dots \Delta_k f(G_k(\theta;t_1,\dots, t_k))$ as 
\begin{align}
\label{finite_differ_int}
\Delta_1\dots \Delta_k f(G_k(\theta;t_1,\dots, t_k))= 
\int_0^1 \dots \int_0^1 
\frac{\partial^k f(G_k(\theta;t_1,\dots, t_k))}{\partial t_1\dots \partial t_k}dt_1\dots dt_k
\end{align}
to obtain from \eqref{finite_differ_B^k} and \eqref{finite_differ_int} the following 
proposition.

\begin{proposition}
\label{prop_B^k_AAA}
Let $U_1,\dots, U_k$ be i.i.d. random variables with uniform distribution in $[0,1]$
(independent of random homotopies $H_1,\dots, H_k$). 
Suppose that $\Theta$ is an open set, function $f$ is $k$ times continuously differentiable in $\Theta,$ 
random homotopy $H$ is $k$ times continuously differentiable in $\Theta\times [0,1]^d$ a.s.
and  
$$
{\mathbb E}\Bigl|\frac{\partial^k}{\partial t_1\dots \partial t_k}f(G_k(\theta;U_1,\dots, U_k))\Bigr|<\infty,
\theta \in \Theta.
$$
Then
\begin{align}
\label{B^k_int}
({\mathcal B}^k f)(\theta)= {\mathbb E} \frac{\partial^k}{\partial t_1\dots \partial t_k}f(G_k(\theta;U_1,\dots, U_k)), \theta\in \Theta.
\end{align}
\end{proposition}

In the cases when random homotopies exactly representing bootstrap chains lack necessary 
smoothness, it makes sense to use instead approximation of the bootstrap chain $\{\hat \theta^{(k)}: k\geq 0\}$ by superpositions of i.i.d. random homotopies. Suppose $H(\theta;t), \theta \in \Theta, t\in [0,1]$ is 
a stochastic process with values in $\Theta$ such that $H(\theta;0)=\theta, \theta \in \Theta$ and 
$H(\theta;1)\sim Q(\theta;\cdot)$ for some Markov kernel $Q$ on the space $\Theta.$ Let 
$\{\thetat^{(k)}: k\geq 0\}$ be the corresponding Markov chain defined in terms of superpositions 
of i.i.d. copies of $H$ (defined as in Lemma \eqref{represent_homotopy}). Let 
$$
P^{(k)}(\theta;A):= {\mathbb P}_{\theta}\{\thetah^{(k)}\in A\},\ \  Q^{(k)}(\theta;A):= 
{\mathbb P}_{\theta}\{\thetat^{(k)}\in A\}, \theta\in \Theta, A\in {\mathcal B}_{\Theta}
$$
be the corresponding $k$-step transition kernels. Our goal is to provide a bound on 
the total variation distance between the measures $P^{(k)}(\theta;\cdot), Q^{(k)}(\theta;\cdot)$
in terms of the corresponding total variation distance between $P(\theta;\cdot), Q(\theta;\cdot).$

Suppose $(\Theta,d)$ is a metric space with Borel $\sigma$-algebra ${\mathcal B}_{\Theta}$
and, for $A\in {\mathcal B}_{\Theta},$ let $A_{\delta}$ denote the $\delta$-neighborhood of $A.$
Then, the following simple proposition holds.

\begin{proposition}
\label{approx_homotopy}
For any set $A\in {\mathcal B}_{\Theta}$ and $\delta>0,$
\begin{align*}
&
\sup_{\theta\in A}\|P^{(k)}(\theta;\cdot)-Q^{(k)}(\theta;\cdot)\|_{TV}
\\
&
\leq k \sup_{\theta\in A_{k\delta}}\|P(\theta;\cdot)-Q(\theta;\cdot)\|_{TV}+
2k \sup_{\theta \in A_{k\delta}}{\mathbb P}_{\theta}\{d(\hat \theta;\theta)\geq \delta\}. 
\end{align*}
\end{proposition}

\begin{proof} 
Note that, for all $f:\Theta \mapsto {\mathbb R}$ with $\|f\|_{L_{\infty}(\Theta)}\leq 1,$
$$
\int_{\Theta} f(t)P^{(k)}(\theta;dt)=\int_{\Theta}\int_{\Theta} f(t)P^{(k-1)}(s;dt)P(\theta;ds)
$$
and 
$$
\int_{\Theta} f(t)Q^{(k)}(\theta;dt)=\int_{\Theta}\int_{\Theta} f(t)Q^{(k-1)}(s;dt)Q(\theta;ds).
$$
For arbitrary $A\in {\mathcal B}_{\Theta}$ and $\delta>0,$ we have 
\begin{align*}
&
\int_{\Theta} f(t)P^{(k)}(\theta;dt)-\int_{\Theta} f(t)Q^{(k)}(\theta;dt)
\\
&
=
\int_{\Theta}\int_{\Theta} f(t)(P^{(k-1)}(s;dt)-Q^{(k-1)}(s;dt))P(\theta;ds)
\\
&
+ 
\int_{\Theta}\int_{\Theta} f(t)Q^{(k-1)}(s;dt)(P(\theta;ds)-Q(\theta;ds))
\\
&
=
\int_{A_{\delta}}\int_{\Theta}f(t) (P^{(k-1)}(s;dt)-Q^{(k-1)}(s;dt))P(\theta;ds)
\\
&
+
\int_{A_{\delta}^c} \int_{\Theta}f(t)(P^{(k-1)}(s;dt)-Q^{(k-1)}(s;dt))P(\theta;ds)
\\
&
+ 
\int_{\Theta}\int_{\Theta} f(t)Q^{(k-1)}(s;dt)(P(\theta;ds)-Q(\theta;ds)).
\end{align*}
This implies 
\begin{align*}
&
\sup_{\theta\in A}\|P^{(k)}(\theta;\cdot)-Q^{(k)}(\theta;\cdot)\|_{TV}=
\sup_{\theta\in A}\sup_{\|f\|_{L_{\infty}(\Theta)}\leq 1}
\biggl|\int_{\Theta} f(t)P^{(k)}(\theta;dt)-\int_{\Theta} f(t)Q^{(k)}(\theta;dt)\biggr|
\\
&
\leq 
\sup_{\theta \in A_{\delta}} \|P^{(k-1)}(\theta;\cdot)-Q^{(k-1)}(\theta;\cdot)\|_{TV}
+ 2\sup_{\theta\in A}P(\theta; A_{\delta}^c) + \sup_{\theta\in A}\|P(\theta;\cdot)-Q(\theta;\cdot)\|_{TV}
\\
&
\leq 
\sup_{\theta \in A_{\delta}} \|P^{(k-1)}(\theta;\cdot)-Q^{(k-1)}(\theta;\cdot)\|_{TV}
+ 2\sup_{\theta\in A}{\mathbb P}_{\theta}\{d(\hat \theta;\theta)\geq \delta\} + 
\sup_{\theta\in A}\|P(\theta;\cdot)-Q(\theta;\cdot)\|_{TV}.
\end{align*}
Iterating the above bound $k$ times yields:
\begin{align*}
&
\sup_{\theta\in A}\|P^{(k)}(\theta;\cdot)-Q^{(k)}(\theta;\cdot)\|_{TV}
\\
&
\leq 
2k \sup_{\theta \in A_{k\delta}}{\mathbb P}_{\theta}\{d(\hat \theta;\theta)\geq \delta\}
+ k \sup_{\theta\in A_{k\delta}}\|P(\theta;\cdot)-Q(\theta;\cdot)\|_{TV}.  
\end{align*}
\qed
\end{proof}

Let 
$$
(\tilde {\mathcal T} f)(\theta):= \int_{\Theta} f(t)Q(\theta;dt), \theta \in \Theta
$$ 
and $\tilde {\mathcal B}:= \tilde {\mathcal T}-{\mathcal I}.$ 
Also denote 
$$
\tilde f_k(\theta):= \sum_{j=0}^k (-1)^j (\tilde {\mathcal B}^{j}f)(\theta), \theta \in \Theta.
$$
The following corollary is immediate.

\begin{corollary}
\label{corr_approx_homotopy}
For any set $A\in {\mathcal B}_{\Theta}$ and $\delta>0,$
\begin{align*}
&
\sup_{\theta\in A}|{\mathcal T}^k f(\theta)- \tilde {\mathcal T}^k f(\theta)|
\\
&
\leq k \|f\|_{L_{\infty}(\Theta)}
\Bigl[\sup_{\theta\in A_{k\delta}}\|P(\theta;\cdot)-Q(\theta;\cdot)\|_{TV}+
2\sup_{\theta \in A_{k\delta}}{\mathbb P}_{\theta}\{d(\hat \theta;\theta)\geq \delta\}\Bigr],
\end{align*}
\begin{align*}
&
\sup_{\theta\in A}|{\mathcal B}^k f(\theta)- \tilde {\mathcal B}^k f(\theta)|
\\
&
\leq k2^k \|f\|_{L_{\infty}(\Theta)}
\Bigl[\sup_{\theta\in A_{k\delta}}\|P(\theta;\cdot)-Q(\theta;\cdot)\|_{TV}+
2\sup_{\theta \in A_{k\delta}}{\mathbb P}_{\theta}\{d(\hat \theta;\theta)\geq \delta\}\Bigr]
\end{align*}
and 
\begin{align*}
&
\sup_{\theta\in A}|f_k(\theta)- \tilde f_k(\theta)|
\\
&
\leq k^2 2^k \|f\|_{L_{\infty}(\Theta)}
\Bigl[\sup_{\theta\in A_{k\delta}}\|P(\theta;\cdot)-Q(\theta;\cdot)\|_{TV}+
2\sup_{\theta \in A_{k\delta}}{\mathbb P}_{\theta}\{d(\hat \theta;\theta)\geq \delta\}\Bigr].
\end{align*}
\end{corollary}

\section{Bounds on H\"older norms of ${\mathcal B}^k f.$}
\label{sec:Holder_norms}

In this section, we prove a bound on the norm $\|{\mathcal B}\|_{C^s\mapsto C^{s-1}}$  of ${\mathcal B}$ as an operator from $C^{s}(\Theta)$ into $C^{s-1}(\Theta)$ for some $s=k+1+\rho,$ $k\geq 1, \rho\in (0,1].$
It will be assumed throughout the section that $\Theta\subset E$ is an open subset, that 
$H(\theta;t), \theta\in \Theta, t\in [0,1]$ is a random homotopy  between 
$\theta$ and $\thetah$ and that it is $k+1$ times continuously differentiable in $\Theta\times [0,1]$ with 
probability $1.$ The following bound will be proved. 

\begin{proposition}
\label{bound_on_norm_B}
Let $s=k+1+\rho$ for $k\geq 1$ and $\rho\in (0,1].$
Suppose that 
\begin{align*}
{\mathbb E}(\|H\|_{C^{s-1,0}(\Theta\times [0,1])}^{-}\vee 1)^{s-1}\|\dot H\|_{C^{s-1,0}(\Theta\times [0,1])} <\infty.
\end{align*}
Then
\begin{align*}
\|{\mathcal B}\|_{C^s\mapsto C^{s-1}} \leq 
4(k+1)^{k+2}
{\mathbb E}(\|H\|_{C^{s-1,0}(\Theta\times [0,1])}^{-}\vee 1)^{s-1}\|\dot H\|_{C^{s-1,0}(\Theta\times [0,1])} .
\end{align*}
In other words, for all $f\in C^s(\Theta),$
\begin{align*}
\|{\mathcal B} f\|_{C^{s-1}(\Theta)} \leq 4 (k+1)^{k+2}
{\mathbb E}(\|H\|_{C^{s-1,0}(\Theta\times [0,1])}^{-}\vee 1)^{s-1}
\|\dot H\|_{C^{s-1,0}(\Theta\times [0,1])} \|f\|_{C^s(\Theta)}.
\end{align*}
\end{proposition}

The next corollary is immediate.

\begin{corollary}
\label{cor:f_k_Hoeld}
Under the assumption of Proposition \ref{bound_on_norm_B},
for all $j=1,\dots, k,$
\begin{align*}
\|{\mathcal B}^j\|_{C^s\mapsto C^{s-j}} \leq 
\biggl(4 (k+1)^{k+2}
{\mathbb E}(\|H\|_{C^{s-1,0}(\Theta\times [0,1])}^{-}\vee 1)^{s-1}\|\dot H\|_{C^{s-1,0}(\Theta\times [0,1])}\biggr)^j,
\end{align*}
or, equivalently, for all $f\in C^s(\Theta),$
\begin{align*}
\|{\mathcal B}^j f\|_{C^{s-j}(\Theta)} \leq 
\biggl(4 (k+1)^{k+2}
{\mathbb E}(\|H\|_{C^{s-1,0}(\Theta\times [0,1])}^{-}\vee 1)^{s-1}\|\dot H\|_{C^{s-1,0}(\Theta\times [0,1])}\biggr)^j \|f\|_{C^{s}(\Theta)}.
\end{align*}
%As a consequence, under the assumption 
%\begin{align*}
%7 (k+1)^{k+2}
%{\mathbb E}(\|DH\|_{C^{s-2,0}(\Theta\times [0,1])}\vee 1)^{k+1}\|\dot H\|_{C^{s-1,0}(\Theta\times [0,1])}
%\leq 1/2,
%\end{align*}
%we have, for all $f\in C^s(\Theta),$
%\begin{align*}
%\|f_k\|_{C^{1+\rho}} \leq 2\|f\|_{C^s}.
%\end{align*}
\end{corollary}

\begin{proof}
Observe that 
\begin{align*}
\|{\mathcal B}^j\|_{C^s\mapsto C^{s-j}}
\leq \|{\mathcal B}\|_{C^{s}\mapsto C^{s-1}}\|{\mathcal B}\|_{C^{s-1}\mapsto C^{s-2}}\dots 
\|{\mathcal B}\|_{C^{s-j+1}\mapsto C^{s-j}}
\end{align*} 
and use the bound of Proposition \ref{bound_on_norm_B}.
%To prove the second bound, observe that 
%\begin{align*}
%&
%\|f_k\|_{C^{1+\rho}} = \biggl\|\sum_{j=0}^k (-1)^j {\mathcal B}^j f\biggr\|_{C^{1+\rho}}
%\leq \sum_{j=0}^k \|{\mathcal B}^j f\|_{C^{1+\rho}}
%\leq \sum_{j=0}^k \|{\mathcal B}^j f\|_{C^{s-j}}
%\\
%&
%\leq \sum_{j=0}^k \|{\mathcal B}^j\|_{C^{s}\mapsto C^{s-j}}\|f\|_{C^s}\leq \sum_{j=0}^k 2^{-j} \|f\|_{C^s}
%\leq 2\|f\|_{C^s}.
%
%\end{align*}

\qed
\end{proof}

The method of the proof of Proposition \ref{bound_on_norm_B} as well
as the proofs in Section \ref{sec:RepresentB^k} relies on Fa\`a di Bruno type 
calculus developed in the literature on combinatorics (see, e.g., \cite{Hardy}).

\begin{proof}
Suppose $f$ is $k+1$ times continuously differentiable in $\Theta.$ 
Under the assumptions on $H$ and $f,$ the function $f\circ H$ is $k+1$ times continuously differentiable in 
$\Theta\times [0,1]$ with probability $1.$ 
%Define 
%$$
%(T^t f)(\theta) := {\mathbb E} f(H(\theta, t)), \theta \in \Theta, t\in [0,1].
%$$
%It is easy to check (using dominated convergence) that $T^t f \in C^{k+1}(\Theta), t\in [0,1].$
Given $h_1,\dots, h_k\in E,$ let 
$$
\theta_{t_1,\dots, t_k} := \theta + \sum_{j=1}^k t_j h_j, t_j\in {\mathbb R}, j=1,\dots, k.
$$
For all $\theta \in \Theta,$ the function $(t, t_1, \dots, t_k)\mapsto (f\circ H)(\theta_{t_1,\dots, t_k};t)$
is $k+1$ times continuously differentiable in the set 
$[0,1]\times U_{\theta},$ where 
$$U_{\theta}:=\{(t_1,\dots, t_k)\in {\mathbb R}^k: \theta_{t_1,\dots, t_k}\in \Theta\}.$$ 
Note that $U_{\theta}$ is open and $(0,\dots, 0)\in U_{\theta}.$

Let $t_0:=t$ and $\bar T_k := \{t_0, t_1,\dots, t_k\}.$ For $T=\{t_{i_1}, \dots, t_{i_l}\}\subset \bar T_k$
and function $V(t_0, t_1,\dots, t_k),$
denote 
$$
\partial_{T} V := \frac{\partial^l V}{\partial t_{i_1} \dots \partial t_{i_l}}. 
$$

For a finite set $F,$ let ${\mathcal D}_F$ be the set of all partitions 
$\Delta:=(\Delta_1,\dots, \Delta_j)$ of set $F$ into disjoint nonempty subsets $\Delta_1, \dots , \Delta_j$ for some $j\geq 1.$ 
We set $|\Delta|:=j.$ The partitions that differ only by the order of their subsets will be considered identical. 
We will also use the notation ${\mathcal D}_{F,j}:=\{\Delta \in {\mathcal D}_F: |\Delta|=j\}.$ 
For a partition $\Delta:=(\Delta_1,\dots, \Delta_j)\in {\mathcal D}_{\bar T_k}$ 
and a function $V(t,t_1,\dots, t_k)$ (with values in a Banach space), 
denote 
$$
\partial_{\Delta} V := \partial_{\Delta_1}V \otimes \dots \otimes \partial_{\Delta_j} V.
$$ 

Observe that 
\begin{align*}
&
\nonumber
\frac{d}{dt} D^k f(H(\theta,t))[h_1\otimes \dots \otimes h_k]=
\partial_{\bar T_k}f(H(\theta_{t_1,\dots, t_k}, t))_{|t_1=\dots=t_k=0} 
\end{align*}

Our first goal is to derive a formula for the partial derivative $\partial_{\bar T_k}f(H(\theta_{t_1,\dots, t_k}, t)).$

\begin{lemma}
For all $\theta\in \Theta, (t,t_1,\dots, t_k) \in [0,1]\times U_{\theta},$ the following formula holds:
\begin{align}
\label{deriv_f_H}
&
\nonumber
\partial_{\bar T_k}f(H(\theta_{t_1,\dots, t_k}, t)) = 
\sum_{\Delta\in {\mathcal D}_{\bar T_k}} (D^{|\Delta|} f)(H(\theta_{t_1,\dots, t_k},t)) 
[\partial_{\Delta} H(\theta_{t_1,\dots, t_k},t)]
\\
&
=\sum_{j=1}^{k+1} \sum_{\Delta\in {\mathcal D}_{\bar T_k, j}} (D^{j} f)(H(\theta_{t_1,\dots, t_k},t)) [\partial_{\Delta_1} H(\theta_{t_1,\dots, t_k},t)\otimes \dots \otimes \partial_{\Delta_j} H(\theta_{t_1,\dots, t_k},t)].
\end{align}
\end{lemma}

\begin{proof}
In fact, we will prove by induction that, for all $l\leq k,$
\begin{align}
\label{partialone_A_A}
\partial_{\bar T_l}f(H(\theta_{t_1,\dots, t_k}, t))= 
\sum_{\Delta\in {\mathcal D}_{\bar T_l}} (D^{|\Delta|}f)(H(\theta_{t_1,\dots, t_k}, t))[\partial_{\Delta}H(\theta_{t_1,\dots, t_k}, t)]. 
\end{align}
Indeed, 
for $l=0,$ by the chain rule, we have
\begin{align*}
\partial_{\{t\}}f(H(\theta_{t_1,\dots, t_k}, t))= Df(H(\theta_{t_1,\dots, t_k}, t))[\partial_{\{t\}} H(\theta_{t_1,\dots, t_k}, t)],
\end{align*}
which is equivalent to \eqref{partialone_A_A}. Assuming that \eqref{partialone_A_A} holds 
for some $l<k$ and denoting $\bar H:=H(\theta_{t_1,\dots, t_k}, t),$ we have 
\begin{align*}
&
\partial_{\bar T_{l+1}}f(\bar H)= \partial_{\{t_{l+1}\}}\partial_{\bar T_l}f(\bar H)
=
\partial_{\{t_{l+1}\}}\sum_{\Delta\in {\mathcal D}_{\bar T_l}} (D^{|\Delta|}f)(\bar H)[\partial_{\Delta} \bar H]
\\
&
=\sum_{\Delta\in {\mathcal D}_{\bar T_l}} \partial_{\{t_{l+1}\}}(D^{|\Delta|}f)(\bar H)[\partial_{\Delta} \bar H]
\\
&
=
\sum_{\Delta\in {\mathcal D}_{\bar T_l}} (D^{|\Delta|+1}f)(\bar H)
[\partial_{\Delta} \bar H \otimes 
\partial_{\{t_{l+1}\}} \bar H]
+ \sum_{\Delta\in {\mathcal D}_{\bar T_l}} (D^{|\Delta|}f)(\bar H)
\biggl[\partial_{\{t_{l+1}\}}\partial_{\Delta}\bar H\biggr]
\\
&
=
\sum_{\Delta\in {\mathcal D}_{\bar T_l}} (D^{|\Delta|+1}f)(\bar H)
[\partial_{\Delta} \bar H\otimes 
\partial_{\{t_{l+1}\}} \bar H]
+ \sum_{\Delta\in {\mathcal D}_{\bar T_l}} (D^{|\Delta|}f)(\bar H)
\biggl[\partial_{\{t_{l+1}\}}\Motimes_{i=1}^{|\Delta|}\partial_{\Delta_i}\bar H\biggr]
\\
&
= 
\sum_{\Delta\in {\mathcal D}_{\bar T_l}} (D^{|\Delta|+1}f)(\bar H)[\partial_{\Delta} \bar H \otimes 
\partial_{\{t_{l+1}\}} \bar H]
+ \sum_{\Delta\in {\mathcal D}_{\bar T_l}} (D^{|\Delta|}f)(\bar H)\biggl[\sum_{i=1}^{|\Delta|}\partial_{\tilde \Delta^{(i)}} \bar H\biggr]
\\
&
= \sum_{\Delta\in {\mathcal D}_{\bar T_l}} (D^{|\Delta|+1}f)(\bar H)[\partial_{\tilde \Delta} \bar H]
+
\sum_{\Delta\in {\mathcal D}_{\bar T_l}} \sum_{i=1}^{|\Delta|}
(D^{|\Delta|}f)(\bar H)\Bigl[\partial_{\tilde \Delta^{(i)}} \bar H\Bigr],
\end{align*}
where, for a partition $\Delta=(\Delta_1,\dots, \Delta_j)\in {\mathcal D}_{\bar T_l},$
$$
\tilde \Delta :=(\Delta_1,\dots, \Delta_j, \{t_{l+1}\}) \in  {\mathcal D}_{\bar T_{l+1}} 
$$
and, for $1\leq i\leq j,$ 
$$
\tilde \Delta^{(i)} = (\Delta_1, \dots, \Delta_i \cup \{t_{l+1}\}, \dots, \Delta_j)\in {\mathcal D}_{\bar T_{l+1}}.
$$
Note that any partition in ${\mathcal D}_{\bar T_{l+1}}$ could be obtained (in a unique way) as 
an image of a partition $\Delta\in {\mathcal D}_{\bar T_{l}}$ under one of the mappings 
$\Delta\mapsto \tilde \Delta,$ $\Delta\mapsto \tilde \Delta_i, 1\leq i\leq j.$
This easily implies \eqref{partialone_A_A} for $l+1.$

\qed
\end{proof}

Recall that we identify partitions $\Delta=(\Delta_1, \dots, \Delta_j)\in {\mathcal D}_{\bar T_k, j}$
with different order of subsets. In formula \eqref{deriv_f_H}, an arbitrary order could be chosen. However, it will be convenient to assume in what follows that $t=t_0\in \Delta_1$ and that 
$|\Delta_2|\geq \dots \geq |\Delta_j|.$  
Let $T=\{t_{i_1}, \dots, t_{i_l}\}\subset \bar T_k$ and let $I:=\{i_1,\dots, i_l\}.$
Denote 
$$
h_I := h_{i_1}\otimes \dots \otimes h_{i_l}.
$$
For $I=\emptyset,$ set $h_I:=1.$
Then 
$$
\partial_T H(\theta_{t_1,\dots, t_k},t)= 
\begin{cases}
D^{l}H (\theta_{t_1,\dots, t_k},t)[h_I]\ \ \ \ \ \ \ \ \ {\rm if}\ 0\not\in I\\
D^{l-1}\dot H (\theta_{t_1,\dots, t_k},t)[h_{I\setminus \{0\}}]\ \ {\rm if}\ 0\in I.
\end{cases}
$$
Let $\bar I_k:= \{0,1,\dots, k\}.$
Denote, for $j=1,\dots, k+1,$ 
$$
{\mathcal K}_j := \Bigl\{(k_1,\dots , k_j): k_i\geq 1, i=1,\dots, j, k_2\geq \dots \geq k_j, \sum_{i=1}^j k_i=k+1\Bigr\}
$$
and, for $(k_1,\dots, k_j)\in {\mathcal K}_j,$
$$
{\mathcal D}_{\bar I_k, k_1,\dots, k_j}
:=\Bigl\{(I_1,\dots, I_j)\in {\mathcal D}_{\bar I_k,j}, |I_1|=k_1, \dots , |I_j|=k_j, I_1\ni 0\Bigr\}. 
$$
Clearly,
$$
{\rm card}({\mathcal D}_{\bar I_k, k_1,\dots, k_j})= \frac{k!}{(k_1-1)! k_2! \dots k_j!}.
$$
With these notations, it easily follows from \eqref{deriv_f_H} that 
\begin{align*}
&
\partial_{\bar T_k}f(H(\theta_{t_1,\dots, t_k}, t)) = 
\\
&
\sum_{j=1}^{k+1} \sum_{(k_1,\dots, k_j)\in {\mathcal K}_j} (D^j f)(H(\theta_{t_1,\dots, t_k},t))
\Bigl[(D^{k_1,\dots, k_j}H)(\theta_{t_1,\dots, t_k},t)[h_{k_1,\dots, k_j}]\Bigr], 
\end{align*}
where 
\begin{align*}
&
(D^{k_1,\dots, k_j}H)(\theta,t)
%\\
%&
= (D^{k_1-1}\dot H)(\theta,t)
\otimes (D^{k_2} H)(\theta,t) \otimes \dots \otimes (D^{k_j} H)(\theta,t)
\end{align*}
and 
$$
h_{k_1,\dots, k_j} := \sum_{(I_1,\dots, I_j)\in {\mathcal D}_{\bar I_k, k_1,\dots, k_j}} h_{I_1\setminus \{0\}} \otimes h_{I_2}
\otimes \dots \otimes h_{I_j}.
$$
Thus, for all $\theta\in \Theta,$ 
\begin{align}
\label{ddtD^k}
&
\nonumber
\frac{d}{dt} D^k f(H(\theta,t))[h_1\otimes \dots \otimes h_k]=
\partial_{\bar T_k}f(H(\theta_{t_1,\dots, t_k}, t))_{|t_1=\dots=t_k=0} 
\\
&
=\sum_{j=1}^{k+1} \sum_{(k_1,\dots, k_j)\in {\mathcal K}_j} (D^j f)(H(\theta,t))
\Bigl[(D^{k_1,\dots, k_j}H)(\theta,t)[h_{k_1,\dots, k_j}]\Bigr].
\end{align}

Assume that $f\in C^{s}(E)$ for some $s=k+1+\rho,$ $k=1,2,\dots, $ $\rho\in (0,1].$
For all $j=1,\dots, k+1,$ 
\begin{align*}
\|(D^{j} f)(H)\|_{L_{\infty}(\Theta\times [0,1])} \leq \|f\|_{C^s}
\end{align*}
and, for all $\theta, \theta'\in \Theta,$ 
\begin{align*}
&
\|(D^j f)(H(\theta,t))-(D^j f)(H(\theta',t))\|
\leq 
2\|f\|_{C^{s}} \|H(\theta,t)-H(\theta',t)\|^{\rho}
\\
&
\leq 
2\|f\|_{C^{s}} (\|H\|_{C^{1,0}}^{-})^{\rho}\|\theta-\theta'\|^{\rho} 
\leq 2\|f\|_{C^{s}} (\|H\|_{C^{s-1,0}}^{-}\vee 1)^{\rho}\|\theta-\theta'\|^{\rho},
\end{align*}
where we used bound \eqref{hoeld_lip}.
Note also that 
\begin{align*}
&
\|D^{k_1-1}\dot H\|_{L_{\infty}}
\leq \|\dot H\|_{C^{s-1,0}}, 1\leq k_1\leq k+1,
\\
&
\|D^{k_i}H\|_{L_{\infty}}\leq \|H\|_{C^{s-1,0}}^{-},
1\leq k_i\leq k, 2\leq i\leq j.
\end{align*}
Moreover, using again bound \eqref{hoeld_lip}, we get
\begin{align*}
&
\|D^{k_1-1} \dot H\|_{C^{\rho,0}} \leq 2 \|\dot H\|_{C^{s-1,0}}, 1\leq k_1\leq k+1
\\
&
\|D^{k_i} H\|_{C^{\rho,0}} \leq 2\|H\|_{C^{s-1,0}}^{-}, 1\leq k_i\leq k, 2\leq i\leq j. 
\end{align*}
It follows from the definition of $D^{k_1,\dots, k_j}H$ that 
\begin{align*}
&
\|D^{k_1,\dots, k_j}H\|_{L_{\infty}} \leq \|D^{k_1-1}\dot H\|_{L_{\infty}}\prod_{i=2}^j \|D^{k_i}H\|_{L_{\infty}}
\leq \|\dot H\|_{C^{s-1,0}} (\|H\|_{C^{s-1,0}}^{-})^{j-1}
\\
&
\leq (\|H\|_{C^{s-1,0}}^{-}\vee 1)^k \|\dot H\|_{C^{s-1,0}}, j=1,\dots, k+1.
\end{align*}
Moreover, for all $\theta, \theta'\in \Theta,$
\begin{align*}
&
\|(D^{k_1,\dots, k_j}H)(\theta,t)-(D^{k_1,\dots, k_j}H)(\theta',t)\|
\\
&
\leq 
\|(D^{k_1-1}\dot H)(\theta,t)-(D^{k_1-1}\dot H)(\theta',t)\|
\|(D^{k_2} H)(\theta,t)\| \dots \|(D^{k_j} H)(\theta,t)\| 
\\
&
+
\|(D^{k_1-1}\dot H)(\theta',t)\|
\|(D^{k_2} H)(\theta,t)-(D^{k_2} H)(\theta',t)\| \dots \|(D^{k_j} H)(\theta,t)\| + \dots 
\\
&
+
\|(D^{k_1-1}\dot H)(\theta',t)\|
\|(D^{k_2} H)(\theta',t)\| \dots \|(D^{k_j} H)(\theta,t)-(D^{k_j} H)(\theta',t)\| 
\\
&
\leq 
\Bigl(2(\|H\|_{C^{s-1,0}}^{-}\vee 1)^{j-1} \|\dot H\|_{C^{s-1,0}}
\\
&
\ \ \ \ \ \ 
+2(j-1)(\|H\|_{C^{s-1,0}}^{-}\vee 1)^{j-2} (\|H\|_{C^{s-1,0}}^{-}\vee 1)\|\dot H\|_{C^{s-1,0}}\Bigr) 
\|\theta -\theta'\|^{\rho}
\\
&
\leq 
2(k+1)(\|H\|_{C^{s-1,0}}^{-}\vee 1)^{k} \|\dot H\|_{C^{s-1,0}} \|\theta -\theta'\|^{\rho}.
\end{align*}
It easily follows from the bounds above that
\begin{align*}
& 
\biggl|(D^j f)(H(\theta,t))\Bigl[(D^{k_1,\dots, k_j}H)(\theta,t)
[h_{k_1,\dots, k_j}]\Bigr]-
(D^j f)(H(\theta',t))\Bigl[(D^{k_1,\dots, k_j}H)(\theta',t)
[h_{k_1,\dots, k_j}]\Bigr]\biggr|
\\
&
\leq 
2(k+2)\|f\|_{C^{s}} (\|H\|_{C^{s-1,0}}^{-}\vee 1)^{s-1}\|\dot H\|_{C^{s-1,0}}
\|h_{k_1,\dots, k_j}\|\|\theta -\theta'\|^{\rho}.
%\\
%&
%+2 \|f\|_{C^{s}} \|\dot H\|_{C^{s-1,0}} (\|H\|_{C^{s-1,0}}\vee 1)^{k+1} 
%\|h_{k_1,\dots, k_j}\|
%\|\theta-\theta'\|^{\rho}
\end{align*}
%Note that, in addition,
%\begin{align*}
%&
%\biggl(\frac{d}{dt} D^k f(H(\theta,t))-\frac{d}{dt} D^k f(H(\theta',t))\biggr) [h_1\otimes \dots \otimes h_k]=
%\\
%&
%=\sum_{j=1}^{k+1} \sum_{(k_1,\dots, k_j)\in {\mathcal K}_j} 
%\Bigl(
%(D^j f)(H(\theta,t))\Bigl[(D^{k_1,\dots, k_j}H)(\theta,t)
%[h_{k_1,\dots, k_j}]\Bigr]
%\\
%&
%\ \ \ \ \ \ \ \ \ \ \ \ \ \ \ \ \ \ \ -
%(D^j f)(H(\theta',t))\Bigl[(D^{k_1,\dots, k_j}H)(\theta',t)
%[h_{k_1,\dots, k_j}]\Bigr]
%\Bigr).
%\end{align*}
Also observe that 
$$
\sup_{\|h_1\|\leq 1, \dots, \|h_k\|\leq 1}\|h_{k_1,\dots, k_j}\| \leq 
{\rm card}({\mathcal D}_{\bar I_k, k_1,\dots, k_j})= \frac{k!}{(k_1-1)! k_2! \dots k_j!}.
$$
Using representation \eqref{ddtD^k}, it is now easy to show the following H\"older condition 
on the function $\theta\mapsto \frac{d}{dt} D^k f(H(\theta,t)):$
\begin{align*}
&
\biggl\|\frac{d}{dt} D^k f(H(\theta,t))-\frac{d}{dt} D^k f(H(\theta',t))\biggr\|
\\
&
\leq 
\sum_{j=1}^{k+1} \sum_{(k_1,\dots, k_j)\in {\mathcal K}_j} \frac{k!}{(k_1-1)! k_2! \dots k_j!}
2(k+2)\|f\|_{C^{s}}(\|H\|_{C^{s-1,0}}^{-}\vee 1)^{k+\rho} \|\dot H\|_{C^{s-1,0}} 
\|\theta-\theta'\|^{\rho}
%
%\biggl(2(k+1)\|(D^j f)(H(\theta,t))\| \|\dot H\|_{C^{s-1,0}} (\|H\|_{C^{s-1,0}}\vee 1)^{k} 
%\\
%&
%+2 \|f\|_{C^{s}} \|\dot H\|_{C^{s-1,0}} (\|H\|_{C^{s-1,0}}\vee 1)^{k+1} 
%\|\biggr)
%\|\theta-\theta'\|^{\rho}
\\
&
\leq 4 (k+1)^{k+2}\|f\|_{C^{s}}(\|H\|_{C^{s-1,0}}^{-}\vee 1)^{s-1}\|\dot H\|_{C^{s-1,0}} \|\theta-\theta'\|^{\rho}.
\end{align*}
Similarly, it could be shown that for all $j=0,\dots, k-1,$
\begin{align*}
&
\biggl\|\frac{d}{dt} D^j f(H(\theta,t))-\frac{d}{dt} D^j f(H(\theta',t))\biggr\|
\\
&
\leq 2 (k+1)^{k+2}\|f\|_{C^{s}}(\|H\|_{C^{s-1,0}}^{-}\vee 1)^{s-1}\|\dot H\|_{C^{s-1,0}} \|\theta-\theta'\|
\end{align*}
(in fact, a slightly better bound holds). Since also 
$$
\Bigl\|\frac{d}{dt}f(H(\theta,t))\Bigr\|\leq \|f\|_{C^1} \|\dot H\|_{L_{\infty}}
\leq \|f\|_{C^s} (\|H\|_{C^{s-1,0}}^{-}\vee 1)^{s-1}\|\dot H\|_{C^{s-1,0}}, 
$$
we can conclude that 
\begin{align*}
&
\biggl\|\frac{d}{dt} f(H(\cdot,t))\biggr\|_{C^{s-1}} \leq 
4(k+1)^{k+2}\|f\|_{C^{s}} (\|H\|_{C^{s-1,0}}^{-}\vee 1)^{s-1}\|\dot H\|_{C^{s-1,0}}.
\end{align*}
Also note that, for all $j\leq k,$ 
\begin{align*}
&
(D^j f)(H(\theta,1))-(D^j f)(H(\theta,0)) 
=\int_0^1\frac{d}{dt} D^j f(H(\theta,t))dt.
\end{align*}
As a consequence, we have 
\begin{align}
\label{Hoelder_f_H}
&
\nonumber
\|f(H(\cdot;1))- f(H(\cdot;0))\|_{C^{s-1}} 
= \biggl\|\int_0^1\frac{d}{dt} f(H(\cdot,t)) dt\biggr\|_{C^{s-1}}
\leq \int_{0}^1 \biggl\|\frac{d}{dt} f(H(\cdot,t))\biggr\|_{C^{s-1}} dt 
\\
&
\leq 4 (k+1)^{k+2}\|f\|_{C^{s}}(\|H\|_{C^{s-1,0}}^{-}\vee 1)^{s-1}\|\dot H\|_{C^{s-1,0}}.
\end{align}
%In view of \eqref{rho=0}, the same bound also holds for $\rho=0$ and $s=k+1.$
It follows that 
\begin{align}
\label{expect_Hoelder_f_H}
&
\nonumber
{\mathbb E}\|f(H(\cdot;1))- f(H(\cdot;0))\|_{C^{s-1}}
\\
&
\leq 4 (k+1)^{k+2}\|f\|_{C^{s}}{\mathbb E}(\|H\|_{C^{s-1,0}}^{-}\vee 1)^{s-1}\|\dot H\|_{C^{s-1,0}}<\infty.
\end{align}
Using Proposition \ref{diff_under_exp}, it is easy to justify differentiation 
under the expectation sign and to prove that the function 
\begin{align*}
{\mathcal B}f(\theta) = {\mathbb E}\Bigl[f(H(\theta;1))- f(H(\theta;0))\Bigr], \theta \in \Theta
\end{align*}
is $k$ times Fr\`echet continuously differentiable with derivatives 
\begin{align*}
(D^j {\mathcal B}f)(\theta)= {\mathbb E}\Bigl[D^j f(H(\theta;1))- D^j f(H(\theta;0))\Bigr], j\leq k.
\end{align*} 
Therefore, bound \eqref{expect_Hoelder_f_H}, easily implies that 
\begin{align*}
\|{\mathcal B}f\|_{C^{s-1}} \leq 4 (k+1)^{k+2}
{\mathbb E} (\|H\|_{C^{s-1,0}}^{-}\vee 1)^{s-1}\|\dot H\|_{C^{s-1,0}} \|f\|_{C^s},
\end{align*}
which completes the proof.

\qed
\end{proof}

\section{Representation formulas for ${\mathcal B}^k f.$}
\label{sec:RepresentB^k}

In this section, we discuss a different approach to controlling H\"older norms of ${\mathcal B}^k f.$
Our starting point will be formula \eqref{B^k_int} 
and we will develop certain representations of partial derivatives 
$\frac{\partial^k f(G_k(\theta;t_1,\dots, t_k))}{\partial t_1\dots \partial t_k}$
involved in this formula to obtain more explicit formulas for functions  
${\mathcal B}^k f.$ To calculate higher order derivatives of superpositions of random homotopies 
involved in the definition of function $G_k,$ we need to develop a version of Fa\`a di Bruno 
type calculus (already used in Section \ref{sec:Holder_norms}) in the context of our problem that relies on some combinatorial structures introduced below. It is assumed throughout the section that $\Theta\subset E$
is an open subset.

\subsection{Trees of multilinear forms.}

We will start with the notion of a {\it tree of multilinear forms} defined as follows.

\begin{definition} 
\normalfont
Let $\tau$ be a rooted tree. We will label the vertices of this tree with multilinear 
forms so that the following conditions hold: 
\begin{enumerate}

\item If $r$ is the root of $\tau$ and  ${\rm deg}(r)=l,$ then $r$ is labeled with an $l$-linear form.

\item If $v$ is a vertex of $\tau,$ $v\neq r$ and ${\rm deg}(v)=l+1,$ then $v$ is labeled with 
an $l$-linear form. 

\item Moreover, if $v$ is a vertex with children $v_1,\dots, v_l$ and $M_v, M_{v_1},\dots, M_{v_l}$
are the corresponding multilinear forms, then $M_v\in {\mathcal M}_l(E_1,\dots, E_l;F)$
for some Banach spaces $E_1,\dots, E_l, F$ and,  for 
all $j=1,\dots, l,$ the form $M_{v_j}$ takes values in $E_j.$   
\end{enumerate}
Such a labeled rooted tree will be called a {\it tree of multilinear forms}. 
\end{definition}

Note that the terminal vertices 
of $\tau$ have degree $1$ and they should be labeled with $0$-linear forms (vectors in some Banach spaces). For any non-terminal vertex $v$ of $\tau,$ let $\tau_v$ denote its subtree rooted at vertex $v.$
It is also a tree of multilinear forms.

Let now $\tau$ be a tree of multilinear forms with root $v$ and the corresponding label $M_v,$
where $M_v$ is an $l$-linear form with values in a Banach space $F.$ Let $v_1,\dots, v_l$
be the children of $v.$
Define recursively the following 
vector $\mu_{\tau}\in F:$ 
\begin{align*}
\mu_{\tau}:= M_v[\mu_{\tau_{v_1}}, \dots, \mu_{\tau_{v_{l}}}].
\end{align*}
If $v_{i, j}, i=1,\dots , l_j$ are the children of $v_i,$
we can write
\begin{align*}
\mu_{\tau} = M_v\Bigl[\mu_{\tau_{v_1}}\otimes \dots \otimes \mu_{\tau_{v_{l}}}\Bigr]
= M_v\Bigl[(M_{v_1}\otimes \dots \otimes M_{v_l})\Bigl[\otimes_{j=1}^l \otimes_{i=1}^{l_j} \mu_{\tau_{v_{i,j}}}\Bigr]\Bigr].
\end{align*}
Denote $M^{(0)}_{\tau}:= M_v,$ $v$ being the root of $\tau.$ 
For arbitrary $j=1,\dots , L,$ where $L$ is the height 
of $\tau,$ let $v_1^{(j)}, \dots, v_{m_j}^{(j)}$ be the vertices of depth $j$ (in other words, the $j$-th generation 
of descendants of $v$). For $j=0,$ we set $m_0:=1$ and $v_1^{(0)}:=v.$ 
Let 
$$
M^{(j)}_{\tau}:= M_{v_1^{(j)}}\otimes \dots \otimes M_{v_{m_j}^{(j)}}
$$ 
(it is assumed that $v_1^{(j)},\dots, v_{m_j}^{(j)}$ are arranged in the natural order determined by the tree). 
Then, by induction,
$$
\mu_{\tau} = M^{(0)}_{\tau}\circ M^{(1)}_{\tau}\circ \dots \circ M^{(L)}_{\tau}. 
$$
In what follows, the vector $\mu_{\tau}\in F$ will be called {\it the superposition of multilinear forms} 
over the tree $\tau.$ 

%Suppose that, for all $j=1,\dots, l,$ 
%the terminal vertices of subtree $\tau_{v_j}$ are labeled with 
%vectors $x_{i,j}\in E_{i,j}, i=1,\dots, m_i$ for some Banach spaces $E_{i,j}: i=1,\dots, m_j.$ 
%Then we can define recursively the following multilinear form $M_{\tau}:$
%\begin{align*}
%&
%M_{\tau}\in {\mathcal M}_{m_1+\dots+m_l} (E_{i,j}, i=1,\dots, m_j, j=1,\dots, l;F), 
%\\
%&
%M_{\tau}[x_{i,j}: i=1,\dots, m_j, j=1,\dots, l]
%= M_v(M_{\tau_{v_1}}[x_{1,1},\dots, x_{m_1,1}], \dots, 
%M_{\tau_{v_l}}[x_{1,l},\dots, x_{m_l,l}])
%\\
%&
%= M_v \Bigl[(M_{\tau_{v_1}}\otimes \dots \otimes M_{\tau_{v_l}})\Bigl[\otimes_{j=1}^l\otimes_{i=1}^{m_{j}} x_{i,j}\Bigr]\Bigr],
%\\
%&
%x_{i,j}\in E_{i,j}, i=1,m_j, j=1,\dots, l.
%\end{align*}
%In other words, we can write 
%$$
%M_{\tau}= M_v \circ (M_{\tau_{v_1}}\otimes \dots \otimes M_{\tau_{v_l}}).
%$$

It could be easily proved by induction that 
%$M_{\tau}$ is indeed multilinear and that 
the following property holds:  

\begin{proposition}
\label{bound_on_mu_tau}
Suppose $\tau$ is a tree of multilinear forms defined above.
%with $K$ vertices and with labels (multilinear forms)
%$M_1,\dots, M_K.$ 
Then 
%$$
%\|\mu_{\tau}\| \leq \|M_1\|\dots \|M_K\|.
%$$
$$
\|\mu_{\tau}\|\leq \prod_{j=0}^L \prod_{i=1}^{m_j} \|M_{v_i^{(j)}}\|.
$$
\end{proposition}

Let now $(T,d)$ be a metric space and consider a labeling of a rooted tree $\tau$ with multilinear 
forms that continuously depend on point $t\in T.$ 
More precisely, for $t\in T,$ suppose that each vertex $v$
of the tree is labeled with a multilinear form $M_v(t)\in {\mathcal M}_l(E_1,\dots, E_l;F)$ (the spaces 
$l$ and $E_1,\dots, E_l, F$ could depend on $v$) so that, for each $t\in T,$ $\tau$ is a tree of multilinear 
forms and, for each vertex $v,$ $T\ni t\mapsto M_v(t)\in {\mathcal M}_l(E_1,\dots, E_l;F)$ is a continuous 
function. Let $\mu_{\tau}(t), t\in T$ be the corresponding superposition of multilinear forms over the tree.

\begin{proposition}
\label{continuity_mult_lin_tree}
Under the above assumptions, $T\ni t\mapsto \mu_{\tau}(t)\in F$ is a continuous function and 
\begin{align}
\label{continuity_mult_lin_tree_1}
\|\mu_{\tau}\|_{L_{\infty}}\leq \prod_{j=0}^L \prod_{i=1}^{m_j} \|M_{v_i^{(j)}}\|_{L_{\infty}}.
\end{align}
Moreover,
\begin{align}
\label{continuity_mult_lin_tree_2}
\|\mu_{\tau}\|_{{\rm Lip}_d(T)}\leq (L+2) \prod_{j=0}^L (m_j+1)
\prod_{j=0}^L \prod_{i=1}^{m_j} \|M_{v_i^{(j)}}\|_{{\rm Lip}_d(T)}.
\end{align}
\end{proposition}

\begin{proof}
The proof of \eqref{continuity_mult_lin_tree_1} immediately follows from the bound 
of Proposition \ref{bound_on_mu_tau}. The proof of \eqref{continuity_mult_lin_tree_2}
follows from the following two bounds 
\begin{align*}
\|\mu_{\tau}\|_{{\rm Lip}_d(T)} \leq (L+2)\prod_{j=0}^L \|M_{\tau}^{(j)}\|_{{\rm Lip}_d(T)}
\end{align*}
and 
\begin{align*}
\|M_{\tau}^{(j)}\|_{{\rm Lip}_d(T)}\leq (m_j+1) 
\prod_{i=1}^{m_j} \|M_{v_i^{(j)}}\|_{{\rm Lip}_d(T)}
\end{align*}
that are based on rather elementary control of the corresponding Lipschitz norms.
\qed
\end{proof}

%{\bf Need to bound $\|\mu_{\tau}\|_{{\rm Lip}_d(T)}$ in terms of Lipschitz norms of the labels !!!}

%Denote $M^{(0)}_{\tau}:= M_v$ ($v$ being the root of $\tau$). For arbitrary $j=1,\dots , L,$ where $L$ is the height 
%of $\tau,$ let $v_1, \dots, v_{l_j}$ be the vertices of depth $j$ (in other words, the $j$ generation 
%of descendants of $v$) and let $M^{(j)}_{\tau}:= M_{v_1}\otimes \dots \otimes M_{v_{l_j}}$ 
%(it is assumed that $v_1,\dots, v_{l_j}$ are arranged in the natural order determined by the tree). 
%Then, by induction, 
%$$
%M_{\tau} = M^{(0)}_{\tau}\circ M^{(1)}_{\tau}\circ \dots \circ M^{(L)}_{\tau}. 
%$$
%In what follows, the multilinear form $M_{\tau}$ will be called {\it the superposition of multilinear forms} 
%over the tree $\tau.$ 

\subsection{Partition trees; $S$- and $M$-labelings.}

%Let $T_0:=\{s\}$ and $T_k:= \{s,t_1,\dots, t_k\}, k\geq 1.$ 

\begin{comment}

For a finite set $T,$ let ${\mathcal D}_T$ be the set of all partitions 
$\Delta:=(\Delta_1,\dots, \Delta_j)$ of set $T$ into disjoint nonempty subsets $\Delta_1, \dots , \Delta_j$ for some $j\geq 1.$ 
We set $|\Delta|:=j.$ The partitions that differ only by the order of their subsets will be considered identical. 
We will also use the notation ${\mathcal D}_{T,j}:=\{\Delta \in {\mathcal D}_T: |\Delta|=j\}.$ 

\end{comment}

Let $T_l:= \{t_1,\dots, t_l\}, l=1,\dots, k.$ For a partition $\Delta=(\Delta_1,\dots, \Delta_j)\in {\mathcal D}_{T_l},$ an arbitrary order of subsets $\Delta_i$ will be chosen and fixed.
%but the last element $t_l$ of the set $T_l$ will play a special role. To be specific, 
%let us always assume that, for $\Delta \in {\mathcal D}_{T_l},$ $t_l\in \Delta_j,$ $j=|\Delta|.$
Let $\Delta:=(\Delta_1,\dots, \Delta_j)\in {\mathcal D}_{T_l}.$ For each $i=1,\dots, j,$ 
let $\Delta_i'$ be a partition of the set $\Delta_i\cap T_{l-1},$ provided that this set 
is nonempty (it is empty only if $\Delta_i=\{t_l\}$). 
%and let 
%$\Delta_j'$ be a partition of the set $\Delta_j \cap T_{l-1}$ (provided that the last set is nonempty). 
This means that 
$$\Delta':=\bigcup_{i: \Delta_i\cap T_{l-1}\neq\emptyset}\Delta_i'$$ 
is a refinement of partition $\Delta \cap T_{l-1}.$
We will write, in this case, that $\Delta' \sqsupset \Delta.$ 
%Again, reordering the elements 
%of the partition if necessary, we assume that $t_{l-1}\in \Delta'_{j'}, j':=|\Delta'|.$ Thus, $\Delta'\in {\mathcal D}_{T_{l-1}}.$

We will now construct a special rooted tree (a partition tree of set $T_k$), which will be labeled with a set 
of Fr\`echet derivatives (multilinear forms).

\begin{definition}{\bf Partition tree and $S$-labeling.}
\label{def_partition_tree}
\normalfont
We will call a labeled rooted tree $\tau$ {\it a partition tree} of set $T_k$ iff the following properties hold:
\begin{enumerate}
\item The height of $\tau$ is $k.$
\item Each vertex $v$ of $\tau$ is labeled with a subset $S_v\subset T_k.$
\item For the root $r$ of $\tau,$ $S_r:=T_k.$
\item For $j=1,\dots, k,$ let $\Delta^{(j)}$ be the set of all the labels of 
the vertices of depth $k-j+1$ and suppose the following 
properties hold. 
\begin{enumerate}
\item $\Delta^{(j)}\in {\mathcal D}_{T_j}.$ 
%\item To be specific, assume that the last set in partition $\Delta^{(j)}$
%(the label of the last vertex of depth $k+1-j$) contains $t_{j}.$
\item $\Delta^{(j-1)}\sqsupset \Delta^{(j)},$ $j=2,\dots, k.$
\item For $j=1,\dots, k,$ the sets from partition $\Delta^{(j)}$ are assigned to the vertices of depth $k-j+1$
in the order of the vertices of tree $\tau$ (from ``left" to ``right"). 
%In particular, $t_j$ belongs to the set corresponding to the last vertex of depth $k-j+1.$
\item Let $v$ be a vertex of depth $k-j+1$ with label $S_v\in \Delta^{(j)}.$ 
Let $v_1,\dots, v_m$ be the children of $v$ in the tree $\tau.$
Then 
$\{S_{v_1},\dots, S_{v_m}\}$ is a partition of $S_v$ if 
%$v$ is not the last vertex of depth $k-j+1$ 
$t_j\not\in S_v,$ or 
a partition of $S_{v}\setminus \{t_j\}$ if 
% $v$  is the last vertex of depth $k-j+1$ 
$t_j\in S_v.$
\end{enumerate} 
\end{enumerate}
\end{definition}

We will call the labeling of partition tree $\tau$ in the last definition (with subsets of $T_k$) an $S$-labeling.
We will also need another labeling of $\tau,$ with multilinear forms (namely, Fr\`echet derivatives of function $f$ and random homotopies $H_j, j=1,\dots, k$), which will be called an $M$-labeling.
For the existence of the derivatives, we assume that $f$ is $k$ times Fr\`echet differentiable in $\Theta$
and that random homotopy $H(\theta;t)$ is $k-1$ times Fr\`echet differentiable with respect 
to $\theta\in \Theta$ a.s., it is also differentiable with respect to $t\in [0,1]$ with derivative 
$\dot H(\theta;t)
%:=\frac{d}{dt}H(\theta;t)
$  
and $\dot H(\theta;t)$ is $k-1$ times Fr\`echet differentiable with respect 
to $\theta\in \Theta$ a.s. Moreover, $D^m H(\theta;t), D^m \dot H(\theta;t)$ will denote the $m$-th order 
Fr\`echet derivatives with respect to $\theta.$

\begin{definition}{\bf $M$-labeling.}
\label{M-labeling}
\normalfont
\begin{enumerate}
\item Assign to the root $r$ of $\tau$ the form 
$$
M_{r}:=(D^{{\rm deg}(r)}f)(G_k(\theta, t_1,\dots, t_k)).
$$
\item For $l=1,\dots ,k,$ let $v$ be a vertex of depth $k-l+1$ such that $t_l\not\in S_v.$
Then 
$$
M_v:=(D^{{\rm deg}(v)-1} H_l)(G_{l-1}(\theta; t_1,\dots, t_{l-1}), t_l).
$$ 
\item For $l=1,\dots ,k,$ let $v$ be a vertex of degree $k-l+1$ such that $t_l\in S_v.$
Then 
$$
M_v:=(D^{{\rm deg}(v)-1} \dot H_l)(G_{l-1}(\theta;t_1,\dots, t_{l-1}), t_l)
$$
(if $v$ is a terminal vertex, then ${\rm deg}(v)=1$ and the label is just a vector $\dot H_l(G_{l-1}(\theta;t_1,\dots, t_{l-1}), t_l);$ in particular, this is the case for $l=1$).  
%\item For $l=0,$ the only vertex $v$ of $\tau$ is terminal with $S$-label equal to the set $\{s\}.$
%We set $M_v:= h.$
\end{enumerate}
\end{definition}

The $M$-labeling of $\tau$ defines a tree of multilinear forms.
Note that, in principle,  the same $M$-labelings could be generated by different $S$-labelings.
In what follows, ${\mathcal T}_k$ denotes the set of all partition trees of $T_k$
provided with both $S$- and $M$-labeling.

For $\tau \in {\mathcal T}_k,$ let $\partial_{\tau} f(G_k)$ be the superposition 
of multilinear forms ($M$-labels) over the tree $\tau.$ It will be called the derivative of $f(G_k)$
over the tree $\tau.$

\begin{example}
\normalfont
Here is an example of an $S$-labeled and $M$-labeled partition tree $\tau\in {\mathcal T}_4:$\\
\Tree [.$\{t_1,t_2,t_3,t_4\}$  [.$\{t_1,t_2\}$ [.$\{t_1,t_2\}$ [.$\{t_1\}$ $\{t_1\}$ ] $\{t_2\}$ ] ] [.$\{t_3,t_4\}$ $\{t_3\}$ ] ]
\hskip 8mm
\Tree [.$D^2f(G_4)$ [.$DH_4(G_3;t_4)$ [.$D^2H_3(G_2;t_3)$ [.$DH_2(G_1,t_2)$ $\dot H_1(\theta;t_1)$ ] 
$\dot H_2(G_1;t_2)$ ] ] [.$D\dot H_4(G_3;t_4)$ $\dot H_3(G_2;t_3)$ ] ]
\\
For this tree,
\begin{align*}
&
\partial_{\tau}f(G_4) = 
\\
&
D^2 f(G_4)[DH_4(G_3;t_4)[D^2H_3(G_2;t_3)[DH_2(G_1;t_2)[\dot H_1(\theta;t_1)], \dot H_2(G_1;t_2)]], D\dot H_4(G_3;t_4)[\dot H_3(G_2;t_3)]]. 
\end{align*}
\end{example}

%\Tree [.S This [.VP [.V is ] \qroof{a simple tree}.NP ] ]

%\Tree [.S when [.NP the cat ]    sleeps ]

\subsection{Representation of partial derivatives of $f(G_k).$}

We will now derive a formula for the partial derivative $\partial_{T_{k}}f(G_k),$ representing 
it as a sum of the derivatives $\partial_{\tau} f(G_k)$ over all labeled partition trees $\tau\in {\mathcal T}_k.$ 

%Given $\vec{m}:= (m_1,\dots, m_j),$ denote 
%$$
%(D^{\vec{m}}H)(\theta,t):= \bigotimes_{i=1}^{j-1}(D^{m_i}H)(\theta,t)\bigotimes (D^{m_j}\dot H)(\theta,t),
%\theta\in \Theta, t\in [0,1]. 
%$$
%We can view $(D^{\vec{m}}H)(\theta,t)$ as a multilinear form on the space $\bigotimes_{i=1}^{j} E^{\otimes m_{j}}\otimes {\mathbb R}$ (the presence of the factor ${\mathbb R}$ in the tensor product is due 
%to the differentiation with respect to $t\in {\mathbb R}$) with values in $E^{\otimes j}.$
%For $\Delta' \sqsupset \Delta,$ denote $\vec{m}(\Delta,\Delta'):= (|\Delta_1'|,\dots, |\Delta_{j-1}'|, |\Delta_j'|).$
%Note that in the case when $\Delta_j\cap T_{l-1}$ is empty and 
%$\Delta':= \Delta_1'\cup \dots\cup \Delta_{j-1}',$ we set in the above definition $|\Delta_j'|=0.$
%Let 
%$$
%(D_{\Delta,\Delta'} H)(\theta,t):= (D^{\vec{m}(\Delta,\Delta')}H)(\theta,t), \theta\in \Theta, t\in [0,1]. 
%$$
%Finally, for a partition tree $\delta=(\Delta^{(0)},\Delta^{(1)}, \dots, \Delta^{(k)})\in {\mathcal T}_k,$ denote 
%$$
%D_j^{\delta} := (D_{\Delta^{(j)}, \Delta^{(j-1)}} H_j)(\bar G_{j-1}(\theta; s, t_1,\dots, t_{j-1}), t_j).
%$$
%For any set of variables $t_1,\dots, t_j, \dots$ and its subset $T:= \{t_{i_1}, \dots, t_{i_m}\},$
%denote 
%$$
%\partial_T := \frac{\partial^m}{\partial t_{i_1}\dots \partial t_{i_m}}.
%$$

\begin{proposition}
\label{partial_deriv_no_s}
Suppose the following conditions hold:
\begin{enumerate}
\item $\Theta\subset E$ is an open subset; 
\item $f:\Theta \mapsto {\mathbb R}$ is 
$k$ times Fr\`echet continuously differentiable function in $\Theta;$
\item 
Random homotopy $H: \Theta \times [0,1]\mapsto \Theta$ 
is, with probability $1,$ $k$ times Fr\`echet continuously differentiable in $\Theta \times [0,1].$ 
%and differentiable with respect to $t\in [0,1]$ with derivative $\dot H(\theta;t).$
%Moreover, $\dot H(\theta, t)$ is $k-1$ times Fr\`echet differentiable with respect to 
%$\theta\in \Theta$ (with probability $1$).
\end{enumerate}
Then, the following representation holds:
\begin{align}
\label{main_represent_A}
\partial_{T_{k}}f(G_k) = \sum_{\tau\in {\mathcal T}_k}\partial_{\tau} f(G_k).
\end{align}
\end{proposition}

We will also need similar formulas for the Fr\`echet derivative of the function $\partial_{T_{k}}f(G_k(\theta;t_1,\dots, t_k))$ with respect to $\theta.$
To this end, define, for a fixed $h\in E,$ 
$$
\bar G_0(\theta,s):=\theta+sh, \ \ 
\bar G_1(\theta,s,t_1):=H_1(\theta+sh, t_1), \theta\in \Theta, s\in {\mathbb R}, t_1\in [0,1]
$$
and 
$$
\bar G_k (\theta,s,t_1,\dots, t_k):= H_k (\bar G_{k-1}(\theta,s,t_1,\dots, t_{k-1}), t_k), \theta\in \Theta, s\in {\mathbb R}, 
t_1,\dots, t_k\in [0,1].
$$
Note that $(\bar G_k)_{|s=0}=G_k.$

Assume that $f:\Theta \mapsto {\mathbb R}$ is 
$k+1$ times Fr\`echet continuously differentiable function in $\Theta$
and the random homotopy $H: \Theta \times [0,1]\mapsto \Theta$ 
is, with probability $1,$ $k+1$ times Fr\`echet continuously differentiable in $\Theta\times [0,1].$ 
Then, for all $\theta\in \Theta,$ the function $(s,t_1,\dots, t_k)\mapsto \bar G_k(\theta, s, t_1,\dots, t_k)$ is well defined and $k+1$ times continuously differentiable in the set 
$U_{\theta}\times [0,1]^k,$ where $U_{\theta}:=\{s\in {\mathbb R}: \theta+sh\in \Theta\}$ 
(with $0\in U_{\theta}$ being an interior point).

Let $\bar T_j:= \{s,t_1,\dots t_j\}, j=1,\dots, k,$ $\bar T_0:= \{s\}.$ 
%Denote 
%$\partial_{T_0}:=\frac{\partial }{\partial s}$ and 
%$$\partial_{T_k}:= \frac{\partial^{k+1}}{\partial t_1\dots \partial t_k \partial s}.$$
Similarly to Definition \ref{def_partition_tree}, we can define partition trees of set $\bar T_k$
along with their $S$-labelings. The only difference is that the trees are now of height $k+1$
and the labelings are defined by sets in partitions $\Delta^{(0)}\sqsupset  \dots \sqsupset  \Delta^{(k)}.$
We have to assume now that $f$ is $k+1$ times Fr\`echet differentiable 
and that random homotopy $H(\theta;t)$ is $k$ times Fr\`echet differentiable with respect 
to $\theta\in \Theta$ a.s., it is also differentiable with respect to $t\in [0,1]$ with derivative 
$\dot H(\theta;t)$  and $\dot H(\theta;t)$ is $k$ times Fr\`echet differentiable with respect 
to $\theta\in \Theta$ a.s.

We also need to modify slightly Definition \ref{M-labeling}: 

\begin{definition}{\bf $M$-labeling (modified).}
\label{M-labeling_modified}
\normalfont
\begin{enumerate}
\item The label of the root $r$ of $\tau$ and, for $l=1,\dots, k,$ the labels
of the vertices of depth $k-l+1$  
are defined as in steps 1-3 of Definition \ref{M-labeling} (with $\bar G_k$ instead of $G_k$).   
 \item For $l=0,$ the only vertex $v$ of depth $k+1$ is terminal with $S$-label equal to the set $\{s\}$
and $M$-label $M_v:=h.$
\end{enumerate}
\end{definition}

Let $\bar {\mathcal T}_k$ be the set of all partition trees of $\bar T_k$ with both $S$- and $M$-labelings. 
For $\tau \in \bar {\mathcal T}_k,$ let $\partial_{\tau} f(\bar G_k)$ be the superposition 
of multilinear forms over the tree $\tau.$ As before, it is called the derivative of $f(\bar G_k)$
over the tree $\tau.$ It is easy to see that $\partial_{\tau} f(\bar G_k(\theta;s;t_1,\dots, t_k))$ can be written as 
$$
\partial_{\tau} f(\bar G_k(\theta;s;t_1,\dots, t_k)) = D_{\tau} f(\bar G_k(\theta;s;t_1,\dots, t_k))[h],
$$
where $D_{\tau} f(\bar G_k(\theta;s;t_1,\dots, t_k))$ is a linear functional on $E$ 
(recall that the only vertex of depth $k+1$ of tree $\tau$ is labeled with $h$).

With these definitions and notations, the following proposition holds.

\begin{proposition}
\label{partial_deriv_s}
Suppose that 
\begin{enumerate}
\item $\Theta\subset E$ is an open subset; 
\item $f:\Theta \mapsto {\mathbb R}$ is 
$k+1$ times Fr\`echet continuously differentiable function in $\Theta;$
\item 
Random homotopy $H: \Theta \times [0,1]\mapsto \Theta$ 
is, with probability $1,$ $k+1$ times Fr\`echet continuously differentiable in $\Theta\times [0,1].$ 
%and differentiable with respect to $t\in [0,1]$ with derivative $\dot H(\theta;t).$
%Moreover, $\dot H(\theta, t)$ is $k$ times Fr\`echet differentiable with respect to 
%$\theta\in \Theta$ (with probability $1$).
\end{enumerate}
Then, for all $\theta\in \Theta, s\in U_{\theta}, (t_1,\dots, t_k)\in [0,1]^k,$ the following representation holds:
\begin{align}
\label{main_represent}
\partial_{\bar T_k}f(\bar G_k) = \sum_{\tau\in \bar {\mathcal T}_k} \partial_{\tau} f(\bar G_k)
=\sum_{\tau\in \bar {\mathcal T}_k}
D_{\tau} f(\bar G_k(\theta;s,U_1,\dots, U_k))[h].
\end{align}
\end{proposition}

%\begin{proposition}
%The following representation holds:
%\begin{align}
%\label{main_represent}
%\partial_{T_k}f(\bar G_k) = \sum_{\delta\in {\mathcal T}_k} (D^{|\Delta^{(k)}|} f)(\bar G_k)[D_k^{\delta}[D_{k-1}^{\delta}[\dots [D_2^{\delta}[D_1^{\delta}[h]]\dots ]]]]. 
%\end{align}
%\end{proposition}

We provide below the proof of Proposition \ref{partial_deriv_s} (the proof of Proposition \ref{partial_deriv_no_s} is almost identical).

\begin{proof}
For a partition $\Delta:=(\Delta_1,\dots, \Delta_j)\in \bar {\mathcal T}_k$ 
and a function $F(s,t_1,\dots, t_k)$ with values in $E,$
denote 
$$
\partial_{\Delta} F := \partial_{\Delta_1}F \otimes \dots \otimes \partial_{\Delta_j} F.
$$ 
First note that 
\begin{align}
\label{partialone}
\partial_{\bar T_k}f(\bar G_k)= \sum_{\Delta\in {\mathcal D}_{\bar T_k}} (D^{|\Delta|}f)(\bar G_k)[\partial_{\Delta} \bar G_k]. 
\end{align}
The proof of this formula is identical to the proof of \eqref{deriv_f_H}.
Let $\Delta=(\Delta_1,\dots, \Delta_j)$ be a partition of $\bar T_k$ with $|\Delta|=j.$ 
%and $t_k\in \Delta_j.$ 
Then, for all $i=1,\dots, j$ similarly to \eqref{partialone},
\begin{align}
\partial_{\Delta_i} \bar G_k = \sum_{\Delta_i'\in {\mathcal D}_{\Delta_i}} (D^{|\Delta_i'|} H_k)(\tilde G_{k-1})[\partial_{\Delta_i'}\bar G_{k-1}]
\end{align}
provided that $t_k\not\in \Delta_i,$ and 
\begin{align}
\label{partialbarG_k}
\partial_{\Delta_i} \bar G_k = \sum_{\Delta_j'\in {\mathcal D}_{\Delta_i\setminus \{t_k\}}} (D^{|\Delta_i'|} \dot H_k)(\tilde G_{k-1})[\partial_{\Delta_i'}\bar G_{k-1}],
\end{align}
provided that $t_k\in \Delta_i.$
Here $\tilde G_{k-1}(\theta, s, t_1,\dots, t_k)= (\bar G_{k-1}(\theta, s, t_1,\dots, t_{k-1}), t_k).$
%These formulas could be proved similarly to \eqref{partialone}. 
%The factor $1$ in the expression 
%$\partial_{\Delta_j'}\bar G_{k-1}\otimes 1$ is due to the fact that $H_k$ is differentiated both 
%with respect to its first and second variables (so, the derivative $(D^{|\Delta_j'|} \dot H_k)(\tilde G_{k-1})$
%is viewed as a multilinear form of $|\Delta_j'|$ variables in $E$ and one variable in ${\mathbb R}$
%(or, in other words, a linear functional on $E^{\otimes |\Delta_j'|}\otimes {\mathbb R}$). 
Note that in a special case 
when $\Delta_{i}=\{t_k\},$ the set $\Delta_i\cap T_{k-1}=\emptyset$ and $|\Delta_i'|=0.$ 
In this case, 
$$(D^{|\Delta_i'|} \dot H_k)(\tilde G_{k-1})=\dot H_k (\tilde G_{k-1})= \dot H_k(\tilde G_{k-1})[\ ]$$
is a $0$-linear form (a vector in $E$) and $\partial_{\Delta_i'}\bar G_{k-1}$ is an ``empty variable".
%is just a linear form on ${\mathbb R}$ and $\partial_{\Delta_j'}\bar G_{k-1}\otimes 1$ reduces
%to just $1:$ 
%$$
%(D^{|\Delta_j'|} \dot H_k)(\tilde G_{k-1})[\partial_{\Delta_j'}\bar G_{k-1}\otimes 1]=
%\dot H_k(\tilde G_{k-1})[1]= \dot H_k(\tilde G_{k-1}).
%$$
%%%% and formula \eqref{partialbarG_k} still holds. 

Assume now that $t_k\in \Delta_r$ for some $r=1,\dots, j.$ Denote by $\Pi_r$ the following 
permutation operator (acting on tensor products of $j$ elements of arbitrary linear spaces):
$$
\Pi_r \Bigl(\Motimes_{i=1,\dots, j, i\neq r} A_i \Motimes A_r\Bigr):= 
A_1\Motimes \dots \Motimes A_{r-1}\Motimes A_r \Motimes A_{r+1}\Motimes \dots \Motimes A_j.
$$
It immediately follows that 
\begin{align}
\label{partialtwo}
&
\nonumber
\partial_{\Delta} \bar G_k = \partial_{\Delta_1}\bar G_k \Motimes \dots \Motimes \partial_{\Delta_j} \bar G_k
\\
&
\nonumber
= \Pi_r \biggl(\Motimes_{i\neq r} \sum_{\Delta_i'\in {\mathcal D}_{\Delta_i}} (D^{|\Delta_i'|} H_k)(\tilde G_{k-1})[\partial_{\Delta_i'}\bar G_{k-1}] \Motimes \sum_{\Delta_r'\in {\mathcal D}_{\Delta_r}\setminus \{t_k\}} (D^{|\Delta_r'|} \dot H_k)(\tilde G_{k-1})[\partial_{\Delta_r'}\bar G_{k-1}]\biggr)
\\
&
\nonumber
=\sum_{\Delta'\sqsupset \Delta}  \Pi_r \biggl(\Motimes_{i\neq r} (D^{|\Delta_i'|} H_k)(\tilde G_{k-1})[\partial_{\Delta_i'}\bar G_{k-1}] \Motimes (D^{|\Delta_r'|} \dot H_k)(\tilde G_{k-1})[\partial_{\Delta_r'}\bar G_{k-1}]\biggr)
\\
&
= 
\sum_{\Delta'\sqsupset \Delta}  \Pi_r\Bigl(\Motimes_{i\neq r} D^{|\Delta_i'|} H_k)(\tilde G_{k-1})\Motimes 
(D^{|\Delta_r'|} \dot H_k)(\tilde G_{k-1})\Bigr)\Bigl[\Pi_r \Bigl(\Motimes_{i\neq r} \partial_{\Delta_i'}\bar G_{k-1}\Motimes  \partial_{\Delta_r'}\bar G_{k-1}\Bigr)\Bigr].
%\\
%&
%=\sum_{\Delta' \sqsupset \Delta} (D_{\Delta,\Delta'} H_k)(\tilde G_{k-1})[\partial_{\Delta'} \bar G_{k-1}].
\end{align}
For $l=0,\dots, k+1,$ let $\tau^{(l)}$ be the subtree of $\tau$ that includes all its vertices of depth $\leq l$ 
and all its edges between these vertices (and also has the same $S$- and $M$-labelings).
Clearly, $\tau^{(k+1)}= \tau.$
Let $\bar {\mathcal T}_k^{(l)}:= \{\tau^{(l)}: \tau \in \bar {\mathcal T}_k\},$
$l\leq k+1.$
Recall the definition of the forms $M_{\tau}^{(i)}, i=0,\dots, k+1.$
It is easy to see that 
$$
M_{\tau^{(l)}}^{(i)}=M_{\tau}^{(i)}, i=0,\dots, l, l=0,\dots, k+1.
$$
Using now the definition of labeled partition tree $\tau,$
it is easy to deduce from \eqref{partialone} and \eqref{partialtwo} that 
\begin{align*}
\partial_{\bar T_k}f(\bar G_k)= 
\sum_{\tau^{(1)}\in \bar {\mathcal T}_k^{(1)}}
(M^{(0)}_{\tau^{(1)}}\circ M^{(1)}_{\tau^{(1)}})[\partial_{\Delta^{(k-1)}} \bar G_{k-1}].
\end{align*}
%\begin{align*}
%\partial_{T_k}f(\bar G_k)= \sum_{\Delta\in {\mathcal D}_{T_k}} \sum_{\Delta' \sqsupset \Delta}
%(D^{|\Delta|}f)(\bar G_k)[(D_{\Delta,\Delta'} H_k)(\tilde G_{k-1})[\partial_{\Delta'} \bar G_{k-1}]],
%\end{align*}
It could be now shown by induction that, for all $l=1,\dots, k,$ 
\begin{align*}
\partial_{\bar T_k}f(\bar G_k)= 
\sum_{\tau^{(l)}\in \bar {\mathcal T}_k^{(l)}}
(M^{(0)}_{\tau^{(l)}}\circ M^{(1)}_{\tau^{(l)}} \circ \dots \circ M^{(l)}_{\tau^{(l)}})
[\partial_{\Delta^{(k-l)}} \bar G_{k-l}].
\end{align*}
Note that for $k=l,$ 
$$
\partial_{\Delta^{(k-l)}} \bar G_{k-l}=\partial_{\{s\}} \bar G_0=h = M_{\tau^{(k+1)}}^{(k+1)}
=M_{\tau}^{(k+1)}.
$$
Hence, we can conclude that 
\begin{align*}
\partial_{\bar T_k}f(\bar G_k)= \sum_{\tau\in \bar {\mathcal T}_k} (M^{(0)}_{\tau} \circ \dots \circ M^{(k+1)}_{\tau})= 
\sum_{\tau\in \bar {\mathcal T}_k} \mu_{\tau} = \sum_{\tau \in \bar {\mathcal T}_k}\partial_{\tau}f(\bar G_k),
\end{align*}
completing the proof. 
\qed
\end{proof}

%\begin{remark}
%\normalfont
%We might also need some modifications of representation \eqref{main_represent_A}. 
%In particular, let $\bar T_k:= \{t_1,\dots, t_k\}$ 
%and let $\bar {\mathcal T}_k$ denotes the set of all partition trees $\tau$ 
%of the set $\bar T_k.$ The only difference with Definition \ref{M-labeling}
%is that variable $s$ is not now included in the sets $\bar T_k.$ 
%This is needed to write down a formula for partial derivative $\partial_{\bar T_k} f(G_k),$
%where 
%$$
%G_k(\theta;t_1,\dots, t_k)= (H_k\bullet \dots \bullet H_1)(\theta;t_1,\dots, t_k)=
%\bar G_k (\theta, 0, t_1,\dots, t_k).
%$$
%The height of the partition tree $\tau$ is now $k$ instead of $k+1$ and, for the only vertex 
%$v$ of depth $k,$ its $S$-label is the set $\{t_1\}$ and its $M$-label is 
%the vector $\dot H (\theta;t_1).$ As before, we define the derivative $\partial_{\tau}f(G_k)$ 
%over the tree $\tau$ as the superposition of $M$-labels over the tree. This yields 
%a representation for $\partial_{\bar T_k} f(G_k):$
%\end{remark}

\subsection{Representation of ${\mathcal B}^k f.$}

We will combine \eqref{B^k_int} with Proposition \ref{main_represent_A} to get a representation 
formula for $({\mathcal B}^k f)(\theta).$ Recall that $U_1,\dots, U_k$ are i.i.d. uniformly 
distributed in $[0,1]$ r.v. (independent of $H_1,\dots, H_k$).

\begin{comment}
It will be convenient to use the following norms for functions $F:\Theta \times [0,1]\mapsto \Theta$
such that $F(\cdot;t)\in C^s(\Theta), t\in [0,1]$ for some $s>0:$ 
$$
\|F\|_{C^{s,0}(\Theta\times [0,1])}:= \sup_{t\in [0,1]}\|F(\cdot;t)\|_{C^{s}(\Theta)}.
$$
Denote 
$$
C^{s,0}(\Theta\times [0,1]):=\Bigl\{F:\Theta \times [0,1]\mapsto \Theta: \|F\|_{C^{s,0}(\Theta\times [0,1])}<+\infty\Bigr\}.
$$
\end{comment}

\begin{proposition}
\label{prop_B^k}
Suppose, for some $k\geq 1,$ $f$ is $k$ times Fr\`echet continuously differentiable in $\Theta$ and 
$H$ is $k$ times Fr\`echet continuously differentiable in $\Theta\times [0,1].$
Moreover, suppose that 
\begin{align}
\label{moment_of_H}
{\mathbb E}\Bigl(\|H\|_{C^{k-1,0}(\Theta\times [0,1])}^{-}\vee 1\Bigr)^{k-1}
\|\dot H\|_{C^{k-1,0}(\Theta\times [0,1])}<\infty.
\end{align}
Then 
\begin{align}
\label{formula_B^k}
({\mathcal B}^k f)(\theta)=\sum_{\tau\in {\mathcal T}_k}{\mathbb E}\partial_{\tau} f(G_k(\theta;U_1,\dots, U_k)), \theta\in \Theta
\end{align}
and 
\begin{align}
\label{bound_on_B^k}
\|{\mathcal B}^k f\|_{L_{\infty}(\Theta)} \leq k^{k^2}\|f\|_{C^{k}(\Theta)}
\biggl({\mathbb E}\Bigl(\|H\|_{C^{k-1,0}(\Theta\times [0,1])}^{-}\vee 1\Bigr)^{k-1}
\|\dot H\|_{C^{k-1,0}(\Theta\times [0,1])}\biggr)^k.
\end{align}
\end{proposition}

\begin{remark}
\normalfont 
For $k=1,$ the assumption becomes 
${\mathbb E}\|\dot H\|_{L_{\infty}(\Theta\times [0,1])}<\infty$
and the bound becomes 
$
\|{\mathcal B}^k f\|_{L_{\infty}(\Theta)} \leq \|f\|_{C^1(\Theta)}
{\mathbb E}\|\dot H\|_{L_{\infty}(\Theta\times [0,1])}^k
$
(replacing $\|H\|_{C^{k-1,0}(\Theta\times [0,1])}$ by $1$).
\end{remark}

\begin{proof}
It easily follows from the definition of $M$-labeling and the bound of Proposition \ref{bound_on_mu_tau}
that, for all $\tau\in {\mathcal T}_k,$
\begin{align}
\label{bound_on_partial_tau}
&
\nonumber
|\partial_{\tau} f(G_k(\theta;U_1,\dots, U_k))| 
\\
&
\leq \|f\|_{C^{k}(\Theta)} 
\prod_{l=1}^{k}\Bigl(\|H_l\|_{C^{k-1,0}(\Theta\times [0,1])}^{-}\vee 1\Bigr)^{l-1}
\|\dot H_l\|_{C^{k-1,0}(\Theta\times [0,1])}.
\end{align}
To see this note that $M$-labelings of any tree $\tau\in {\mathcal T}_k$ involve 
the Fr\`echet derivatives of $f$ of order at most $k$ and the Fr\`echet derivatives of 
$H$ and $\dot H$ of order at most $k-1.$ Moreover, for $l=1,\dots, k,$ the number of vertices of 
depth $k-l+1$ labeled with Fr\`echet derivatives of $H$ is, at most, $l-1$ while the number 
of vertices of depth $k-l+1$ labeled by a Fr\`echet derivative of $\dot H$ is exactly $1.$   
Since $\|H_l\|_{C^{k-1,0}(\Theta\times [0,1])}^{-}$ could be smaller than $1,$ we need to take 
a maximum with $1$ to get a valid bound in the case when the number of vertices labeled with Fr\`echet derivatives of $H$ is smaller than $l-1.$

Recalling that $H_1,\dots, H_k$ are i.i.d. copies of $H,$ condition \eqref{moment_of_H} and bound 
\eqref{bound_on_partial_tau} imply 
integrability of $\partial_{\tau} f(G_k(\theta;U_1,\dots, U_k))$ and formula \eqref{formula_B^k}
now follows from propositions \ref{prop_B^k_AAA}  and \ref{partial_deriv_no_s}.
Moreover,
\begin{align*}
%\label{bound_on_partial_tau}
&
\nonumber
{\mathbb E}|\partial_{\tau} f(G_k(\theta;U_1,\dots, U_k))| 
\\
&
\leq \|f\|_{C^{k}(\Theta)} 
{\mathbb E}\prod_{l=1}^{k}\Bigl(\|H_l\|_{C^{k-1,0}(\Theta\times [0,1])}^{-}\vee 1\Bigr)^{l-1}
\|\dot H_l\|_{C^{k-1,0}(\Theta\times [0,1])}
\\
&
\leq 
\|f\|_{C^{k}(\Theta)} 
\biggl({\mathbb E}\Bigl(\|H\|_{C^{k-1,0}(\Theta\times [0,1])}^{-}\vee 1\Bigr)^{k-1}
\|\dot H\|_{C^{k-1,0}(\Theta\times [0,1])}\biggr)^k.
\end{align*}
Since ${\rm card}({\mathcal T}_k)\leq k^{k^2},$ 
we get bound \eqref{bound_on_B^k}.
\qed
\end{proof}

Using Proposition \ref{partial_deriv_s}, we can similarly get a representation as well as bounds for Fr\`echet 
derivative of ${\mathcal B}^k f(\theta).$

\begin{proposition}
\label{prop_B^k_deriv}
Suppose, for some $k\geq 1,$ $f$ is $k+1$ times Fr\`echet continuously differentiable in $\Theta$ and 
$H$ is $k+1$ times Fr\`echet continuously differentiable in $\Theta\times [0,1].$
Moreover, suppose that 
\begin{align}
\label{moment_of_H_deriv}
{\mathbb E}\Bigl(\|H\|_{C^{k,0}(\Theta\times [0,1])}^{-}\vee 1\Bigr)^{k}
\|\dot H\|_{C^{k,0}(\Theta\times [0,1])}<\infty.
\end{align}
Then $\theta\ni \Theta\mapsto ({\mathcal B}^k f)(\theta)$ is Fr\`echet continuously differentiable in $\Theta$
with derivative
\begin{align}
\label{formula_B^k_deriv}
({\mathcal B}^k f)^{\prime}(\theta)=\sum_{\tau\in \bar {\mathcal T}_k}
{\mathbb E}D_{\tau} f(\bar G_k(\theta;0,U_1,\dots, U_k)), \theta\in \Theta
\end{align}
and
\begin{align}
\label{bound_on_B^k_deriv}
\|(B^{k} f)^{\prime}\|_{L_{\infty}(\Theta)} \leq (k+1)^{(k+1)^2}\|f\|_{C^{k+1}(\Theta)}
\biggl({\mathbb E}\Bigl(\|H\|_{C^{k,0}(\Theta\times [0,1])}^{-}\vee 1\Bigr)^{k}
\|\dot H\|_{C^{k,0}(\Theta\times [0,1])}\biggr)^k.
\end{align}
If, in addition $f\in C^s(\Theta)$ for some $s=k+1+\rho, \rho\in (0,1]$ 
and 
\begin{align}
\label{moment_of_H_deriv_+}
{\mathbb E}\Bigl(\|H\|_{C^{s-1,0}(\Theta\times [0,1])}^{-}\vee 1\Bigr)^{s-1}
\|\dot H\|_{C^{s-1,0}(\Theta\times [0,1])}<\infty,
\end{align} 
then 
\begin{align}
\label{bound_on_B^k_deriv_Lip}
&
\|B^{k} f\|_{C^{1+\rho}(\Theta)} 
\nonumber
\\
&
\leq 3 (k+2)^{2(k+2)^2}\|f\|_{C^{s}(\Theta)}
\biggl({\mathbb E}\Bigl(\|H\|_{C^{s-1,0}(\Theta\times [0,1])}^{-}\vee 1\Bigr)^{s-1}
\|\dot H\|_{C^{s-1,0}(\Theta\times [0,1])}\biggr)^k.
\end{align}
\end{proposition}

\begin{proof}
Note that, by Proposition \ref{partial_deriv_s}, 
\begin{align*}
&
D\partial_{T_k} f(G_k(\theta;U_1,\dots, U_k))[h]= \partial_{{\bar T}_k} f(\bar G_k(\theta;0;U_1,\dots, U_k))
\\
&
=\sum_{\tau\in \bar {\mathcal T}_k}
D_{\tau} f(\bar G_k(\theta;0,U_1,\dots, U_k))[h].
\end{align*}
Similarly to \eqref{bound_on_partial_tau},
\begin{align}
\label{bound_on_D_tau}
&
\nonumber
\|D_{\tau} f(\bar G_k(\theta;0,U_1,\dots, U_k))\| 
\\
&
\leq \|f\|_{C^{k+1}(\Theta)} 
\prod_{l=1}^{k}\Bigl(\|H_l\|_{C^{k,0}(\Theta\times [0,1])}^{-}\vee 1\Bigr)^{l}
\|\dot H_l\|_{C^{k,0}(\Theta\times [0,1])}.
\end{align}
Under assumption \eqref{moment_of_H_deriv}, it easily follows that 
$$
{\mathbb E}\|D\partial_{T_k} f(G_k(\theta;U_1,\dots, U_k))\| <\infty.
$$
Since, by representation \eqref{B^k_int},
$$
({\mathcal B}^k f)(\theta)= {\mathbb E}\partial_{T_k} f(G_k(\theta;U_1,\dots, U_k)),
$$
we can now use Proposition \ref{diff_under_exp} to prove continuous differentiability 
of $({\mathcal B^k} f)(\theta)$ and, using \eqref{bound_on_D_tau} again,
prove bound \eqref{bound_on_B^k_deriv}.

Finally, to prove the last bound \eqref{bound_on_B^k_deriv_Lip}, we need  to control the ${\rm Lip}_{d,\rho}$-norm 
of functions  
$\theta\mapsto D_{\tau} f(\bar G_k(\theta;0,U_1,\dots, U_k))$ involved in the expression for 
$({\mathcal B}^{k} f)^{\prime}(\theta)$ as well as ${\rm Lip}_d$-norm of functions 
$\theta\mapsto \partial_{\tau} f(G_k(\theta; U_1,\dots, U_k))$ involved in the expression for 
$({\mathcal B}^{k} f)(\theta).$  
Here $d(\theta,\theta'):= \|\theta-\theta'\|, \theta,\theta'\in E$ is the metric of Banach space $E.$ 
It follows from the definition of $M$-labeling of tree $\tau$ that its $M$-labels and   
the derivative $D_{\tau} f(\bar G_k(\theta;0,U_1,\dots, U_k))[h]$ could be viewed 
as functions of variable 
$$
\eta(\theta):=\Bigl(\bar G_j(\theta;0, U_1,\dots, U_j): 0\leq j\leq k\Bigr)\in \Theta^{k+1}= 
\Theta \times \underset{k+1}{\cdots}  \times \Theta.   
$$
Define the distance between $\eta, \eta'\in \Theta^{k+1}$ as follows:
$$
\Delta(\eta, \eta'):= \max_{0\leq j\leq k}\|\eta_j-\eta_j'\|.
$$
%To this end, we introduce another distance 
%\begin{align*}
%\Delta(\theta,\theta'):= \max_{0\leq j\leq k} \|\bar G_j(\theta;0, U_1,\dots, U_j)- \bar G_j(\theta;0, U_1,\dots, U_j)\|,
%\theta, \theta'\in \Theta.
%\end{align*} 
Since functions $\bar G_j$ are defined in terms of superpositions of functions $H_j,$ it is easy 
to check that 
\begin{align*}
\Delta(\eta(\theta),\eta(\theta')) \leq \prod_{j=1}^k(\|H_j\|_{C^{1,0}(\Theta\times [0,1])}^{-}\vee 1) d(\theta,\theta'),
\theta, \theta'\in \Theta.
\end{align*}
On the other hand, using bound \eqref{continuity_mult_lin_tree_2} (and also bound \eqref{hoeld_lip}), 
we get 
\begin{align*}
&
\|D_{\tau} f(\bar G_k(\cdot;0,U_1,\dots, U_k))\|_{{\rm Lip}_{\Delta,\rho}}
\\
&
\leq 
(k+3)!\ 2^{(k+1)(k+2)}\|f\|_{C^{s}(\Theta)} 
\prod_{l=1}^{k}\Bigl(\|H_l\|_{C^{s-1,0}(\Theta\times [0,1])}^{-}\vee 1\Bigr)^{l}
\|\dot H_l\|_{C^{s-1,0}(\Theta\times [0,1])}.
\end{align*}
This implies 
\begin{align*}
&
\|D_{\tau} f(\bar G_k(\cdot;0,U_1,\dots, U_k))\|_{{\rm Lip}_{d,\rho}}
\\
&
\leq 
(k+3)!\ 2^{(k+1)(k+2)} \|f\|_{C^{s}(\Theta)} 
\prod_{l=1}^{k}\Bigl(\|H_l\|_{C^{s-1,0}(\Theta\times [0,1])}^{-}\vee 1\Bigr)^{l+\rho}
\|\dot H_l\|_{C^{s-1,0}(\Theta\times [0,1])}
\end{align*}
and provides a way to bound the ${\rm Lip}_{d,\rho}$-norm 
of the derivative $({\mathcal B}^k f)'(\theta)$ as follows: 
\begin{align*}
&
\|({\mathcal B}^{k} f)'\|_{{\rm Lip}_{d,\rho}(\Theta)} 
\\
&
\leq 3 (k+2)^{2(k+2)^2}\|f\|_{C^{s}(\Theta)}
\biggl({\mathbb E}\Bigl(\|H\|_{C^{s-1,0}(\Theta\times [0,1])}^{-}\vee 1\Bigr)^{s-1}
\|\dot H\|_{C^{s-1,0}(\Theta\times [0,1])}\biggr)^k.
\end{align*}
A similar bound for $\|{\mathcal B}^{k}f\|_{{\rm Lip}_d}$ could be proved with minor modifications 
of the argument. 
Together with \eqref{bound_on_B^k} this yields bound \eqref{bound_on_B^k_deriv_Lip}.
%We skip relatively elementary details of this computation. 
\qed
\end{proof}

\begin{example}
\normalfont
(See \cite{Koltchinskii_Zhilova}).
In the case of random shift model with $\Theta=E,$ $H(\theta,t)=\theta+t\xi,$  
$DH(\theta,t)=I,$ $\dot H(\theta,t)=\xi,$ and all the higher order derivatives of $H$ are equal to zero. 
Let $\xi_1, \xi_2, \dots$ be i.i.d. copies of $\xi.$ Then 
$$
\bar G_k (\theta, s, U_1,\dots, U_k) = \theta +sh + \sum_{j=1}^k U_j \xi_j.
$$
It is easy to see that in this case the only 
partition tree $\tau$ for which the terms in the right hand side of \eqref{main_represent} are 
not equal to zero is the tree corresponding to the following sequence of partitions: 
$(\Delta^{(k)}, \dots, {\Delta^{(0)}})$ with $\Delta^{(j)}:= (\{s\}, \{t_1\}, \dots, \{t_j\}).$
%In this case, \eqref{main_represent} implies 
%$$
%\partial_{\bar T_k}f(\bar G_k)= D^{k+1} f(\bar G_k)[\xi, \dots, \xi, h].
%$$
For instance, below is tree $\tau$ with $S$- and $M$-labeling for $k=3:$
\vskip 1mm
\Tree [.$\{s,t_1,t_2,t_3\}$ [.$\{s\}$ [.$\{s\}$ [.$\{s\}$ [.$\{s\}$ ] ] ] ] [.$\{t_1\}$ [.$\{t_1\}$ [.$\{t_1\}$ ] ] ] [.$\{t_2\}$ [.$\{t_2\}$ ] ] [.$\{t_3\}$ ] ]
\hskip 8mm
\Tree [.$f^{(4)}(\bar G_4)$ [.$I$ [.$I$ [.$I$ [.$h$ ] ] ] ] [.$I$ [.$I$ [.$\xi_1$ ] ] ] [.$I$ [.$\xi_2$ ] ] [.$\xi_3$ ] ]
\\
\\
For this tree,
\begin{align*}
&
D_{\tau} f(\bar G_k(\theta;s,U_1,\dots, U_k))[h] 
\\
&
= 
f^{(4)}\biggl(\theta +sh + \sum_{j=1}^k U_j \xi_j\biggr)[I[I[I[h]]],I[I[\xi_1]],I[\xi_2], \xi_3]
\\
&
=
f^{(4)}\biggl(\theta +sh + \sum_{j=1}^k U_j \xi_j\biggr)[h,\xi_1,\xi_2, \xi_3].
\end{align*}
It easily follows from propositions \ref{prop_B^k} and \ref{prop_B^k_deriv} that 
$$
({\mathcal B}^k f)(\theta) = {\mathbb E} f^{(k)}\biggl(\theta + \sum_{j=1}^k U_j \xi_j\biggr)[\xi_1,\dots, \xi_k]
$$
and 
$$
({\mathcal B}^k f)^{\prime}(\theta)[h] = {\mathbb E} f^{(k+1)}\biggl(\theta + \sum_{j=1}^k U_j \xi_j\biggr)[h,\xi_1,\dots, \xi_k]
$$
Of course, in this examples, there is a more direct and much simpler way to obtain these formulas (see \cite{Koltchinskii_Zhilova}).
\end{example}

%Similarly to Corollary \ref{cor:f_k_Hoeld}, we have the following.

%\begin{corollary}
%Suppose $f\in C^s(\Theta)$ for some $s=k+1+\rho, k\geq 1, \rho\in [0,1)$ 
%and 
%\begin{align*}
%(k+2)^{2(k+2)^2/k}
%{\mathbb E}\Bigl(\|D H\|_{C^{s-2,0}(\Theta\times [0,1])}\vee 1\Bigr)^{k}
%\|\dot H\|_{C^{s-1,0}(\Theta\times [0,1])}\leq 1/2,
%\end{align*} 
%then 
%\begin{align*}
%\|f_k\|_{C^{1+\rho}} \leq 2 \|f\|_{C^s}.
%\end{align*}
%\end{corollary}

\section{Proofs of the main results for normal models.} 
\label{normal_proofs}

In sections \ref{sec:approx_homot} and \ref{sec:conc} below, we will prove the following upper bound on the norm $\|f_k(\thetah)-f(\theta)\|_{L_\ell({\mathbb P}_{\theta})}$
for estimators $f_k(\thetah)$ of $f(\theta), k\geq 1:$

\begin{theorem}
\label{theorem_upper_detailed}
Let $\ell\in {\mathcal L}.$ 
Suppose $f\in C^s(\Theta)$ for some $s=k+1+\rho, k\geq 1, \rho\in (0,1]$
and assume that, for some $\epsilon>0,$
\begin{align}
\label{cond_a_epsilon_d_n}
a^{s-1+\epsilon} \sqrt{\frac{d}{n}}\leq c_s,
\end{align}
where $c_s\leq 1/2$ is a sufficiently small constant.
Then the following bound holds:
\begin{align}
\label{bound_on_norm_ell_fthetah}
&
\nonumber
\sup_{\theta\in \Theta(a;d)}\|f_k(\thetah)-f(\theta)\|_{L_{\ell}({\mathbb P}_{\theta})}
\\
&
\lesssim_{s,\epsilon, \ell}
\|f\|_{C^{s}(\Theta)}
\Biggl(\frac{a}{\sqrt{n}}
\bigvee a^{(s-1+\epsilon)k+1+\rho}\biggl(\sqrt{\frac{d}{n}}\biggr)^{s}
\bigvee 
\frac{1}{\ell^{-1}\Bigl((\ell(1)\wedge 1)\exp\Bigl\{\frac{n}{c^2(k+1)^2}\Bigr\}\Bigl)}
\Biggr).
\end{align}
Moreover, if $\ell (u)\leq e^{b u}, u\geq 0$ for some constant $b>0,$ then 
\begin{align}
\label{bound_on_norm_ell_fthetah_simple}
&
\sup_{\theta\in \Theta(a;d)}\|f_k(\thetah)-f(\theta)\|_{L_{\ell}({\mathbb P}_{\theta})}
\lesssim_{s,\epsilon,\ell}
\|f\|_{C^{s}(\Theta)}
\Biggl(\frac{a}{\sqrt{n}}
\bigvee a^{(s-1+\epsilon)k+1+\rho}\biggl(\sqrt{\frac{d}{n}}\biggr)^{s}
\Biggr).
\end{align}
\end{theorem}

This will be the main step in the proof of the bound 
of Theorem \ref{main_theorem_1} stated in Section \ref{MainResults}.
In Section \ref{sec:efficiency}, we will present the proof of Theorem \ref{main_theorem_2}.

\subsection{Approximation by a random homotopy and bounding the bias.}
\label{sec:approx_homot}

A natural choice of random homotopy to represent estimator $\hat \theta$ was described in 
Example \ref{ex_mean_covariance}. 
Using bound \eqref{bias_XXX_AAA'''} for this homotopy, we easily get the following bound 
on the bias of plug-in estimator $f(\thetah)=f_0(\thetah)$
\begin{align}
\label{bd_bias_plug-in}
|{\mathbb E}_{\theta} f(\thetah)-f(\theta)| \lesssim_{\rho} \|f\|_{C^{1+\rho}(\Theta)} a^{1+\rho}\biggl(\sqrt{\frac{d}{n}}\biggr)^{1+\rho}
\end{align}
that holds for all $\rho\in (0,1].$ Together with concentration bounds of Section \ref{sec:conc}
this allows us to complete the proof of Theorem \ref{main_theorem_1} in the simple case of $s=1+\rho,$
$k=0.$

However, the random homotopy of Example \ref{ex_mean_covariance} does not have bounded H\"older norms 
on the whole parameter space $\Theta={\mathbb R}^d \times {\mathcal C}_+^d$ (due to the presence of function $\Sigma^{1/2}$) which makes it impossible to use bounds of Theorem \ref{th_H_bound}
and Corollary \ref{cor_bd_f_k_hold} needed to analyze the estimators $f_k(\thetah)$ for $k\geq 1.$
To overcome this difficulty, we develop a smooth random homotopy that,
as in Proposition \ref{approx_homotopy} and Corollary \ref{corr_approx_homotopy}, 
approximates estimator $\hat \theta$ (and the corresponding bootstrap chain). To this end, 
we replace $\Sigma^{1/2}$ in the definition of $H(\theta;t)$ of Example \ref{ex_mean_covariance}
with a function $\gamma (\Sigma)$ which is smooth on the whole space ${\mathcal S}^d$ of symmetric 
matrices equipped with the operator norm and which coincides with $\Sigma^{1/2}$ on the set of covariance 
matrices $\Sigma$ such that $\sigma(\Sigma)\subset [1/(2a), 2a].$ 
In fact, $\gamma (\Sigma)$ is an application 
to $\Sigma$ of a function $\gamma$ of real variable defined as follows. 
Let $\gamma$ be a $C^{\infty}$
function in ${\mathbb R}$ such that 
\begin{itemize}
\item $0\leq \gamma(u)\leq \sqrt{u}, u\geq 0;$ 
\item $\gamma(u)=\sqrt{u}$ for all $u\in [(2a)^{-1}, 2a];$ 
\item ${\rm supp}(\gamma)\subset [(4a)^{-1}, 4a];$ 
\item $\|\gamma\|_{L_{\infty}}\lesssim \sqrt{a},$ $\|\gamma^{(j)}\|_{L_{\infty}}\lesssim a^{j-1/2}$
for all $1\leq j\leq s-1$ and $\|\gamma\|_{C^{s-1+\epsilon}}\lesssim a^{s-3/2+\epsilon}$ for an arbitrary $\epsilon>0.$
\end{itemize}
Recall that $s=k+1+\rho, k\geq 1,$ so, $s-1=k+\rho>1.$
A possible choice of such a function is $\gamma(u):= \lambda(2au) \sqrt{u} (1-\lambda (u/(4a))),$
where $\lambda$ is a nondecreasing $C^{\infty}$ function with values in $[0,1],$ $\lambda(u)=0, u\leq 1/2$ and 
$\lambda(u)=1, u\geq 1.$ 

It follows from Corollary 2 and other results of Section 2.2 of \cite{Koltchinskii_2017} that $A\mapsto \gamma(A)$ is 
a $C^{s-1}$ operator function from ${\mathcal S}^d$ (equipped with the operator norm) into itself 
and the following properties hold:
\begin{itemize}
\item for an arbitrary $\epsilon>0,$\footnote{Here $\|\cdot\|_{B_{\infty,1}^{s-1}}({\mathbb R})$ is a Besov norm;
it could be bounded in terms of $C^{s-1+\epsilon}({\mathbb R})$ norm, see \cite{Triebel}, Section 2.3.}  
\begin{align}
\label{bd_holder_gamma}
\|\gamma\|_{C^{s-1}({\mathcal S}^d)}\lesssim_s \|\gamma\|_{B_{\infty,1}^{s-1}({\mathbb R})}\lesssim_{\epsilon} 
\|\gamma\|_{C^{s-1+\epsilon}}\lesssim a^{s-3/2+\epsilon}
\end{align}
and 
\begin{align}
\label{bd_deriv_gamma}
\|\gamma\|_{L_{\infty}({\mathcal S}^d)}\lesssim \sqrt{a},\ \ 
\|D^{j}\gamma\|_{L_{\infty}({\mathcal S}^d)}\lesssim_{\epsilon} a^{j-1/2+\epsilon}, 1\leq j\leq s-1;
\end{align} 
\item $\Sigma- \gamma(\Sigma)^2\in {\mathcal C}_+^d;$ 
\item $\|\gamma(\Sigma)\|\leq \|\Sigma\|^{1/2}$ for all $\Sigma\in {\mathcal C}_+^d;$  
\item $\gamma(\Sigma)=\Sigma^{1/2}$ for all $\Sigma$ with $\sigma(\Sigma)\subset [(2a)^{-1}, 2a].$
\end{itemize}

Define now a random homotopy between $\theta$ and $\tilde \theta=H(\theta;1)$ as follows: 
$$
H(\theta;t)= \theta+ tE(\theta), \theta = (\mu, \Sigma)\in \Theta={\mathbb R}^d \times {\mathcal C}_+^d,
$$
where 
$$
E(\theta):= (\gamma (\Sigma)\bar Z, \gamma(\Sigma) (\hat \Sigma_Z-I_d)\gamma(\Sigma).
$$
Note that 
\begin{align*}
&
\Sigma + t  (\gamma(\Sigma) (\hat \Sigma_Z-I_d)\gamma(\Sigma))
=\Sigma - \gamma(\Sigma)^2 + (1-t) \gamma(\Sigma)^2 + t(\gamma (\Sigma)\hat \Sigma_Z \gamma(\Sigma))
\in {\mathcal C}_+^d.
\end{align*}
Indeed, recall that $\gamma^2(\lambda)\leq \lambda, \lambda \in {\mathbb R}_+.$ If $\gamma(\Sigma)^2 \neq \Sigma,$ then there exists $\lambda\in \sigma(\Sigma)$ such that $\gamma^2(\lambda)<\lambda.$ Therefore, $\Sigma-\gamma^2(\Sigma)\in {\mathcal C}_+^d.$ 
On the other hand, 
if $\gamma (\Sigma)^2=\Sigma,$ then $(1-t) \gamma(\Sigma)^2= (1-t)\Sigma\in {\mathcal C}_+^d$
for all $t\in [0,1).$ In each of these two cases, 
$\Sigma + t  (\gamma(\Sigma) (\hat \Sigma_Z-I_d)\gamma(\Sigma))\in {\mathcal C}_+^d.$
Finally, if $\gamma (\Sigma)^2=\Sigma$ and $t=1,$ then 
$$
\Sigma + t  (\gamma(\Sigma) (\hat \Sigma_Z-I_d)\gamma(\Sigma))= \gamma (\Sigma)\hat \Sigma_Z \gamma(\Sigma).
$$
Under assumption \eqref{cond_a_epsilon_d_n}, we have $d\leq n-1.$ It is well known that in this 
case the sample covariance matrix $\hat \Sigma_Z$ is nonsingular a.s.. Since also $\gamma(\Sigma)=\Sigma^{1/2}$ is nonsingular, we conclude that $\gamma (\Sigma)\hat \Sigma_Z \gamma(\Sigma)\in 
{\mathcal C}_+^d,$ implying the claim. 
Thus, it follows that $H(\theta;t)\in \Theta={\mathbb R}^d \times {\mathcal C}_+^d$ for all 
$\theta \in \Theta, t\in [0,1]$ a.s..

\begin{lemma}
\label{bounds_on_EEE}
The following bounds hold:
\begin{align*}
\|E\|_{L_{\infty}(\Theta)} \lesssim \sqrt{a}\|\bar Z\|+ a\|\hat \Sigma_Z-I_d\| 
\end{align*}
and 
\begin{align*}
\|E\|_{C^{s-1}(\Theta)}\lesssim_{\epsilon, s} 
a^{s-3/2+\epsilon}\|\bar Z\|+ a^{s-1+\epsilon}\|\hat \Sigma_Z-I_d\|.
\end{align*}
\end{lemma}

\begin{proof}
Note that $E(\theta)$ depends only on the operator component $\Sigma$ of $\theta=(\mu,\Sigma)$
through operator function $\gamma\in C^{s-1}({\mathcal S}^d),$
which allows us to replace $C^{s-1}(\Theta)$-norm with $C^{s-1}({\mathcal S}^d)$-norm.
Using \eqref{bd_holder_gamma}, it is easy to show that 
$$
\|\gamma(\cdot)\bar Z\|_{C^{s-1}({\mathcal S}^d)} \leq \|\gamma\|_{C^{s-1}({\mathcal S}^d)}
\|\bar Z\| \lesssim_{\epsilon, s} a^{s-3/2+\epsilon}\|\bar Z\|.
$$
On the other hand, by a version of Leibnitz formula, 
\begin{align*}
&
D^{k}(\gamma(\Sigma) (\hat \Sigma_Z-I_d)\gamma(\Sigma))[H_1,\dots, H_k]
\\
&
= \sum_{I\subset \{1,\dots, k\}}(D^{|I|} \gamma)(\Sigma)[\otimes_{i\in I}H_i](\hat \Sigma_Z-I_d)(D^{|I^c|} \gamma)(\Sigma)[\otimes_{i\in I^c} H_i], H_1,\dots, H_k \in {\mathcal S}^d. 
\end{align*}
Using bounds \eqref{bd_deriv_gamma} and \eqref{bd_holder_gamma} to control 
the norms of the derivatives of $\gamma$ involved in the last expression, it is easy to complete 
the proof. 
\qed
\end{proof}

Recall the definitions of operators $\tilde {\mathcal T},$ $\tilde {\mathcal B}$ and function $\tilde f_k$
(see Corollary \ref{corr_approx_homotopy}).

\begin{proposition}
Suppose condition \eqref{cond_a_epsilon_d_n} holds.
Then, for all $\theta\in \Theta,$
\begin{align}
\label{bias_ZZZ_XXX}
|{\mathbb E}_{\theta} \tilde f_k(\thetat)-f(\theta)|
\lesssim_{s,\epsilon} 
\|f\|_{C^{s}(\Theta)}
a^{(s-1+\epsilon)k+1+\rho}\biggl(\sqrt{\frac{d}{n}}\biggr)^{s}. 
\end{align}
\end{proposition}

\begin{proof}
%It is enough to prove the claim in the case when $a^{s-1+\epsilon}\sqrt{\frac{d}{n}}\leq 1.$
%Otherwise, it would trivially follow from bounds \eqref{f_k_L_infty} and \eqref{f_k_L_infty_tilde}.
Recall the following well known bounds: for all $p\geq 1,$ 
$$
{\mathbb E}^{1/p}\|\bar Z\|^p \lesssim_p \sqrt{\frac{d}{n}}
$$
and, under an additional assumption that $d\lesssim n,$
$$
{\mathbb E}^{1/p}\|\hat \Sigma_Z-I_d\|^p \lesssim_p \sqrt{\frac{d}{n}}.
$$
Using bounds of Lemma \ref{bounds_on_EEE}, we get 
%\begin{align}
%\label{bd_EEE_zero}
%&
%\|{\mathbb E}E(\theta)\| =\|n^{-1}I_d\|=n^{-1}
%\end{align}
\begin{align}
\label{bd_EEE_one}
{\mathbb E}^{1/p} \|E\|_{L_{\infty}(\Theta)}^p
\lesssim_p a \sqrt{\frac{d}{n}}
\end{align}
and 
\begin{align}
\label{bd_EEE_two}
{\mathbb E}^{1/p} \|E\|_{C^{s-1}(\Theta)}^p
\lesssim_{p,s,\epsilon} a^{s-1+\epsilon} \sqrt{\frac{d}{n}}.
\end{align}
%(the additional term $n^{-1}$ is dominated by other terms since $a\geq 1$ and $d\geq 1$).
It follows from bound \eqref{bias_ZZZ} that 
\begin{align}
\label{bias_ZZZ_AAA}
|{\mathbb E}_{\theta} \tilde f_k(\thetat)-f(\theta)|
\lesssim_s 
\|f\|_{C^{s}(\Theta)}
\biggl({\mathbb E}\Bigl(\|E\|_{C^{s-1}(\Theta)}\vee 1\Bigr)^{s-1}
\|E\|_{C^{s-1}(\Theta)}\biggr)^k {\mathbb E}\|E\|_{L_{\infty}(\Theta)}^{1+\rho}.
\end{align}
As a consequence of bounds \eqref{bd_EEE_one} and \eqref{bd_EEE_two},
\begin{align*}
{\mathbb E}\|E\|_{L_{\infty}(\Theta)}^{1+\rho}\lesssim a^{1+\rho}\biggl(\sqrt{\frac{d}{n}}\biggr)^{1+\rho}
\end{align*}
and, under the assumption \eqref{cond_a_epsilon_d_n},
\begin{align}
\label{bd_EEE_hold}
&
\nonumber
{\mathbb E}\Bigl(\|E\|_{C^{s-1}(\Theta)}\vee 1\Bigr)^{s-1}
\|E\|_{C^{s-1}(\Theta)}
\leq 
{\mathbb E}^{1/2}\Bigl(\|E\|_{C^{s-1}(\Theta)}+1\Bigr)^{2(s-1)}
{\mathbb E}^{1/2}\|E\|_{C^{s-1}(\Theta)}^2
\\
&
\nonumber 
\lesssim_s ({\mathbb E}^{1/2}\|E\|_{C^{s-1}(\Theta)}^{2(s-1)}+1){\mathbb E}^{1/2}\|E\|_{C^{s-1}(\Theta)}^2
\lesssim_{s,\epsilon} \biggl(a^{(s-1+\epsilon)(s-1)}\biggl(\sqrt{\frac{d}{n}}\biggr)^{s-1}+1\biggr)
a^{(s-1+\epsilon)}\sqrt{\frac{d}{n}} 
\\
&
\lesssim_{s,\epsilon} 
a^{(s-1+\epsilon)}\sqrt{\frac{d}{n}}.
\end{align}
Therefore, bound \eqref{bias_ZZZ_AAA} implies \eqref{bias_ZZZ_XXX}.
\qed
\end{proof}

\begin{proposition}
Suppose condition \eqref{cond_a_epsilon_d_n} holds.
Then, for all $\theta\in \Theta(a;d),$
\begin{align}
\label{bias_ZZZ_aaa}
|{\mathbb E}_{\theta} f_k(\thetah)-f(\theta)|
\lesssim_{s,\epsilon} 
\|f\|_{C^{s}(\Theta)}
a^{(s-1+\epsilon)k+1+\rho}\biggl(\sqrt{\frac{d}{n}}\biggr)^{s}. 
\end{align}
\end{proposition}

\begin{proof}
Note that, for $\theta = (\mu,\Sigma)\in \Theta (2a;d),$ we have $\sigma(\Sigma)\subset [1/(2a), 2a]$ and 
$\gamma(\Sigma)=\Sigma^{1/2}.$ Therefore, $\thetat = H(\theta;1)\overset{d}{=} \thetah$
and \eqref{bias_ZZZ_XXX} implies 
\begin{align}
\label{bias_ZZZ_thetah}
|{\mathbb E}_{\theta} \tilde f_k(\thetah)-f(\theta)|
\lesssim_{s,\epsilon} 
\|f\|_{C^{s}(\Theta)}
a^{(s-1+\epsilon)k+1+\rho}\biggl(\sqrt{\frac{d}{n}}\biggr)^{s},\ \theta \in \Theta(2a;d). 
\end{align}

Let $Q(\theta;\cdot)$ be the distribution of $\thetat=H(\theta;1)$ and $P(\theta;\cdot)$ be 
the distribution of $\thetah.$ Recall Proposition \ref{approx_homotopy} and its notations. 
Note that, for all $\theta \in \Theta(2a;d),$ 
$Q(\theta;\cdot)=P(\theta;\cdot).$ Let $\delta:= a/(k+1)$ and $A:= \Theta (a+\delta;d).$ It is easy 
to check that $A_{k\delta}\subset \Theta (2a;d).$
As a result, it follows from the last bound of Corollary \ref{corr_approx_homotopy}
that 
\begin{align}
\label{bd_f_k_tilde_f_k}
&
\sup_{\theta\in \Theta(a+\delta;d)}|f_k(\theta)- \tilde f_k(\theta)|
\leq k^2 2^{k+1} \|f\|_{L_{\infty}(\Theta)}
\sup_{\theta \in \Theta(2a;d)}{\mathbb P}_{\theta}\{\|\hat \theta-\theta\|\geq \delta\}.
\end{align}
We will use now the following exponential bound for $\|\thetah -\theta\|$
that is a corollary of well known bounds for the Euclidean norm $\|\bar X-\mu\|$ and 
the operator norm $\|\hat \Sigma-\Sigma\|:$ 
\begin{align*}
{\mathbb P}_{\theta}\biggl\{\|\hat \theta-\theta\|\geq ca\Bigl(\sqrt{\frac{d}{n}}\vee \sqrt{\frac{t}{n}}\vee 
\frac{t}{n}\Bigr)\biggr\}\leq e^{-t}.
\end{align*}
It holds with some constant $c\geq 1$ for all $d\lesssim n,$ all $\theta \in \Theta(2a;d)$ and all $t\geq 0.$ 
Assuming that $ca\sqrt{\frac{d}{n}}\leq \frac{a}{k+1}=\delta,$ or, equivalently, $d\leq \frac{n}{c^2(k+1)^2},$
and setting $t:= \frac{n}{c^2(k+1)^2},$ we get 
$$
\sup_{\theta \in \Theta(2a;d)}{\mathbb P}_{\theta}\{\|\hat \theta-\theta\|\geq \delta\}\leq 
\exp\Bigl\{-\frac{n}{c^2(k+1)^2}\Bigr\}
$$
and bound \eqref{bd_f_k_tilde_f_k} implies that 
\begin{align}
\label{bd_f_k_tilde_f_k'''}
&
\sup_{\theta\in \Theta(a+\delta;d)}|f_k(\theta)- \tilde f_k(\theta)|
\leq k^2 2^{k+1} \|f\|_{L_{\infty}(\Theta)}\exp\Bigl\{-\frac{n}{c^2(k+1)^2}\Bigr\}.
\end{align}
Note also that, for all $\theta \in \Theta(a;d),$
\begin{align}
\label{thetah_not_in_Theta}
{\mathbb P}_{\theta}\{\thetah \not \in \Theta(a+\delta;d)\}
\leq {\mathbb P}_{\theta}\{\|\hat \theta-\theta\|\geq \delta\}\leq 
\exp\Bigl\{-\frac{n}{c^2(k+1)^2}\Bigr\}.
\end{align}

Note that 
\begin{align*}
\|B^k f\|_{L_{\infty}(\Theta)}= \sup_{\theta\in \Theta}\biggl|{\mathbb E}_{\theta}\sum_{j=0}^k (-1)^{k-j}{k\choose j}
f(\thetah^{(j)})\biggr| \leq 2^k \|f\|_{L_{\infty}(\Theta)},
\end{align*}
\begin{align}
\label{f_k_L_infty}
\|f_k\|_{L_{\infty}(\Theta)}= \biggl\|\sum_{j=0}^k (-1)^j B^{j}f\biggr\|_{L_{\infty}(\Theta)}
\leq \sum_{j=0}^k 2^{j}\|f\|_{L_{\infty}(\Theta)}\leq 2^{k+1}\|f\|_{L_{\infty}(\Theta)}
\end{align}
and, similarly, 
\begin{align}
\label{f_k_L_infty_tilde}
\|\tilde f_k\|_{L_{\infty}(\Theta)}\leq 2^{k+1}\|f\|_{L_{\infty}(\Theta)}.
\end{align}
Using  \eqref{bd_f_k_tilde_f_k'''}, \eqref{thetah_not_in_Theta}, \eqref{f_k_L_infty} and \eqref{f_k_L_infty_tilde}, we can conclude that, 
for all $\theta\in \Theta(a;d),$ 
\begin{align}
\label{compare_f_k_tilde_f_k}
&
\nonumber
\Bigl|{\mathbb E}_{\theta}f_k(\thetah)- {\mathbb E}_{\theta}\tilde f_k(\thetah)\Bigr|
\\
&
\nonumber
\leq {\mathbb E}_{\theta} |f_k(\thetah)- \tilde f_k(\thetah)|I(\thetah \in \Theta(a+\delta;d))
+ (\|f_k\|_{L_{\infty}(\Theta)}+\|\tilde f_k\|_{L_{\infty}(\Theta)})
{\mathbb P}_{\theta}\{\thetah \not \in \Theta(a+\delta;d)\}
\\
&
\nonumber
\leq \sup_{\theta\in \Theta(a+\delta;d)}|f_k(\theta)-\tilde f_k(\theta)|
+(\|f_k\|_{L_{\infty}(\Theta)}+\|\tilde f_k\|_{L_{\infty}(\Theta)})
{\mathbb P}_{\theta}\{\thetah \not \in \Theta(a+\delta;d)\}
\\
&
\leq (k^2 +2) 2^{k+1} \|f\|_{L_{\infty}(\Theta)}\exp\Bigl\{-\frac{n}{c^2(k+1)^2}\Bigr\}.
\end{align}
Combining this with bound \eqref{bias_ZZZ_thetah}, we get that, for all $\theta \in \Theta(a;d),$ 
\begin{align}
\label{bias_ZZZ_thetah_fin}
|{\mathbb E}_{\theta} f_k(\thetah)-f(\theta)|
\lesssim_{s,\epsilon} 
\|f\|_{C^{s}(\Theta)}
\biggl(a^{(s-1+\epsilon)k+1+\rho}\biggl(\sqrt{\frac{d}{n}}\biggr)^{s}+\exp\Bigl\{-\frac{n}{c^2(k+1)^2}\Bigr\}\biggr).
\end{align}
Since, for $a\geq 1, d\geq 1,$ 
\begin{align}
\label{remove_exp}
\exp\Bigl\{-\frac{n}{c^2(k+1)^2}\Bigr\}\lesssim_s n^{-s/2}\lesssim_s a^{(s-1+\epsilon)k+1+\rho}\biggl(\sqrt{\frac{d}{n}}\biggr)^{s},
\end{align}
this concludes the proof in the case when $d\leq \frac{n}{c^2(k+1)^2}.$
In the opposite case, the proof is straightforward in view of bound
\eqref{f_k_L_infty}.

\qed
\end{proof}

\subsection{Concentration.}
\label{sec:conc}

For a locally Lipschitz function $g:S\mapsto {\mathbb R}$ on a metric space $(S;d),$
define 
$$
(Lg)(x):= \inf_{U\ni x}  \sup_{x_1,x_2\in U,x_1\neq x_2} \frac{|g(x_1)-g(x_2)|}{d(x_1,x_2)},
$$
where the infimum is taken over all neighborhoods $U$ of $x.$
If $S={\mathbb R}^N$ (viewed as a Euclidean space) and $g$ is a locally Lipschitz function,
then, by Rademacher's theorem, $g$ is differentiable a.s. in ${\mathbb R}^N$ and it is easy to see that, 
at the points of differentiability, $(Lf)(x)$ coincides with the Euclidean norm 
of the gradient of $f:$ $(Lf)(x)=\|\nabla f(x)\|$ a.s. in ${\mathbb R}^N.$

The following inequality (due to Maurey and Pisier, see, e.g., \cite{Pisier}) is well known.
 
\begin{proposition}
Let $\ell: {\mathbb R}\mapsto {\mathbb R}_+$ be a nonnegative convex function. 
Let $X, X'$ be independent standard normal r.v. in ${\mathbb R}^N$ and let $g:{\mathbb R}^N\mapsto {\mathbb R}$ be a locally Lipschitz function. Then 
\begin{align}
\label{Maurey}
{\mathbb E}\ell\Bigl(g(X)-{\mathbb E}g(X)\Bigr)\leq {\mathbb E}\ell\Bigl(\frac{\pi}{2}\langle \nabla g(X), X'\rangle \Bigr). 
\end{align} 
\end{proposition}

Given $\ell\in {\mathcal L},$ define 
$$
\ell^{\sharp} (u):= {\mathbb E} \ell (u Z), u\in {\mathbb R},
$$
where $Z\sim N(0,1).$ Clearly, it also holds that 
$
\ell^{\sharp} (u)= {\mathbb E} \ell (u |Z|), u\in {\mathbb R}
$ 
and $\ell^{\sharp}\in {\mathcal L}.$
Note that function $\ell^{\sharp}$ could take infinite values even when $\ell$ is finite. 
The set of points where 
it is finite is a symmetric interval (in principle, it could be $\{0\}$).

The following fact easily follows from \eqref{Maurey}.

\begin{proposition}
For any locally Lipschitz function $g:{\mathbb R}^N \mapsto {\mathbb R},$
\begin{align}
\label{Maurey_A}
{\mathbb E}\ell\Bigl(g(X)-{\mathbb E}g(X)\Bigr)
\leq {\mathbb E}\ell^{\sharp}\Bigl(\frac{\pi}{2}(Lg)(X)\Bigr). 
\end{align}  
\end{proposition}

\begin{proof}
Indeed, we have 
\begin{align*}
&
{\mathbb E}\ell\Bigl(g(X)-{\mathbb E}g(X)\Bigr)\leq 
{\mathbb E}\ell\Bigl(\frac{\pi}{2}\langle \nabla g(X), X'\rangle \Bigr). 
\\
&
= 
{\mathbb E}{\mathbb E}_{X'}\ell\Bigl(\frac{\pi}{2}\|\nabla g(X)\|\langle v(X), X'\rangle \Bigr)
={\mathbb E}{\mathbb E}_{X'}\ell\Bigl(\frac{\pi}{2}(Lg)(X)\langle v(X), X'\rangle \Bigr),
\end{align*}
where 
$
v(X)=\frac{\nabla g(X)}{\|\nabla g(X)\|}.
$
Conditionally on $X,$ the distribution of r.v. $\langle v(X), X'\rangle$ is standard normal, implying that 
\begin{align*}
{\mathbb E}_{X'}\ell\Bigl(\frac{\pi}{2}(Lg)(X)\langle v(X), X'\rangle \Bigr)= \ell^{\sharp} \Bigl(\frac{\pi}{2}(Lg)(X)\Bigr),
\end{align*}
and bound \eqref{Maurey_A} follows. 

\qed
\end{proof}

%In what follows, it will be convenient to use the Orlicz norm $\|\cdot\|_{\ell}$ defined as follows:
%$$
%\|\xi\|_{\ell}:= \inf\Bigl\{c>0: {\mathbb E}\ell\Bigl(\frac{|\xi|}{c}\Bigr)\leq 1\Bigr\}.
%$$
For a loss function $\ell\in {\mathcal L},$ define 
$$
c(\ell):= \|1\|_{\ell}= \inf\Bigl\{c>0: \ell\Bigl(\frac{1}{c}\Bigr)\leq 1\Bigr\}.
$$
Bound \eqref{Maurey_A} implies that 
\begin{align}
\label{Maurey_A_Orl}
\|g(X)-{\mathbb E}g(X)\|_{\ell}\leq \frac{\pi}{2}\|(Lg)(X)\|_{\ell^{\sharp}}.
\end{align}
Note that, if $g$ is Lipschitz with constant $L_g,$  bound \eqref{Maurey_A_Orl} implies 
that 
\begin{align}
\label{Maurey_A_A}
\|g(X)-{\mathbb E}g(X)\|_{\ell}
\leq \frac{\pi}{2}c(\ell^{\sharp}) L_g. 
\end{align}  
If $\ell (u):= |u|^p, u\in {\mathbb R}, p\geq 1,$ then 
$
\ell^{\sharp}(u) = {\mathbb E}|Z|^p |u|^p, u\in {\mathbb R}
$
and $c(\ell^{\sharp})=\|Z\|_{L_p}.$
The next Sobolev type bound follows from \eqref{Maurey_A_Orl} 
\begin{align}
\label{Sobolev}
\|g(X)-{\mathbb E}g(X)\|_{L_p} \leq \frac{\pi}{2}\|Z\|_{L_p}\|(Lg)(X)\|_{L_p}
\lesssim \sqrt{p}\|(Lg)(X)\|_{L_p},
\end{align}
yielding in the case of Lipschitz function $g$
\begin{align*}
\|g(X)-{\mathbb E}g(X)\|_{L_p} 
\lesssim \sqrt{p}L_g.
\end{align*}
Using the last bound for $p=t$ and combining it with Markov inequality, 
we easily get that, with probability at least $1-e^{-t},$
\begin{align*}
|g(X)-{\mathbb E}g(X)| \lesssim L_g \sqrt{t},
\end{align*}
which is a version of Gaussian concentration inequality. 

In the case when $g$ is not Lipschitz, inequality \eqref{Maurey_A} could still provide 
useful bounds, for instance, by the following iterative argument (that resembles the method 
already used in \cite{Adamczak}).

\begin{proposition}
Suppose that 
$$(Lg)(x)\leq \varphi (x), x\in {\mathbb R}^N,$$ 
where $\varphi :{\mathbb R}^N\mapsto {\mathbb R}$ is a Lipschitz function with constant $L_{\varphi}.$ 
Then
\begin{align}
\label{Maurey_ABC}
\|g(X)-{\mathbb E}g(X)\|_{\ell}
\leq \frac{\pi}{2}c(\ell^{\sharp}){\mathbb E}\varphi(X)+ \Bigl(\frac{\pi}{2}\Bigr)^2 c((\ell^{\sharp})^{\sharp}) L_{\varphi}.
\end{align}
\end{proposition}

\begin{proof}
Indeed,
\begin{align*}
&
\|g(X)-{\mathbb E}g(X)\|_{\ell}\leq \frac{\pi}{2}\|(Lg)(X)\|_{\ell^{\sharp}}
\leq \frac{\pi}{2} \|\varphi(X)\|_{\ell_{\sharp}}
\leq \frac{\pi}{2} \|{\mathbb E}\varphi(X)\|_{\ell_{\sharp}}+ \frac{\pi}{2} \|\varphi(X)-{\mathbb E}\varphi(X)\|_{\ell_{\sharp}}
\\
&
\leq \frac{\pi}{2} c(\ell^{\sharp}) {\mathbb E}\varphi(X)+ \Bigl(\frac{\pi}{2}\Bigr)^2 \|(L\varphi)(X)\|_{(\ell^{\sharp})^{\sharp}}
\leq \frac{\pi}{2} c(\ell^{\sharp}) {\mathbb E}\varphi(X)+ \Bigl(\frac{\pi}{2}\Bigr)^2 c((\ell^{\sharp})^{\sharp}) L_{\varphi}.
\end{align*}
\qed

\end{proof}

In particular, it follows from \eqref{Maurey_ABC} that, for all $p\geq 1,$ 
\begin{align*}
\|g(X)-{\mathbb E}g(X)\|_{L_p} 
\lesssim \sqrt{p}{\mathbb E}\varphi(X) + p L_{\varphi}
\end{align*}
that, in turn, implies the following concentration bound: with probability at least $1-e^{-t},$
\begin{align*}
|g(X)-{\mathbb E}g(X)| 
\lesssim \sqrt{t}{\mathbb E}\varphi(X) + t L_{\varphi}.
\end{align*}

We will now apply these inequalities to the function $f\circ \hat \theta,$ where $f$ is a Lipschitz 
function on ${\mathbb R}^d \times {\mathcal C}_+^d$ and $\hat \theta=(\hat \mu, \hat \Sigma)$
($\hat \mu=\bar X$ being the sample mean and $\hat \Sigma$ being the sample covariance)
is a standard estimator of the parameter $\theta = (\mu,\Sigma)$ of a normal model 
$
X_1,\dots, X_n\ {\rm i.i.d.}\ \sim N(\mu,\Sigma)\ {\rm in}\ {\mathbb R}^d.
$ 
\begin{comment}
Namely, 
$$
\hat \mu := \bar X = \frac{X_1+\dots + X_n}{n}, \ \ \hat \Sigma:= \frac{1}{n-1}\sum_{j=1}^n (X_j-\bar X)\otimes (X_j-\bar X)
$$
are the sample mean and the sample covariance. 
\end{comment}

\begin{proposition}
\label{conc_bd_smooth}
Let $f:{\mathbb R}^d \times {\mathcal C}_+^d\mapsto {\mathbb R}$ be a Lipschitz
function with Lipschitz constant $L_f>0.$
Then, for $d\lesssim n,$
\begin{align*}
\|f(\thetah)-{\mathbb E}_{\theta}f(\thetah)\|_{\ell}
\lesssim 
\frac{c(\ell^{\sharp})L_f \|\Sigma\|^{1/2}(1+\|\Sigma\|^{1/2})}{\sqrt{n}} + 
\frac{c((\ell^{\sharp})^{\sharp}) L_f \|\Sigma\|}{n}.
\end{align*}
\end{proposition}

\begin{proof}
Since we can represent $X_j$ as $X_j=\mu +\Sigma^{1/2}Z_j,$ where $Z_1,\dots, Z_n$ are i.i.d. $\sim N(0;I_d),$ we can view $f\circ \hat \theta$ as a function 
of a standard normal vector $Z=(Z_1,\dots, Z_n)\in {\mathbb R}^{N},$ where $N=n d.$ 
With a little abuse of notation, we will also write $\hat \theta, \hat \mu, \hat \Sigma$ as 
functions of $Z;$ $Z$ denotes both a random variable 
in ${\mathbb R}^N$ and a point in ${\mathbb R}^n$ (depending on the context).

\begin{lemma}
\label{lemma_L_bounds}
The following bounds hold for all $Z\in {\mathbb R}^N:$
\begin{align}
\label{Lhatmu}
(L\hat \mu)(Z) \leq \frac{\|\Sigma\|^{1/2}}{\sqrt{n}},
\end{align} 
\begin{align}
\label{LhatSigma}
(L\hat \Sigma)(Z)\leq \frac{4\sqrt{2}\|\Sigma\|^{1/2}\|\hat \Sigma(Z)\|^{1/2}}{\sqrt{n-1}},\ 
(L\|\hat \Sigma\|^{1/2})(Z)\leq \frac{2\sqrt{2}\|\Sigma\|^{1/2}}{\sqrt{n-1}},
\end{align}
and 
\begin{align}
\label{Lhattheta}
(L\hat \theta)(Z) \leq \frac{\|\Sigma\|^{1/2}}{\sqrt{n}}+\frac{4\sqrt{2}\|\Sigma\|^{1/2}\|\hat \Sigma(Z)\|^{1/2}}{\sqrt{n-1}}.
\end{align}
\end{lemma}

Note that the bound on $(L\|\hat \Sigma\|^{1/2})(Z)$ means that $\|\hat \Sigma(Z)\|^{1/2}$
is a Lipschitz function.

\begin{proof} 
We will show only the bound on $(L\hat \Sigma)(Z)$ (other proofs are similar). Denoting $Z':=(Z_1',\dots, Z_n'),$
$X_j':=\mu+\Sigma^{1/2}Z_j'$ and $\hat \Sigma'=\hat \Sigma(Z'),$
we have 
\begin{align*}
&
\|\hat \Sigma(Z)-\hat \Sigma(Z')\|= \sup_{\|u\|\leq 1, \|v\|\leq 1}\Bigl|\langle (\hat \Sigma-\hat \Sigma')u,v\rangle\Bigr|
\\
&
= \sup_{\|u\|\leq 1, \|v\|\leq 1} \biggl|
 \frac{1}{n-1}\sum_{j=1}^n \langle X_j-\bar X, u\rangle \langle X_j-\bar X, v\rangle
 -
 \frac{1}{n-1}\sum_{j=1}^n \langle X_j'-\bar X', u\rangle \langle X_j'-\bar X', v\rangle
 \biggr|
 \\
 &
 \leq \frac{1}{\sqrt{n-1}}\sup_{\|u\|\leq 1}\biggl(\sum_{j=1}^n \Bigl\langle X_j-X_j'-(\bar X-\bar X'), u\Bigr\rangle^2\biggr)^{1/2}
 \sup_{\|v\|\leq 1} \biggl(\frac{1}{n-1}\sum_{j=1}^n \langle X_j-\bar X, v\rangle^2\biggr)^{1/2}
 \\
 &
 +
 \frac{1}{\sqrt{n-1}}\sup_{\|v\|\leq 1}\biggl(\sum_{j=1}^n \Bigl\langle X_j-X_j'-(\bar X-\bar X'), v\Bigr\rangle^2\biggr)^{1/2}
 \sup_{\|u\|\leq 1} \biggl(\frac{1}{n-1}\sum_{j=1}^n \langle X_j'-\bar X', u\rangle^2\biggr)^{1/2} 
 \\
 &
 \leq
\frac{\|\Sigma\|^{1/2}(\|\hat \Sigma\|^{1/2}+ \|\hat \Sigma'\|^{1/2})}{\sqrt{n-1}}
\biggl(\sum_{j=1}^n \|Z_j-Z_j'-(\bar Z-\bar Z')\|^2\biggr)^{1/2}
\\
&
\leq 
\frac{\sqrt{2}\|\Sigma\|^{1/2}(\|\hat \Sigma\|^{1/2}+ \|\hat \Sigma'\|^{1/2})}{\sqrt{n-1}}
\biggl(\sum_{j=1}^n \|Z_j-Z_j'\|^2 + n\|\bar Z-\bar Z'\|^2\biggr)^{1/2}
\\
&
\leq 
\frac{\sqrt{2}\|\Sigma\|^{1/2}(\|\hat \Sigma\|^{1/2}+ \|\hat \Sigma'\|^{1/2})}{\sqrt{n-1}}
\biggl(
\biggl(\sum_{j=1}^n \|Z_j-Z_j'\|^2\biggr)^{1/2} + 
\sqrt{n}\|\bar Z-\bar Z'\|\biggr)
\\
&
\leq 
\frac{2\sqrt{2}\|\Sigma\|^{1/2}(\|\hat \Sigma\|^{1/2}+ \|\hat \Sigma'\|^{1/2})}{\sqrt{n-1}}
\|Z-Z'\|,
\end{align*}
which implies the bound on $(L\hat \Sigma)(Z)$ in \eqref{LhatSigma}.
\qed
\end{proof}

It follows from \eqref{Lhattheta} that, for a Lipschitz function $f,$  
\begin{align*}
&
L(f\circ \hat \theta)(Z) \leq 
\frac{L_{f}\|\Sigma\|^{1/2}}{\sqrt{n}}
\biggl(1+4\sqrt{2}\sqrt{\frac{n}{n-1}}\|\hat \Sigma(Z)\|^{1/2}\Biggr)
\\
&
\leq \frac{L_{f}\|\Sigma\|^{1/2}}{\sqrt{n}}(1+8\|\hat \Sigma(Z)\|^{1/2})=:\varphi(Z), n\geq 2.
\end{align*}
Using bound on $(L\|\hat \Sigma\|^{1/2})(Z)$ from \eqref{LhatSigma}, it is also easy to check that function $\varphi$ is Lipschitz with 
$
L_{\varphi}\leq \frac{16\sqrt{2}L_f\|\Sigma\|}{n-1}.
$ 
Note that, for $d\lesssim n,$
\begin{align*}
&
{\mathbb E}\varphi (Z) = \frac{L_{f}\|\Sigma\|^{1/2}}{\sqrt{n}}(1+8{\mathbb E}_{\theta}
\|\hat \Sigma\|^{1/2})
\\
&
\leq 
\frac{L_{f}\|\Sigma\|^{1/2}}{\sqrt{n}}(1+8{\mathbb E}_{\theta}^{1/2}\|\hat \Sigma\|)
\leq \frac{L_{f}\|\Sigma\|^{1/2}}{\sqrt{n}}\Bigl(1+8\|\Sigma\|^{1/2}+
8{\mathbb E}_{\theta}^{1/2}\|\hat \Sigma-\Sigma\|\Bigr)
\\
&
\lesssim 
\frac{L_{f}\|\Sigma\|^{1/2}}{\sqrt{n}}\Bigl(1+\|\Sigma\|^{1/2}+
\|\Sigma\|^{1/2}\Bigl(\sqrt{\frac{d}{n}}\vee \frac{d}{n}\Bigr)^{1/2}\Bigr)
\lesssim \frac{L_{f}\|\Sigma\|^{1/2}}{\sqrt{n}}(1+\|\Sigma\|^{1/2}).
\end{align*}
The result follows from bound \eqref{Maurey_ABC}.
\qed
\end{proof}

Now we are ready to proof Theorem \ref{theorem_upper_detailed}.

\begin{proof}
First note that 
\begin{align}
\label{bd_step_0}
&
\|f_k(\thetah)-f(\theta))\|_{\ell}
\leq 
\|\tilde f_k(\thetah)-{\mathbb E}_{\theta}\tilde f_k(\thetah)\|_{\ell}
+ \|{\mathbb E}_{\theta}\tilde f_k(\thetah)-f(\theta)\|_{\ell}
+ 
\|f_k(\thetah)-\tilde f_k(\thetah)\|_{\ell}.
\end{align} 
Recall that, under condition \eqref{cond_leq_0.5}, by Corollary \ref{cor_bd_f_k_hold}  
we have 
\begin{align*}
L_{\tilde f_k}\leq \|\tilde f_k\|_{C^{1+\rho}(\Theta)}\leq 2\|f\|_{C^{s}(\Theta)}. 
\end{align*}
Condition \eqref{cond_leq_0.5} holds 
if 
\begin{align*}
{\mathbb E}\Bigl(\|E\|_{C^{s-1}(\Theta)}\vee 1\Bigr)^{s-1}
\|E\|_{C^{s-1}(\Theta)}\leq d_s
\end{align*}
for a small enough constant $d_s>0.$
By bound \eqref{bd_EEE_hold}, the last condition holds under assumption \eqref{cond_a_epsilon_d_n}
with a sufficiently small constant $c_s.$
In this case, it follows from Proposition \ref{conc_bd_smooth}
that
\begin{align}
\label{bd_step_1}
&
\nonumber
\|\tilde f_k(\hat \theta)-{\mathbb E}_{\theta}\tilde f_k(\hat \theta)\|_{\ell}
\lesssim 
\frac{c(\ell^{\sharp}) \|f\|_{C^{s}(\Theta)}\|\Sigma\|^{1/2}(1+\|\Sigma\|^{1/2})}{\sqrt{n}} + 
\frac{c((\ell^{\sharp})^{\sharp}) \|f\|_{C^{s}(\Theta)}\|\Sigma\|}{n}
\\
&
\lesssim \frac{c(\ell^{\sharp}) \|f\|_{C^{s}(\Theta)}a}{\sqrt{n}} + 
\frac{c((\ell^{\sharp})^{\sharp}) \|f\|_{C^{s}(\Theta)}a}{n}.
\end{align}
Using bound \eqref{bias_ZZZ_thetah},
we get that, for all $\theta\in \Theta(a;d),$
\begin{align}
\label{bd_step_2}
\|{\mathbb E}_{\theta}\tilde f_k(\thetah)-f(\theta)\|_{\ell}
\lesssim_{s,\epsilon}
c(\ell)\|f\|_{C^{s}(\Theta)}
a^{(s-1+\epsilon)k+1+\rho}\biggl(\sqrt{\frac{d}{n}}\biggr)^{s}.
\end{align}
Note that, for any event $A,$ $\|I_A\|_{\ell}= \frac{1}{\ell^{-1}(1/{\mathbb P}(A))}.$ 
Using this fact along with bounds \eqref{bd_f_k_tilde_f_k'''}, \eqref{f_k_L_infty} and \eqref{thetah_not_in_Theta}, we get
\begin{align*}
&
\|f_k(\thetah)-\tilde f_k(\thetah)\|_{\ell}
\\
&
\leq 
\Bigl\|(f_k(\thetah)-\tilde f_k(\thetah))I(\thetah \in\Theta (a+\delta;d)\Bigr\|_{\ell}
+ 
\Bigl\|(f_k(\thetah)-\tilde f_k(\thetah))I(\thetah \not\in\Theta (a+\delta;d)\Bigr\|_{\ell}
\\
&
\leq 
\Bigl\|\sup_{\theta\in \Theta(a+\delta;d)}|f_k(\theta)- \tilde f_k(\theta)|\Bigr\|_{\ell}
+
(\|f_k\|_{L_{\infty}(\Theta)}+\|\tilde f_k\|_{L_{\infty}(\Theta)}))\Bigl\|I(\thetah \not\in\Theta (a+\delta;d)\Bigr\|_{\ell}
\\
&
\leq 
k^2 2^{k+1} \|f\|_{L_{\infty}(\Theta)}\exp\Bigl\{-\frac{n}{c^2(k+1)^2}\Bigr\}+
\frac{2^{k+2}\|f\|_{L_{\infty}(\Theta)}}{\ell^{-1}\biggl(1/{\mathbb P}_{\theta}\{\thetah \not\in \Theta (a+\delta;d)\}\biggr)}
\\
&
\leq k^2 2^{k+1} \|f\|_{L_{\infty}(\Theta)}\exp\Bigl\{-\frac{n}{c^2(k+1)^2}\Bigr\}+
\frac{2^{k+2}\|f\|_{L_{\infty}(\Theta)}}{\ell^{-1}\biggl(\exp\Bigl\{\frac{n}{c^2(k+1)^2}\Bigr\}\biggr)}.
\end{align*}
It is easy to check that, for $p\in (0,1),$ $p\leq \frac{1}{\ell^{-1}(\ell(1)/p)}.$ Applying this 
to $p=\exp\Bigl\{-\frac{n}{c^2(k+1)^2}\Bigr\},$ we easily get that 
\begin{align}
\label{bd_step_3}
&
\nonumber
\|f_k(\thetah)-\tilde f_k(\thetah)\|_{\ell}
\\
&
\nonumber
\leq k^2 2^{k+1} \|f\|_{L_{\infty}(\Theta)}\exp\Bigl\{-\frac{n}{c^2(k+1)^2}\Bigr\}+
\frac{2^{k+2}\|f\|_{L_{\infty}(\Theta)}}{\ell^{-1}\biggl(\exp\Bigl\{\frac{n}{c^2(k+1)^2}\Bigr\}\biggr)}
\\
&
\leq \frac{(k^2+2)2^{k+1}\|f\|_{L_{\infty}(\Theta)}}{{\ell^{-1}\biggl((\ell(1)\wedge 1)
\exp\Bigl\{\frac{n}{c^2(k+1)^2}\Bigr\}\biggr)}}.
\end{align}
Bound \eqref{bound_on_norm_ell_fthetah} easily follows from \eqref{bd_step_0}, \eqref{bd_step_1}, \eqref{bd_step_2} and \eqref{bd_step_3}. To prove \eqref{bound_on_norm_ell_fthetah_simple},
note that $\ell (u)\leq \ell(1)\wedge 1$ for all $u\in [0, 1\wedge \ell^{-1}(1)].$ Let 
$\lambda:=\frac{\log (1/(\ell(1)\wedge 1))}{1\wedge \ell^{-1}(1)}.$ It is easy to check 
that the condition $\ell(u)\leq e^{bu}, u\geq 0$ implies 
$$\ell (u)\leq (\ell(1)\wedge 1)e^{(b+\lambda)u}, u\geq 0.$$ 
It follows from the last inequality that 
$$
\frac{1}{\ell^{-1}\Bigl((\ell(1)\wedge 1)e^{(b+\lambda)u}\Bigr)}\leq \frac{1}{u}.
$$
Using this bound for $u:= \frac{n}{(b+\lambda)c^2(k+1)^2}$ allows us to get 
\begin{align*}
\frac{(k^2+2)2^{k+1}\|f\|_{L_{\infty}(\Theta)}}{{\ell^{-1}\biggl((\ell(1)\wedge 1)\exp\Bigl\{\frac{n}{c^2(k+1)^2}\Bigr\}\biggr)}}\lesssim_{s,\ell} \frac{\|f\|_{C^s(\Theta)}}{n} \lesssim_{s,l} \|f\|_{C^s(\Theta)}\frac{a}{\sqrt{n}},
\end{align*}
so, the last term of bound \eqref{bound_on_norm_ell_fthetah} could be dropped.

\qed
\end{proof}

Finally, we provide the proof of Theorem \ref{main_theorem_1}. 

\begin{proof}
First, note that, for $s=1+\rho$ with $\rho\in (0,1]$ and for $d\lesssim n,$
the bound of Proposition \ref{conc_bd_smooth} easily implies 
that 
\begin{align*}
\|f(\thetah)-{\mathbb E}_{\theta}f(\thetah)\|_{\ell}
\lesssim_{\ell} \|f\|_{C^s(\Theta)} \frac{a}{\sqrt{n}}, \theta \in \Theta(a;d).
\end{align*}
This could be combined with bound 
\eqref{bd_bias_plug-in} to get that 
\begin{align*}
\|f(\thetah)-f(\theta)\|_{\ell}
\lesssim_{s,\ell} \|f\|_{C^s(\Theta)}\biggl(\frac{a}{\sqrt{n}} \bigvee a\biggl(\sqrt{\frac{d}{n}}\biggr)^s\biggr), \theta \in \Theta(a;d).
\end{align*}
Since also $\|f(\thetah)\|\leq \|f\|_{L_{\infty}(\Theta)}\leq \|f\|_{C^s(\Theta)}$ and 
$\|f(\theta)\|\leq \|f\|_{L_{\infty}(\Theta)}\leq \|f\|_{C^s(\Theta)},$
we conclude that 
\begin{align*}
\|f(\thetah)-f(\theta)\|_{\ell} \lesssim_{\ell} \|f\|_{C^s(\Theta)},
\end{align*}
implying that, for all $\theta\in \Theta(a;d),$
\begin{align*}
\|f(\thetah)-f(\theta)\|_{\ell}
\lesssim_{s,\ell} \|f\|_{C^s(\Theta)}
\biggl[\biggl(\frac{a}{\sqrt{n}} \bigvee a\biggl(\sqrt{\frac{d}{n}}\biggr)^s\biggr)\bigwedge 1\biggr].
\end{align*}

To prove the bound of Theorem \ref{main_theorem_1} for $s=1+k+\rho$ with $k\geq 1, \rho\in (0,1],$ 
first note that by \eqref{f_k_L_infty},
\begin{align}
\label{bd_konec_1}
\|f_k(\thetah)-f(\theta)\|_{\ell} \leq \|f_k(\thetah)\|_{\ell}+ \|f(\theta)\|_{\ell} 
\lesssim_{\ell, k} \|f\|_{L_{\infty}(\Theta)}\lesssim_{\ell,k} \|f\|_{C^{s}(\Theta)},
\theta \in \Theta.
\end{align}
Then observe that, under assumption \eqref{cond_a_epsilon_d_n},
bound \eqref{bound_on_norm_ell_fthetah_simple} of Theorem \ref{theorem_upper_detailed}
implies that 
\begin{align}
\label{bd_konec_2}
\nonumber
\sup_{\theta\in \Theta(a;d)}\|f_k(\thetah)-f(\theta)\|_{L_{\ell}({\mathbb P}_{\theta})}
&
\nonumber
\lesssim_{s,\epsilon,\ell}
\|f\|_{C^{s}(\Theta)}
\Biggl(\frac{a}{\sqrt{n}}
\bigvee a^{(s-1+\epsilon)k+1+\rho}\biggl(\sqrt{\frac{d}{n}}\biggr)^{s}
\Biggr)
\\
&
\lesssim_{s,\epsilon,\ell}
\|f\|_{C^{s}(\Theta)}
\Biggl(\frac{a}{\sqrt{n}}
\bigvee a^{(s-1+\epsilon)s}\biggl(\sqrt{\frac{d}{n}}\biggr)^{s}
\Biggr).
\end{align}
It remains to set $\epsilon:=\beta-(s-1)$ and to combine bounds \eqref{bd_konec_1}
and \eqref{bd_konec_2} to complete the proof.

\qed
\end{proof}

\subsection{Normal approximation and efficiency.}
\label{sec:efficiency}

In this section, we provide the proof of Theorem \ref{main_theorem_2}.

\begin{proof}
For a differentiable function $g:\Theta\mapsto {\mathbb R},$
let 
$$
S_g(\theta;h) := g(\theta+h)-g(\theta)-\langle g^{\prime}(\theta), h\rangle 
$$
be the remainder of its first order Taylor expansion. 
We will use the following obvious representations:
\begin{align}
\label{represent_odin}
f_k(\thetah)-f(\theta) = \tilde f_k(\thetah)- {\mathbb E}_{\theta} \tilde f_k(\thetah) +
{\mathbb E}_{\theta} \tilde f_k(\thetah) - f(\theta) + f_k(\thetah)-\tilde f_k(\thetah)
\end{align}
and 
\begin{align}
\label{represent_dva}
&
\nonumber
\tilde f_k(\thetah)- {\mathbb E}_{\theta} \tilde f_k(\thetah)
\\
&
=\langle f^{\prime}(\theta), \thetah-\theta\rangle + 
\langle \tilde f_k^{\prime}(\theta)- f^{\prime}(\theta), \thetah-\theta\rangle
+ S_{\tilde f_k}(\theta, \thetah-\theta) - {\mathbb E}_{\theta}S_{\tilde f_k}(\theta, \thetah-\theta).
\end{align}
The proof of the next lemma is elementary. 

\begin{lemma}
Denote $\tilde \Sigma_Z:= n^{-1}\sum_{j=1}^n Z_j\otimes Z_j.$
For all $w\in {\mathbb R}^d, W\in {\mathcal S}^d,$ the following bounds hold:
\begin{align}
\label{hat_tilde_Sigma_Z}
\Bigl\|\langle W, \hat \Sigma_Z\rangle - \langle W, \tilde \Sigma_Z\rangle \Bigr\|_{L_2({\mathbb P})}
\leq \frac{2\sqrt{2}\|W\|_2}{(n-1)\sqrt{n}},
\end{align}
\begin{align}
\label{barZSigma_Z}
\biggl|\Bigl\|\langle w,\bar Z\rangle + \langle W, \hat \Sigma_Z-I_d\rangle\Bigr\|_{L_2({\mathbb P})}
- \sqrt{\frac{\|w\|^2}{n}+ \frac{2\|W\|_2^2}{n}}\biggr| \leq \frac{4 \|W\|_2}{\sqrt{(n-1)n}}
\end{align}
and 
\begin{align}
\label{barZSigma_Z_01}
\Bigl\|\langle w,\bar Z\rangle + \langle W, \hat \Sigma_Z-I_d\rangle\Bigr\|_{L_2({\mathbb P})}
\leq \frac{\|w\|}{\sqrt{n}}+\frac{(4+\sqrt{2})\|W\|_2}{\sqrt{n}}.
\end{align}
\end{lemma}

Note that, for $\theta=(\mu,\Sigma)$
\begin{align*}
&
\langle f^{\prime}(\theta), \thetah-\theta\rangle = 
\langle f^{\prime}_{\mu}(\mu,\Sigma), \bar X-\mu\rangle
+ \langle f^{\prime}_{\Sigma}(\mu,\Sigma), \hat \Sigma-\Sigma\rangle
= \langle w,\bar Z\rangle + \langle W, \hat \Sigma_Z-I_d\rangle,
\end{align*}
where $w:= \Sigma^{1/2}f^{\prime}_{\mu}(\mu,\Sigma), 
W:=\Sigma^{1/2} f^{\prime}_{\Sigma}(\mu,\Sigma)\Sigma^{1/2}.$
Thus, in view of \eqref{barZSigma_Z} and the definition of $\sigma_f(\theta),$ we get 
\begin{align*}
\Bigl|\Bigl\|\langle f^{\prime}(\theta), \thetah-\theta\rangle\Bigr\|_{L_2({\mathbb P})}- 
\frac{\sigma_f(\theta)}{\sqrt{n}}\Bigr|
\leq \frac{4 \|\Sigma^{1/2} f^{\prime}_{\Sigma}(\mu,\Sigma)\Sigma^{1/2}\|_2}{\sqrt{(n-1)n}}.
\end{align*}
For $\theta \in \Theta(a;d),$ this implies 
\begin{align}
\label{first_term}
\Bigl|\Bigl\|\langle f^{\prime}(\theta), \thetah-\theta\rangle\Bigr\|_{L_2({\mathbb P})}- \frac{\sigma_f(\theta)}{\sqrt{n}}\Bigr|
\leq \frac{4\|f^{\prime}\|_{L_{\infty}(\Theta)}a}{\sqrt{(n-1)n}}
\end{align}
and, by a similar computation, bound \eqref{barZSigma_Z_01} yields 
\begin{align*}
\Bigl\|\langle f^{\prime}(\theta), \thetah-\theta\rangle\Bigr\|_{L_2({\mathbb P})}
\leq \frac{7\|f^{\prime}\|_{L_{\infty}(\Theta)}a}{\sqrt{n}}.
\end{align*}
This could be applied to function $\tilde f_k-f$ to get that 
\begin{align*}
\Bigl\|\langle \tilde f_k^{\prime}(\theta)-f^{\prime}(\theta), \thetah-\theta\rangle\Bigr\|_{L_2({\mathbb P})}
\leq \frac{7\|\tilde f_k^{\prime}-f^{\prime}\|_{L_{\infty}(\Theta)}a}{\sqrt{n}}.
\end{align*}

The next lemma easily follows from Theorem \ref{th_H_bound} and bound \eqref{bd_EEE_hold}.

\begin{lemma}
\label{f_k_f_k_tilde}
If $d\leq c n$ for a sufficiently small constant $c>0,$ then, for all $\epsilon>0,$ 
\begin{align*}
\|\tilde f_k^{\prime}-f^{\prime}\|_{L_{\infty}(\Theta)} \lesssim_{s,\epsilon} \|f\|_{C^s(\Theta)} a^{s-1+\epsilon}
\sqrt{\frac{d}{n}}. 
\end{align*}
\end{lemma}

Using the bound of Lemma \ref{f_k_f_k_tilde}, we get 
\begin{align}
\label{second_term}
\Bigl\|\langle \tilde f_k^{\prime}(\theta)-f^{\prime}(\theta), \thetah-\theta\rangle\Bigr\|_{L_2({\mathbb P})}
\lesssim_{s,\epsilon} \frac{\|f\|_{C^s(\Theta)} a^{s+\epsilon}}{\sqrt{n}}\sqrt{\frac{d}{n}}. 
\end{align}

To control the term 
$S_{\tilde f_k}(\theta, \thetah-\theta) - {\mathbb E}_{\theta}S_{\tilde f_k}(\theta, \thetah-\theta),$
the following lemma will be used.

\begin{lemma}
For any $g\in C^{1+\rho}(\Theta),$
\begin{align}
\label{conc_taylor_rem}
\Bigl\|S_{g}(\theta, \thetah-\theta) - {\mathbb E}_{\theta}S_{g}(\theta, \thetah-\theta)\Bigr\|_{L_2({\mathbb P})} 
\lesssim \frac{\|g\|_{C^{1+\rho}(\Theta)}(\|\Sigma\|^{1+\rho}\vee \|\Sigma\|^{(1+\rho)/2})}{\sqrt{n}} 
\Bigl(\sqrt{\frac{d}{n}}\Bigr)^{\rho}.
\end{align}
\end{lemma}

\begin{proof}
We use the following elementary Lipschitz bound on the remainder $S_g(\theta;h)$
that holds for all $\theta, \theta+h, \theta+h'\in \Theta$ (see \cite{Koltchinskii_Zhilova}, Lemma 2.1):
\begin{align}
\label{remainder_lipschitz}
|S_g(\theta;h)-S_g(\theta;h')|\lesssim \|g\|_{C^{1+\rho}(\Theta)}(\|h\|\vee \|h'\|)^{\rho}\|h-h'\|.  
\end{align}
Given $\theta\in \Theta,$ denote $G(h):=G_{\theta}(h):= S_g(\theta;h).$ Then, \eqref{remainder_lipschitz}
implies that
\begin{align}
\label{remainder_lipschitz_AB}
(LG)(h) \lesssim \|g\|_{C^{1+\rho}(\Theta)}\|h\|^{\rho}.
\end{align}
As in Section \ref{sec:conc}, we represent $X_j= \mu +\Sigma^{1/2} Z_j,$ where 
$Z_j, j=1,\dots, n$ are i.i.d. $N(0;I_d)$ r.v. and we view 
$
S_g(\theta; \thetah-\theta) = G(\thetah -\theta) = G\circ (\thetah -\theta)
$
as a function of $Z=(Z_1,\dots, Z_n)\in {\mathbb R}^N$ with $N=nd.$
Using bound \eqref{Lhattheta} of Lemma \ref{lemma_L_bounds}, we get for this function 
\begin{align*}
L(G\circ (\thetah -\theta))(Z) \lesssim \|g\|_{C^{1+\rho}(\Theta)} \|\thetah-\theta\|^{\rho}
(1+\|\hat \Sigma\|^{1/2})
\frac{\|\Sigma\|^{1/2}}{\sqrt{n}}.
\end{align*}
Using \eqref{Sobolev} for $p=2,$ we have 
\begin{align*}
&
\Bigl\|S_{g}(\theta, \thetah-\theta) - {\mathbb E}_{\theta}S_{g}(\theta, \thetah-\theta)\Bigr\|_{L_2({\mathbb P})}
\\
&
\lesssim
\|g\|_{C^{1+\rho}(\Theta)} \frac{\|\Sigma\|^{1/2}}{\sqrt{n}}
\Bigl\|\|\thetah-\theta\|^{\rho}\Bigr\|_{L_4({\mathbb P})}
\Bigl(1+\Bigl\|\|\hat \Sigma\|^{1/2}\Bigr\|_{L_4({\mathbb P})}\Bigr)
\end{align*}
Using a relatively standard bound 
\begin{align*}
\Bigl\|\|\thetah-\theta\|^{\rho}\Bigr\|_{L_4({\mathbb P})}
\lesssim (\|\Sigma\|\vee \|\Sigma\|^{1/2})^{\rho} \Bigl(\sqrt{\frac{d}{n}}\Bigr)^{\rho},
\end{align*}
the concentration bound for $\|\hat \Sigma\|^{1/2}$
\begin{align*}
\Bigl\|
\|\hat \Sigma\|^{1/2}-{\mathbb E}_{\theta}\|\hat \Sigma\|^{1/2}
\Bigr\|_{L_4({\mathbb P})}
\lesssim \frac{\|\Sigma\|^{1/2}}{\sqrt{n}}
\end{align*}
and, for $d\lesssim n,$ the bound 
\begin{align*}
{\mathbb E}_{\theta}\|\hat \Sigma\|^{1/2} \leq \|\Sigma\|^{1/2}
+ {\mathbb E}_{\theta}^{1/2}\|\hat \Sigma-\Sigma\|
\lesssim \|\Sigma\|^{1/2}\Bigl(1+\Bigl(\frac{d}{n}\Bigr)^{1/4}\Bigr)\lesssim \|\Sigma\|^{1/2},
\end{align*}
it is easy to complete the proof.

\qed
\end{proof}

Suppose $a^{s-1+\epsilon}\sqrt{\frac{d}{n}}\leq c_s$ for small enough constant $c_s.$ 
This holds under the assumptions of Theorem \ref{main_theorem_2} provided that $n$
is large enough.
It follows from Corollary \ref{cor_bd_f_k_hold} 
and bound \eqref{bd_EEE_hold} that
$\|\tilde f_k\|_{C^{1+\rho}(\Theta)}\leq 2\|f\|_{C^{s}(\Theta)}.$ Thus, bound \eqref{conc_taylor_rem}
implies that, for all $\theta \in \Theta(a;d),$
\begin{align}
\label{conc_taylor_rem_AAA}
\Bigl\|S_{\tilde f_k}(\theta, \thetah-\theta) - {\mathbb E}_{\theta}S_{\tilde f_k}(\theta, \thetah-\theta)\Bigr\|_{L_2({\mathbb P})} 
\lesssim \frac{\|f\|_{C^{s}(\Theta)}a^{1+\rho}}{\sqrt{n}} 
\Bigl(\sqrt{\frac{d}{n}}\Bigr)^{\rho}.
\end{align} 

Similarly to the proof of bound \eqref{compare_f_k_tilde_f_k}, 
we have 
\begin{align}
\label{compare_f_k_tilde_f_k_L_2}
&
\|f_k(\thetah)- \tilde f_k(\thetah)\|_{L_2({\mathbb P})}
\leq (k^2 +2) 2^{k+1} \|f\|_{L_{\infty}(\Theta)}\exp\Bigl\{-\frac{n}{2 c^2(k+1)^2}\Bigr\}.
\end{align}
Using representations \eqref{represent_odin}, \eqref{represent_dva}, 
bounds \eqref{first_term}, \eqref{second_term}, \eqref{conc_taylor_rem_AAA},
\eqref{bias_ZZZ_thetah}, \eqref{compare_f_k_tilde_f_k_L_2} and \eqref{remove_exp},
we get that, for all $\theta\in \Theta (a;d),$
\begin{align*}
&
\Bigl|\|f_k(\thetah)-f_k(\theta)\|_{L_2({\mathbb P})}-\frac{\sigma_f(\theta)}{\sqrt{n}}\Bigr|
%\\
%&
\lesssim_{s,\epsilon}
\frac{\|f^{\prime}\|_{L_{\infty}(\Theta)}a}{\sqrt{(n-1)n}}+ 
\frac{\|f\|_{C^s(\Theta)} a^{s+\epsilon}}{\sqrt{n}}\sqrt{\frac{d}{n}}
\\
&
+
\frac{\|f\|_{C^{s}(\Theta)}a^{1+\rho}}{\sqrt{n}} \Bigl(\sqrt{\frac{d}{n}}\Bigr)^{\rho}
+
\|f\|_{C^{s}(\Theta)}a^{(s-1+\epsilon)k+1+\rho}\biggl(\sqrt{\frac{d}{n}}\biggr)^{s}.
\end{align*}
Under the assumption that $d=d_n\leq n^{\alpha}$ for some $\alpha\in (0,1)$
and $s>\frac{1}{1-\alpha},$ the last bound implies that 
\begin{align*}
\sup_{\|f\|_{C^s(\Theta)}\leq 1} \sup_{\theta\in \Theta(a;d_n)}\Bigl|\|f_k(\thetah)-f_k(\theta)\|_{L_2({\mathbb P})}-\frac{\sigma_f(\theta)}{\sqrt{n}}\Bigr|= o(n^{-1/2})
\end{align*}
as $n\to \infty.$ It is also easy to see that 
$
\sup_{\|f\|_{C^s(\Theta)}\leq 1} \sup_{\theta\in \Theta(a;d_n)}\sigma_f(\theta)=O(1),
$
and claim \eqref{efficient_funct} of Theorem \ref{main_theorem_2} follows.

To prove the normal approximation \eqref{normal_approx_funct}, note that 
by representations \eqref{represent_odin}, \eqref{represent_dva} and 
bounds \eqref{second_term}, \eqref{conc_taylor_rem_AAA},
\eqref{bias_ZZZ_thetah}, \eqref{compare_f_k_tilde_f_k_L_2} and \eqref{remove_exp},
\begin{align}
\label{zzz_111}
f_k (\thetah)-f(\theta) = \langle w, \bar Z\rangle + \langle W, \hat \Sigma_Z-I_d\rangle+
\zeta_n, 
\end{align}
where $w=\Sigma^{1/2}f^{\prime}_{\mu}(\theta),$ $W=\Sigma^{1/2}f^{\prime}_{\Sigma}(\theta)\Sigma^{1/2}$
and, for all $\theta \in \Theta(a;d),$ 
\begin{align}
\label{zzz_222}
&
\nonumber
\|\zeta_n\|_{L_2({\mathbb P})} 
\lesssim_{s,\epsilon}
\\
&
\frac{\|f\|_{C^s(\Theta)} a^{s+\epsilon}}{\sqrt{n}}\sqrt{\frac{d}{n}}
+
\frac{\|f\|_{C^{s}(\Theta)}a^{1+\rho}}{\sqrt{n}} \Bigl(\sqrt{\frac{d}{n}}\Bigr)^{\rho}
+
\|f\|_{C^{s}(\Theta)}a^{(s-1+\epsilon)k+1+\rho}\biggl(\sqrt{\frac{d}{n}}\biggr)^{s}.
\end{align}
Recall also that $\bar Z$ and $\hat \Sigma_Z$ are independent r.v. Therefore, if $(Z_1', \dots, Z_n')$
is an independent copy of the sample $(Z_1,\dots, Z_n)$ and $\hat \Sigma_{Z'}$ is 
the sample covariance based on $(Z_1',\dots, Z_n'),$
then 
\begin{align}
\label{zzz_333}
\langle w, \bar Z\rangle + \langle W, \hat \Sigma_Z-I_d\rangle
\overset{d}{=}
\langle w, \bar Z\rangle + \langle W, \hat \Sigma_{Z'}-I_d\rangle.
\end{align}
Using bound \eqref{hat_tilde_Sigma_Z}, we can write 
\begin{align}
\label{zzz_444}
\langle w, \bar Z\rangle + \langle W, \hat \Sigma_{Z'}-I_d\rangle
= \langle w, \bar Z\rangle + \langle W, \tilde \Sigma_{Z'}-I_d\rangle
+ \eta_n,
\end{align}
where 
\begin{align}
\label{zzz_555}
\|\eta_n\|_{L_2({\mathbb P})}
\leq \frac{2\sqrt{2}\|\Sigma^{1/2}f^{\prime}_{\Sigma}(\theta)\Sigma^{1/2}\|_2}{(n-1)\sqrt{n}}
\leq \frac{2\sqrt{2}\|f^{\prime}\|_{L_{\infty}(\Theta)}a}{(n-1)\sqrt{n}}.
\end{align}
To complete the proof of \eqref{normal_approx_funct}, it remains to use the following 
two lemmas that can be proved similarly to lemmas 9 and 10 in \cite{Koltchinskii_2017}.  
The first lemma provides a normal approximation bound for r.v. 
\begin{align*}
\langle w, \bar Z\rangle + \langle W, \tilde \Sigma_{Z'}-I_d\rangle
= n^{-1} \sum_{j=1}^n \langle w, Z_j\rangle + n^{-1} \sum_{j=1}^n (\langle WZ_j',Z_j'\rangle-
{\mathbb E} \langle WZ',Z'\rangle). 
\end{align*}

\begin{lemma}
\label{B_E_type}
The following bound holds:
\begin{align*}
\sup_{x\in {\mathbb R}}\biggl|
{\mathbb P}
\biggl\{\frac{\sqrt{n}\bigl(\langle w, \bar Z\rangle + \langle W, \tilde \Sigma_{Z'}-I_d\rangle\bigr)}{\sqrt{\|w\|^2+ 2\|W\|_2^2}}\leq x\biggr\}-{\mathbb P}\{Z\leq x\}\biggr|
\lesssim \frac{\|W\|}{\|W\|_2}\frac{1}{\sqrt{n}}\lesssim \frac{1}{\sqrt{n}},
\end{align*}
where $Z\sim N(0,1).$
\end{lemma}

To state the second lemma, denote for r.v. $\xi_1, \xi_2,$
\begin{align*}
\Delta(\xi_1,\xi_2) := \sup_{x\in {\mathbb R}}|{\mathbb P}\{\xi_1\leq x\}-{\mathbb P}\{\xi_2\leq x\}|.
\end{align*}

\begin{lemma}
\label{xi_1xi_2}
If $Z\sim N(0,1),$ then for all r.v. $\xi_1, \xi_2,$
\begin{align*}
\Delta(\xi_1,Z)\leq \Delta(\xi_2,Z)+ 2\|\xi_1-\xi_2\|_{L_2({\mathbb P})}^{2/3}.
\end{align*}
\end{lemma}

Lemmas \ref{B_E_type} and \ref{xi_1xi_2} are first applied to 
$\xi_1:=\frac{n^{1/2}(\langle w, \bar Z\rangle + \langle W, \hat \Sigma_{Z'}-I_d\rangle)}{\sqrt{\|w\|^2+ 2\|W\|_2^2}}$ and 
$\xi_2:=
\frac{n^{1/2}(\langle w, \bar Z\rangle + \langle W, \tilde \Sigma_{Z'}-I_d\rangle)}{\sqrt{\|w\|^2+ 2\|W\|_2^2}}.$ 
Taking into account \eqref{zzz_333}, \eqref{zzz_444} and \eqref{zzz_555},
this yields 
\begin{align*}
\sup_{x\in {\mathbb R}}\biggl|
{\mathbb P}
\biggl\{
\frac{\sqrt{n}\bigl(\langle w, \bar Z\rangle + \langle W, \hat \Sigma_{Z}-I_d\rangle\bigr)}
{\sqrt{\|w\|^2+ 2\|W\|_2^2}}\leq x\biggr\}-{\mathbb P}\{Z\leq x\}\biggr|
\lesssim \frac{1}{\sqrt{n}}+ \frac{\|f^{\prime}\|_{L_{\infty}(\Theta)}^{2/3}a^{2/3}}{\sigma_f^{2/3}(\theta)(n-1)^{2/3}},
\end{align*}
where we also used that $\sigma_f(\theta)=\sqrt{\|w\|^2+ 2\|W\|_2^2}.$
Next we apply lemmas \ref{B_E_type} and \ref{xi_1xi_2} to $\xi_1:=\frac{n^{1/2}(f_k(\thetah)-f(\theta))}{\sigma_f(\theta)}$ and $\xi_2:=\frac{n^{1/2}(\langle w, \bar Z\rangle + \langle W, \hat \Sigma_{Z'}-I_d\rangle)}{\sqrt{\|w\|^2+ 2\|W\|_2^2}}.$ Along with \eqref{zzz_111} and \eqref{zzz_222} this yields 
\begin{align*}
&
\sup_{x\in {\mathbb R}}\biggl|
{\mathbb P}
\biggl\{
\frac{\sqrt{n}(f_k(\thetah)-f(\theta))}
{\sigma_f(\theta)}\leq x\biggr\}-{\mathbb P}\{Z\leq x\}\biggr|
\\
&
\lesssim_{s,\epsilon} \frac{1}{\sqrt{n}}+ \frac{\|f^{\prime}\|_{L_{\infty}(\Theta)}^{2/3}a^{2/3}}{\sigma_f^{2/3}(\theta)(n-1)^{2/3}}
+\frac{\|f\|_{C^s(\Theta)}^{2/3} a^{2(s+\epsilon)/3}}{\sigma_f^{2/3}(\theta)}\Bigl(\sqrt{\frac{d}{n}}\Bigr)^{2/3}
+
\frac{\|f\|_{C^{s}(\Theta)}^{2/3}a^{2(1+\rho)/3}}{\sigma_f(\theta)^{2/3}} \Bigl(\sqrt{\frac{d}{n}}\Bigr)^{2\rho/3}
\\
&
+
\frac{\|f\|_{C^{s}(\Theta)}^{2/3}a^{(2/3)((s-1+\epsilon)k+1+\rho)}}{\sigma_f^{2/3}(\theta)}\Bigl(\sqrt{n}\biggl(\sqrt{\frac{d}{n}}\biggr)^{s}\Bigr)^{2/3}.
\end{align*}
Under the assumptions of the theorem, the last bound implies  \eqref{normal_approx_funct}.  

\qed
\end{proof}


\begin{thebibliography}{10}

\bibitem{Aleksandrov_Peller}  A.~Aleksandrov and V.~Peller. 
\newblock Operator Lipschitz Functions.  
\newblock {\em Uspekhi Mat. Nauk}, 2016, 71, 4(430), 3--106.
\newblock arXiv: 1611.01593.

\bibitem{Adamczak} R.~Adamczak and P.~Wolffe. Concentration inequalities for non-Lipschitz
functions with bounded derivatives of higher orders.
\newblock {\em Probability Theory and Related Fields}, 2015, 162 (3-4), 531--586.

 


\bibitem{Bickel_Ritov} P.~Bickel and Y.~Ritov.
\newblock Estimating integrated square density derivatives: sharp best order of convergence 
estimates. 
\newblock {\em Sankhya}, 1988, 50, 381--393. 

\bibitem{BKRW} P.J.~Bickel, C.A.J.~Klaassen, Y.~Ritov and J.A.~Wellner. 
\newblock {\em Efficient and Adaptive Estimation for Semiparametric Models}.
\newblock Johns Hopkins University Press, Baltimore, 1993.

\bibitem{Birge} L.~Birg\'e and P.~Massart. 
\newblock Estimation of integral functionals of a density. 
\newblock {\em Annals of Statistics}, 1995, 23, 11-29.

%\bibitem{BlanchardBousquetZwald} G.~Blanchard, O.~Bousquet, and L.~Zwald. 
%\newblock Statistical properties of kernel principal component analysis. 
%\newblock { \em Machine Learning,} 2007, 66(2-3), 259-294. 



\bibitem{Cai_Low_2005a} T.T.~Cai and M.~Low. On adaptive estimation of linear 
functionals. {\em Annals of Statistics}, 2005, 33, 2311--2343. 

\bibitem{Cai_Low_2005b} T.T.~Cai and M.~Low. Non-quadratic estimators 
of a quadratic functional. {\em Annals of Statistics}, 2005, 33, 2930--2956.


\bibitem{C_C_Tsybakov} O.~Collier, L.~Comminges and A.~Tsybakov.
Minimax estimation of linear and quadratic functionals on sparsity classes.
{\em Annals of Statistics}, 2017, 45, 3, 923--958. 


\bibitem{Donoho_1} D.~Donoho and R.~Liu.
\newblock On minimax estimation of linear functionals. 
\newblock {Technical Report N 105. Department of Statistics, UC Berkeley}, August 1987. 

\bibitem{Donoho_2} D.~Donoho and R.~Liu.
\newblock Geometrizing rates of convergence, II.
\newblock {\em Annals of Statistics}, 1991, 19, 2, 633-667.

\bibitem{Donoho_Nussbaum} D.~Donoho and M.~Nussbaum. 
\newblock Minimax quadratic estimation 
of a quadratic functional. 
\newblock {\em J. Complexity}, 1990, 6, 290--323.  
	
	
	
\bibitem{GillLevit}
R.D.~Gill and B.Y.~Levit. 
\newblock Applications of the van Trees inequality: a Bayesian Cram\'er-Rao bound.
\newblock {\em Bernoulli,} 1995, 1(1-2), 59--79. 
	

\bibitem{Girko} V.L.~Girko. 
\newblock Introduction to general statistical analysis. 
\newblock{\em Theory Probab. Appl.}, 1987, 32, 2: 229--242.

\bibitem{Girko-1} V.L.~Girko. 
\newblock {\em Statistical analysis of observations of increasing dimension}.  
\newblock Springer, 1995. 

\bibitem{Han} Y.~Han, J.~Jiao and R. Mukherjee. On estimation of $L_r$-norms in Gaussian white noise 
model. 2017, {\em arXiv:1710.03863}.

\bibitem{Hardy} M.~Hardy. Combinatorics of partial derivatives. 
\newblock{\em Electronic J. Combinatorics}, 2006, 13, 2. 

\bibitem{Ibragimov} I. A.~Ibragimov and R.Z.~Khasminskii.
\newblock{\em Statistical Estimation: Asymptotic Theory}. 
\newblock Springer-Verlag, New York, 1981. 

\bibitem{Ibragimov_Khasm_Nemirov} I.A.~Ibragimov, A.S.~Nemirovski and R.Z.~Khasminskii. 
\newblock Some problems of nonparametric estimation in Gaussian white noise. 
\newblock {\em Theory of Probab. and Appl.}, 1987, 31, 391--406.  


%\bibitem{Montanari} A.~Javanmard and A.~Montanari. Hypothesis testing in high-dimensional
%regression under the Gaussian random design model: Asymptotic theory. {\em IEEE Transactions on Information %Theory}, 2014, 60, 10, 6522--6554.




\bibitem{Jiao} J.~Jiao,  Y.~Han  and T.~Weissman.  
Bias correction with Jackknife, Bootstrap and Taylor Series., 2017, {\it arXiv:1709.06183}.

\bibitem{Jost} J.~Jost, M. Kell and C. Rodrigues. Representations of Markov chains by random 
maps: existence and regularity conditions. 
\newblock {\em Calculus of Variation and Partial Differential Equations}, 
2015, 54, 3, 2637--2655.


\bibitem{Klemela} J.~Klemel\"a. 
\newblock Sharp adaptive estimation of quadratic functionals. 
\newblock {\em Probability Theory and Related Fields},
\newblock 2006, 134, 539--564.




\bibitem{Koltchinskii_2017} V.~Koltchinskii. 
\newblock Asymptotically Efficient Estimation of Smooth Functionals of Covariance Operators. 
\newblock{\em J. European Mathematical Society}, 2019, to appear. {\it arxiv:1710.09072}

\bibitem{Koltchinskii_2018} V.~Koltchinskii. 
\newblock Asymptotic Efficiency in High-Dimensional Covariance Estimation.
\newblock {\em Proc. ICM 2018}, Rio de Janeiro, 2018, vol. 3, 2891--2912.

\bibitem{Koltchinskii_Lounici_bilinear}
V.~Koltchinskii and K.~Lounici.
\newblock  Asymptotics and concentration bounds for bilinear forms of spectral
 projectors of sample covariance. {\em Ann. Inst. H. Poincar\'e Probab. Statist.},
 2016, 52, 4, 1976--2013. 

\bibitem{Koltchinskii_Lounici_arxiv}
V.~Koltchinskii and K.~Lounici.
\newblock Concentration inequalities and moment bounds for sample covariance
operators.
\newblock {\em Bernoulli}, 2017, 23, 1, 110--133.

\bibitem{Koltchinskii_Lounici_AOS}
V.~Koltchinskii and K.~Lounici.
\newblock Normal approximation and concentration of spectral projectors of sample covariance.
\newblock{\em Annals of Statistics}, 2017, 45, 1, 121--157. 

\bibitem{Koltchinskii_Nickl} V.~Koltchinskii, M.~L\"offler and R.~Nickl.
\newblock Efficient Estimation of Linear Functionals of Principal Components.
\newblock{\em Annals of Statistics}, 2019, to appear. {\it arXiv:1708.07642}.

\bibitem{Koltchinskii_Xia} V.~Koltchinskii and D.~Xia. 
\newblock Perturbation of linear forms of singular vectors
under Gaussian noise. 
\newblock In: {\em High Dimensional Probability VII: The Carg\`ese Volume},
Progress in Probability, vol 71, Birkh\"auser, 
pp. 397--423.   

\bibitem{Koltchinskii_Zhilova} V.~Koltchinskii and M.~Zhilova. Efficient estimation of 
smooth functionals in Gaussian shift models. 2018, {\it arXiv:1810.02767}.



%\bibitem{Kwapien} S.~Kwapien and B.~Szymanski.
%\newblock Some remarks on Gaussian measures on Banach spaces.
%\newblock {\em Probab. Math. Statist.}, 1980, 1, 1, 59--65. 


\bibitem{Laurent} B.~Laurent.
\newblock Efficient estimation of integral functionals of a density. 
\newblock {\em Annals of Statistics}, 1996, 24, 659--681. 


%\bibitem{Ledoux}
%M.~Ledoux.
%\newblock{\em The Concentration of Measure Phenomenon}.
%\newblock American Mathematical Society. 2001.

\bibitem{Levit_1} B.~Levit. 
\newblock
On the efficiency of a class of non-parametric estimates.
\newblock {\em Theory of Prob. and applications}, 1975, 20(4), 723--740.

\bibitem{Levit_2} B.~Levit. 
\newblock 
Asymptotically efficient estimation of nonlinear functionals.
\newblock {\em Probl. Peredachi Inf. (Problems of Information Transmission)}, 1978, 14(3), 65--72. 

\bibitem{Lepski} O.~Lepski, A. Nemirovski and V. Spokoiny.
\newblock On estimation of the $L_r$ norm of a regression function. 
\newblock {\em Probab. Theory Relat. Fields}, 1999, 113, 221--253. 


\bibitem{Mukherjee} R.~Mukherjee, W.~Newey and J.~Robins.
\newblock Semiparametric Efficient Empirical Higher Order Influence Function Estimators.
\newblock 2017, {\em arXiv:1705.07577.}

\bibitem{Nemirovski_1990} A.~Nemirovski.
\newblock On necessary conditions for the efficient estimation of functionals of a nonparametric signal which is observed in white noise.
\newblock {\em Theory of Probab. and Appl.}, 1990, 35, 94--103.



\bibitem{Nemirovski} A.~Nemirovski.
\newblock{\em Topics in Non-parametric Statistics}.
\newblock{Ecole d'Ete de Probabilit\'es de Saint-Flour}.
\newblock
Lecture Notes in Mathematics, v. 1738, Springer, New York, 2000.


%\bibitem{Peller_87} V.V.~Peller. 
%\newblock Hankel operators in the perturbation theory of unitary and self-adjoint 
%operators. 
%\newblock {\em Funk. anal. i ego pril.}, 1985, 19(2), 37--51 (In Russian),
%English transl.: {\em Func. Anal. Appl.}, 1985, 19(2), 111--123.    

\bibitem{Pisier} G.~Pisier. Probabilistic methods in geometry of Banach 
spaces. \newblock{\em Lecture Notes in Mathematics}, Springer, 1986,
pp. 105--136. 


%\bibitem{RamsaySilverman} J.O.~Ramsay and B.W.~Silverman. 
%\newblock {\em Functional Data Analysis}. 
%\newblock  {Springer Series in Statistics.} Springer, 2005.


\bibitem{Robins} J.~Robins, L.~Li, E. Tchetgen and A. van der Vaart.
\newblock Higher order influence functions and minimax estimation of nonlinear 
functionals. 
\newblock {\em IMS Collections Probability and Statistics: Essays in Honor of David. A. Freedman}, 
2008, vol. 2, 335-421.

\bibitem{Robins_1} J.~Robins, L.~Li, E. Tchetgen and A. van der Vaart.
\newblock Asymptotic Normality of Quadratic Estimators.
\newblock {\em Stochastic Processes and Their Applications.}
2016, 126(12), 3733--3759.


\bibitem{Triebel} H.~Triebel. 
\newblock {\em Theory of function spaces}, 
\newblock {Birkh\"auser, 1983.}

%\bibitem{Tsybakov} A.B.~Tsybakov. 
%\newblock {\em Introduction to Nonparametric Estimation}.
%\newblock Springer, 2009. 

%\bibitem{vdgBuhlmannRitovDezeureAOS} S.~van de Geer, P.~B\"uhlmann, Y.~Ritov and R.~Dezeure.  On asymptotically optimal confidence regions and tests for high-dimensional models. 
%{\em Annals of Statistics}, 2014, 42(3), 1166--1202. 

\bibitem{van der Vaart} A.~van der Vaart. Higher order tangent spaces and influence functions.
{\em Statistical Science}, 2014, 29, 4, 679--686.

\bibitem{Villani} C.~Villani. \newblock{\em Optimal Transport. Old and New.}
\newblock{Springer, 2009.} 


%\bibitem{Zhang_Zhang} C.-H.~Zhang and S.S.~Zhang. Confidence intervals for low dimensional parameters in high dimensional linear models. {\em J. R. Stat. Soc. Ser. B Stat. Methodol.}, 2014, 76 217--242. 

	
	
\end{thebibliography}
\end{document}